\documentclass[11pt, reqno]{amsart}
\usepackage{geometry}                
\usepackage{graphicx}
\usepackage{amssymb}
\usepackage{epstopdf}
\usepackage{enumitem}
\usepackage{dsfont}
\usepackage{tabu}
\usepackage{bold-extra}
\usepackage{cite}
\usepackage{appendix}

\newtheorem{theorem}{Theorem}[section]
\newtheorem{corollary}[theorem]{Corollary}
\newtheorem{lemma}[theorem]{Lemma}
\newtheorem{prop}[theorem]{Proposition}

\newtheorem{conj}[theorem]{Conjecture}

\numberwithin{equation}{section}

\theoremstyle{definition}

\theoremstyle{definition}
\newtheorem{remark}[theorem]{Remark}

\theoremstyle{remark}

\theoremstyle{remark}

\theoremstyle{theorem}
\newtheorem*{problem}{Model Problem}

\theoremstyle{definition}
\newtheorem*{ack}{Acknowledgments}

\binoppenalty=\maxdimen
\relpenalty=\maxdimen

\title[A variation of the prime k-tuples conjecture and quantum limits]{A variation of the prime k-tuples conjecture with applications to quantum limits}
\author{Oliver McGrath}
\address{Mathematical Institute, University of Oxford, Oxford, OX2 6GG, UK}
\email{oliver.mcgrath@maths.ox.ac.uk}

\begin{document}
\maketitle

\begin{abstract}
Let $\mathcal{H}^{*}=\{h_1,h_2,\ldots\}$ be an ordered set of integers. We give sufficient conditions for the existence of increasing sequences of natural numbers $a_j$ and $n_k$ such that $n_k+h_{a_j}$ is a sum of two squares for every $k\geq 1$ and $1\leq j\leq k.$ Our method uses a novel modification of the Maynard-Tao sieve together with a second moment estimate. As a special case of our result, we deduce a conjecture due to D.~Jakobson which has several implications for quantum limits on flat tori.
\end{abstract}


\section{Introduction}

We say that a set $\mathcal{H}=\{h_1,\ldots,h_k\}$ of distinct integers is {\it admissible} if $\#\{\mathcal{H}\,\,(\text{mod}\,\,p)\}<p$ for every prime $p.$ An outstanding problem in analytic number theory is the prime $k$-tuples conjecture, which asserts the following.
\begin{conj}
Let $\mathcal{H}=\{h_1,\ldots,h_k\}$ be admissible. Then there exists infinitely many integers $n$ such that the translates $n+h_1,\ldots,n+h_k$ are prime.
\end{conj}
A proof of this conjecture is far out of reach of current techniques. However, we have been successful in establishing various weak versions of this result using sieve methods. For example, the Maynard-Tao sieve can be used to show that $\gg \log{k}$ of the translates are simultaneously prime infinitely often, when $k$ is sufficiently large (cf.~\cite{Maynard, Poly}). 

We extend the definition of admissibility to infinite ordered sets and say $\mathcal{H}^*=\{h_1,h_2,\ldots\}$ is admissible if the finite truncation $\{h_1,\ldots,h_k\}\subseteq \mathcal{H}^*$ is admissible for every $k\geq 1.$ In this paper we are interested in the following variation of this conjecture, for numbers representable as a sum of two squares.

\begin{conj}\label{conj:inf}
Let $\mathcal{H}^*=\{h_1,h_2,\ldots\}$ be admissible. Then there exists an increasing sequence of integers $n_k$ such that, for every $k\geq 1,$ the translates $n_k+h_1,\ldots,n_k+h_k$ are sums of two squares.
\end{conj}
 
We remark that if we replaced ``sums of two squares" with ``prime" here, then this would simply be a reformulation of Conjecture~1.1. (It is easy to show that any finite admissible set can be extended to an infinite admissible set.)

Our interest in this version of the conjecture stems from a problem which appears towards the end of D.~Jakobson's {\it``Quantum limits on flat tori"} paper~\cite{Jak}. In this paper Jakobson is concerned with characterising the possible quantum limits that can arise on the standard flat $d$-dimensional torus $\mathbb{T}^d=\mathbb{R}^d/\mathbb{Z}^d.$ A complete classification of such objects is established in two dimensions, with possible behaviours in higher dimensions described unconditionally for $d\geq 4,$ and conditionally for $d=3$ on a weak version of Conjecture~\ref{conj:inf} \hbox{(cf.~\cite[Conjecture 8.2]{Jak})}. 

In this paper we establish Jakobson's conjecture.

\begin{theorem}\label{theo:JAK}
\sloppy There exists increasing sequences of natural numbers $a_j$ and $M_k$ such that $M_k-a_j^2$ is a sum of two squares for $1\leq j \leq k.$  Moreover, the sequence $a_j$ is such that:
\begin{enumerate} 
	\item $r_{2}(a_j)<r_{2}(a_{j+1})$ for all $j\geq 1.$ 
	\item The even parts are uniformly bounded; that is to say, if we write $a_j = 2^{b_j}m_j$ where $(m_j,2)=1,$ then $b_j = O(1)$ uniformly for $j\geq 1$.
\end{enumerate}
\end{theorem}
Here, $r_2(n)$ denotes the number of representations of $n$ as a sum of two squares. We deduce Theorem~\ref{theo:JAK} from the following general result.
\begin{theorem}\label{theo:intro1.1}
\sloppy Let $\mathcal{H}^{*}=\{h_1,h_2,\ldots\}$ be admissible such that each $h_i$ is divisible by 4. Then there exists increasing sequences of natural numbers $a_j$ and $n_k$ such that $n_k+h_{a_j}$ is a sum of two squares for every $k\geq 1$ and $1\leq j\leq k.$
\end{theorem}
\sloppy For example, Theorem~\ref{theo:intro1.1} applied to the admissible set \hbox{$\mathcal{H}^{*}=\{h_1,h_2,\ldots\}$} with elements $h_i=-(2\cdot 5^{i})^2,$ yields a solution to Theorem~\ref{theo:JAK}.

As mentioned above, Theorem~\ref{theo:JAK} allows us conclude results about quantum limits on flat tori. Let $(\lambda_j)_{j\geq 1}$ be a sequence of eigenvalues of the Laplace operator $\nabla$ on $\mathbb{T}^d$ such that $\lambda_j\rightarrow \infty,$ and let $\varphi_j$ be corresponding eigenfunctions with $\left\lVert\varphi_j\right\rVert_2=1$. If the sequence of probability measures $\mathrm{d}\mu_j=|\varphi_j|^2\mathrm{d}x$ has a weak-$*$ limit $\mathrm{d}v$, then we call $\mathrm{d}v$ a {\it quantum limit}. (Here $\mathrm{d}x$ is the normalised Riemannian volume.) 

It can be shown that all limits of such sequences $\mathrm{d}\mu_j$ are absolutely continuous with respect to the Lebesgue measure on $\mathbb{T}^d$ (cf.~\cite[Theorem 1.3]{Jak}), and so one can consider the Fourier expansion
\begin{equation}\label{eq:fourier1}
\mathrm{d}v = \sum_{\tau \in \mathbb{Z}^d} c_{\tau} e^{2\pi i \langle \tau, x \rangle}\mathrm{d}x.
\end{equation}
Among other things, Jakobson shows that in two dimensions all quantum limits are necessarily trigonometric polynomials (cf. \cite[Theorem 1.2]{Jak}). The same result isn't true for $d\geq 4$, and conjecturally not true for $d=3$ either (cf. \cite[Conjecture 8.2]{Jak} and the following discussion). With Theorem~\ref{theo:JAK}, we can now complete this aspect of the classification of quantum limits on flat tori.

\begin{theorem}\label{theo:QLmain}
There exists quantum limits on $\mathbb{T}^3$ that are not trigonometric polynomials.
\end{theorem}

As further consequences to Theorem~\ref{theo:JAK} we are able to show the following results for quantum limits whose Fourier expansions are described as in~(\ref{eq:fourier1}).

\begin{theorem}\label{theo:QL} Let $\epsilon>0.$ We have the following.
\begin{enumerate}[label=(\roman*)]
	\item For $d \geq 4$ there exists quantum limits $\mathrm{d}v$ on $\mathbb{T}^d$ with densities that are not in $l^{2-\epsilon}$ (i.e. for which $\sum_{\tau}|c_{\tau}|^{2-\epsilon}$ diverges).
	\item For $d\geq 5$ there exists quantum limits $\mathrm{d}v$ on $\mathbb{T}^d$ for which
\begin{equation*}
\limsup_{\rho\rightarrow \infty} \frac{\Sigma(\rho)}{\rho^{d-4-\epsilon}}=+\infty,
\end{equation*}
where $\Sigma(\rho)$ is defined as
\begin{equation}\label{eq:Sigma}
\Sigma(\rho) = \sum_{\substack{\tau \in \mathbb{Z}^d \\ |\tau| < \rho}}|c_{\tau}|.
\end{equation}
\end{enumerate}
\end{theorem}

The results contained in Theorem~\ref{theo:QL} improve upon various results found in~\cite{Jak}. Part~(i) was previously shown for $d\geq 5$, and has now been extended to the case $d=4$ where it is now optimal (cf. \cite[Theorem 1.4]{Jak}). Part (ii) improves on the weaker lower bound 
$$\limsup_{\rho\rightarrow \infty} \frac{\Sigma(\rho)}{\rho^{d-5-\epsilon}}=+\infty$$
which was shown for $d\geq 6$. The lower bound we prove is believed to be optimal for all $d\geq 5$ (cf.~\cite[Proposition 1.2]{Jak} and comments shortly after).

\begin{remark}
It is well-known that the eigenvalues of $\nabla$ on $\mathbb{T}^d$ are the numbers $4\pi^2 k$ for non-negative integers $k$, and they occur with multiplicity $r_d(k)$ (the number of representations of $k$ as the sum of $d$ squares). This means various constructions associated to quantum limits on flat tori can be translated to problems in number theory involving sums of squares.
\end{remark}

\begin{remark}
\sloppy Jakobson shows how Theorem~\ref{theo:JAK} follows from weak form of the prime \hbox{$k$-tuples} conjecture, essentially by using the fact primes $p\equiv 1\,\,(\text{mod}\,\,4)$ are the sum of two squares (cf. discussion at the end of~\cite[Section 8]{Jak}). We note that the weak form of the conjecture Jakobson uses is still far out of reach of current methods.
\end{remark}


\section{Outline of new sieve ideas}

In this section, let $\mathcal{A}\subseteq \mathbb{N}$ denote a set of arithmetic interest, which for our purposes is the set of numbers representable as a sum of two squares (but the following discussion holds more generally). We will denote random variables by boldfaced letters, for example $\mathbf{X}$. We will let $\mathbb{P}(\cdot)$ denote a probability measure and by $\mathbb{E}[\cdot]$ the expectation operator. 

\subsection{A model problem} Our aim is to prove Theorem~\ref{theo:intro1.1}. By a pigeonhole argument (see~Proposition~\ref{prop:pigeonhole}), it suffices to consider the following model problem.

\begin{problem}
Fix an admissible set $\mathcal{H}^{*}=\{h_1,h_2,\ldots\}$ of integers and a partition $\mathcal{H}^{*}=B_1\cup B_2\cup \ldots$ where each bin $B_i$ is a fixed, finite size $k_i$. Is it the case that for every $M\geq 1$ there exists elements $h_{a_1},\ldots,h_{a_M}$ and infinitely many integers $n$ such that $h_{a_j}\in B_j$ and $n+h_{a_j}\in\mathcal{A}$ for $1\leq j\leq M$?
\end{problem}

We realise the above set-up as the output of a sieving process. For notational purposes we order $B_{i}=\{h_{k_0+\ldots+k_{i-1}+1},\ldots,h_{k_0+\ldots+k_{i}}\}$ for $i\geq 1,$ with the convention that $k_0=0.$ Let $k=k_{0}+\ldots+k_M$ for some large $M$. Given $n\in[N,2N)$ for some large $N$, let $\mathbf{X}_i$ denote the random variable that counts the number of $h\in B_i$ such that $n+h\in\mathcal{A}$, and let $\mathbf{X}=\mathbf{X}_1+\ldots\mathbf{X}_{M}.$


%


The current method we use to detect primes in $k$-tuples is the GPY method. For general sets $\mathcal{A},$ the aim is to show the first moment inequality
\begin{equation}\label{eq:GPY}
S_{\mathcal{A}}=\sum_{N\leq n< 2N} \bigg(\sum_{i=1}^{k}\mathds{1}_{\mathcal{A}}(n+h_i)-m\bigg)w(n)>0
\end{equation}
\sloppy holds for some integer $m\geq 1$, where $\mathds{1}_{\mathcal{A}}$ denotes the indicator function of the set $\mathcal{A}$ and $w(n)$ are non-negative weights (cf.~\cite{GPY, Maynard, Poly}). If we normalise the weights to sum to 1, then this is saying ``if we choose $n$ randomly from the interval $[N,2N)$ with probability $w(n)$, then \hbox{$\mathbb{E}[\mathbf{X}]>m$}." From this we can deduce the existence of an $n\in[N,2N)$ for which $m+1$ of the translates $n+h_i\in\mathcal{A}$. We say such a translate has been ``accepted." Exactly which translates are accepted is unknown. This is a limitation of the first moment method.

It is clear that for our model problem, we require more information about which translates $n+h$ appear. Namely, we need to be obtaining an accepted translate from each of the $M$ bins $B_1,\ldots,B_M$ (recall $k=k_0+\ldots+k_M$). This presents two obvious difficulties.

 \begin{enumerate}
 	\item For any $1\leq i\leq k$ the probability of the event $n+h_i\in \mathcal{A}$ depends on $k$, and tends to 0 as $k\rightarrow \infty$. This would mean any bin of {\it fixed} size expects to get fewer and fewer elements as $k$ gets large. In particular, we cannot hope the hypotheses hold for every $M\geq 1.$
	\item Even in the situation where $\mathbb{E}[\mathbf{X}_i]>1$ holds for each $1\leq i\leq M$, we cannot conclude anything about $\mathbb{P}((\mathbf{X}_1>0)\cap\ldots \cap(\mathbf{X}_M>0))$ unless we input some information about the joint distribution of the bins.
 \end{enumerate}
 
 We are able to overcome these issues by modifying the sieve weights and using a second moment estimate.

\subsection{Choice of sieve weights} We solve the first problem by modifying the sieve weights to put more emphasis on the earlier bins. This way, we can guarantee that $\mathbb{P}(n+h\in\mathcal{A}|h\in B_i)=c_i$ where the constant $c_i$ depends solely on the bin. This also means that we can guarantee $\mathbb{E}[\mathbf{X}_i]$ is large for each $i$ (provided $k_i$ is large enough in terms of $c_i$). We will consider Maynard-Tao sieve weights with a fixed factorisation 
\begin{equation}
w(n) = \bigg(\sum_{\substack{d_1,\ldots,d_k \\ d_i|n+h_i}} \prod_{i=1}^{M}\lambda_{d_{k_0+\ldots+k_{i-1}+1},\ldots,d_{k_0+\ldots+k_{i}}}^{(i)}\bigg)^2,
\end{equation}
where
\begin{equation}
\lambda_{d_{k_0+\ldots+k_{i-1}+1},\ldots,d_{k_0+\ldots+k_i}}^{(i)} \approx \bigg(\prod_{j=k_0+\ldots+k_{i-1}+1}^{k_0+\ldots+k_i}\mu(d_{j})\bigg)f_i(d_{k_0+\ldots+k_{i-1}+1},\ldots,d_{k_0+\ldots+k_i}),
\end{equation}
and $f_i$ is a suitable smooth function supported on the simplex 
\begin{equation}
R_{B_i,\beta_i}=\{(x_{k_0+\ldots+k_{i-1}+1},\ldots,x_{k_0+\ldots+k_i})\in [0,1]^{k_i}: 0\leq \sum_{j=k_0+\ldots+k_{i-1}+1}^{k_0+\ldots+k_i}x_j \leq \beta_i\}.
\end{equation} 
Here $(\beta_i)_{i\geq 1}$ is a sequence of real numbers such that $\sum_{i=1}^{\infty}\beta_i \leq 1$ (cf. the sieve weights defined in~\hbox{\cite[Proposition 4.1]{Maynard}}). We will take $\beta_i=2^{-i},$ and in this instance one might say ``we have allocated 50\% of the sieve power to $B_1$."

\subsection{Concentration of measure} We can deal with the second problem by showing the random variables $\mathbf{X}_i$ exhibit ``enough" independence. This is precisely what concentration of measure arguments are used for. For example, an application of the union bound and Chebychev's inequality tells us that
\begin{equation}
\mathbb{P}(|\mathbf{X}_i-\mathbb{E}[\mathbf{X}_i]| < t_i \text{ for all } i) \geq 1 - \sum_{i=1}^{M} \frac{\mathbb{E}[\mathbf{X}_i-\mathbb{E}[\mathbf{X}_i]]^2}{t_i^2}
\end{equation}
where $t_i \geq 1$ are concentration parameters. Thus, if we can show the variances $\mathbb{E}[\mathbf{X}_i-\mathbb{E}[\mathbf{X}_i]]^2$ are small, then we should be able to show each random variable concentrates in a (small) interval about its mean with high probability. In particular, we should be able to show that we get (at least) one accepted translate coming from each bin after the sieving process. We  implement this analytically by using a second moment estimate (see~Proposition~\ref{prop:2ndmoment}).

\begin{remark}\label{rmk:10}
A similar second moment estimate was considered by Banks, Freiberg and Maynard in their paper~\cite{BFM}. They showed a partition result (for primes), where the bins are allowed to grow with $k$. A key aspect of our work is that the bin sizes are fixed. Moreover they only needed upper bounds of the correct order of magnitude for the sieve sums, whereas we require precise asymptotics.
\end{remark}

\subsection{Hooley's $\rho$ function.} In practice, utilising a second moment estimate requires an understanding of the two-point correlations
\begin{equation}\label{eq:2pt}
\sum_{N\leq n<2N} \rho_{\mathcal{A}}(n+h) \rho_{\mathcal{A}}(n+h')
\end{equation}
where $\rho_{\mathcal{A}}$ is a non-negative function supported on $\mathcal{A}.$ This means our methods are limited to cases in which estimates of the above type are known. In particular we cannot deal with the case of primes, as evaluating the above sum asymptotically with $\mathds{1}_{\mathbb{P}}$ or the \hbox{von Mangoldt} function $\Lambda$ (say) is equivalent to the twin prime conjecture.\footnote{This is precisely why the authors were limited to using upper bounds and not asymptotics in \cite{BFM} (cf.~Remark 2.1).}

One can do much better when working with sums of two squares. We cannot evaluate~(\ref{eq:2pt}) asymptotically using the indicator function, but we can if we work with the representation function $r_2(n)$ instead. Unfortunately $r_2(n)$ is too large for our purposes, and it proves necessary to consider a weighted version instead. In Hooley's work \cite{Hooley} on the distribution of numbers representable as the sum of two squares, he considers a weighted representation function $\rho(n)=t(n)r_2(n)$ where 
\begin{equation}\label{eq:HooleysRho}
t(n) = t_{N,\theta_1}(n) = \sum_{\substack{a|n,\,\,a\leq v \\ p|a\Rightarrow p\equiv 1\,\,(\text{mod}\,\,4)}}\frac{\mu(a)}{g_2(a)}\bigg(1-\frac{\log{a}}{\log{v}}\bigg),\,\,\,\,\,(v=N^{\theta_1})
\end{equation}
and $\theta_1$ is a suitably small, fixed constant (for example Hooley takes $\theta_1=1/20$). Here $g_2(p)$ is the multiplicative function defined on primes by
\begin{equation}\label{eq:psi}
g_2(p) =
\begin{cases}
2-\frac{1}{p}\,\,&\text{if $p\equiv 1\,\,(\text{mod}\,\,4),$}\\
\frac{1}{p}\,\,&\text{if $p\equiv 3\,\,(\text{mod}\,\,4).$}
\end{cases} 
\end{equation}
The $t(n)$ factor acts to dampen down the oscillations due to $r_2(n).$ Thus $\rho(n)$ acts a proxy for the indicator function $\mathds{1}_{n=\Box+\Box}$ and moreover asymptotics for (\ref{eq:2pt}) are available for $\rho(n).$ This is the function we will be working with. 

\subsection{Outline of the paper.} In Section~\ref{section:QLproofs} we deduce the results about quantum limits contained in Theorem's~\ref{theo:QLmain} and \ref{theo:QL} from Theorem~\ref{theo:JAK}. In Section~\ref{section:APs} we state a few preliminary lemmas that will be needed in the sieve calculations. We defer the proofs of these results to the Appendix. In Section~\ref{section:sieveresults} we state our main sieve results, and from them we deduce Theorem~\ref{theo:intro1.1}. We isolate a key lemma (see Lemma~\ref{lemma:generalsieve}) from which all of our sieve estimates follow. Section~\ref{section:sievetechnical} and Section~\ref{section:sieveproofs} are dedicated to proving this lemma. 

\begin{ack}
The author would like to thank James Maynard for introducing this problem to them in the first instance, and for being available for many helpful conversations thereafter. We would also like to thank the anonymous referee for a careful and thorough reading of the paper. The author is funded by an EPSRC Studentship.
\end{ack}


\section{Proofs of quantum limit results}\label{section:QLproofs}
In this section we deduce the results of Theorem~\ref{theo:QLmain} and Theorem~\ref{theo:QL} from Theorem~\ref{theo:JAK}. Following \cite{Jak}, we note that $\varphi_k$ is an eigenfunction of the Laplacian on $\mathbb{T}^d$ with eigenvalue $\lambda_k = 4\pi^2 n_k$ for some $n_k\in\mathbb{N}$ if and only if its Fourier expansion is of the form
\begin{equation}\label{eq:fourier2}
\varphi_k(x) = \sum_{\substack{\xi\in\mathbb{Z}^d \\ |\xi|^2=n_k}}a_{\xi}e^{2\pi i \langle \xi,x \rangle},
\end{equation}
for $a_{\xi}\in\mathbb{C}.$ Moreover $\left\lVert \varphi_k \right\rVert_2=1$ if and only if $\sum_{\xi}|a_{\xi}|^2=1$. It follows that 
\begin{align}\nonumber
|\varphi_k(x)|^2 &= \sum_{\substack{\tau\in\mathbb{Z}^d }}b_{\tau}(k)e^{2\pi i \langle \tau,x \rangle},\\ \label{eq:fourier3}
b_{\tau}(k) &= \sum_{\substack{\xi-\eta=\tau \\ |\xi|^2=|\eta|^2=n_k}}a_{\xi}\overline{a_{\eta}}.
\end{align}
Let $\mathrm{d}v$ be a quantum limit on $\mathbb{T}^d$ with Fourier expansion as in (\ref{eq:fourier1}). By $|\varphi_k|^2\mathrm{d}x \rightarrow \mathrm{d}v$ weak-$*$ as $k\rightarrow \infty$ we mean that for every $\tau\in \mathbb{Z}^d$ we have $
c_{\tau} = \lim_{k\rightarrow \infty} b_{\tau}(k).$ 

Fix $a_1<a_2<\ldots$ and $M_1<M_2<\ldots$ as in the statement of Theorem \ref{theo:JAK}, and let $b_j^{(k)},c_{j}^{(k)}\in\mathbb{Z}$ be such 
\begin{equation}
M_k = a_j^2 + (b_{j}^{(k)})^2 + (c_{j}^{(k)})^2,\,\,\,\,\,\,(1\leq j\leq k).
\end{equation}
Let $0<\epsilon<2$ and let $F=F_{\epsilon}:\mathbb{N}\rightarrow \mathbb{N}$ be a rapidly increasing function whose rate of growth will be specified later. As we are assuming both $a_i\rightarrow \infty$ and $r(a_i)\rightarrow \infty$, by passing to a subsequence if necessary (and relabelling the indices of the sequence $M_k$), we may suppose $a_i,r(a_i) \gg F(i)$. 

We will require information about the number of integer points on the surface of the $d$-dimensional sphere. For this we recall the following results: writing $n=2^km$ and letting $\sigma(n)=\sum_{d|n}d$ denote the sum-of-divisors function, we have the identities
\begin{align*}
r_3(n^2) &= 6 \prod_{p^a||m}(\sigma(p^{a})-(-1)^{\frac{p-1}{2}}\sigma(p^{a-1})), \\
r_4(n^2) &= 24 \sigma(m^2), \\
r_d(n^2) &= C_d(n^2) n^{d-2}\,\,\,\,\text{ for }\,\,\,\, d\geq 5.
\end{align*}
Here $C_d(n^2)$ is a singular series which satisfies $C_d(n^2) \asymp_d 1.$ 

We prove each statement similarly - in each case we consider a suitable sequence of $L^2$-normalised eigenfunctions with eigenvalues $\lambda_k = 4\pi^2 M_k$ and show that the limit has the desired property. 
\\ 
\begin{proof}[Proof of (Theorem~\ref{theo:JAK} $\Rightarrow$ Theorem~\ref{theo:QLmain})]

Consider the sequence of $L^{2}$-normalised eigenfunctions on $\mathbb{T}^3$ that arise by choosing coefficients
$$
a_{\xi} =
\begin{cases}
\sqrt{\frac{2^{k}}{2^k-1}}\cdot \frac{1}{2^{(j+1)/2}}\,\,\,&\text{if $\xi=(\pm a_j, b_j^{(k)},c_j^{(k)})$ for some $j,$} \\
0\,\,\,&\text{otherwise.}
\end{cases}
$$
Fix $i\geq 1.$ With this choice, for any $k\geq i$ we obtain 
\begin{equation*}
b_{(2a_i,0,0)}(k) = \sum_{\substack{\xi-\eta = (2a_i,0,0) \\ |\xi|^2=|\eta|^2=M_k}}a_{\xi}\overline{a_{\eta}} = \frac{2^k}{2^k-1}\cdot \frac{1}{2^{i+1}},
\end{equation*}
because the $a_i$ are distinct and so the only contribution to the sum comes from $\xi=(a_i,b_i^{(k)},c_i^{(k)})$ and $\eta=(-a_i,b_i^{(k)},c_i^{(k)})$. Hence 
\begin{equation*}
c_{(2a_i,0,0)} = \lim_{k\rightarrow \infty} b_{(2a_i,0,0)}(k) = \frac{1}{2^{i+1}} > 0,
\end{equation*}
which proves the theorem.
\end{proof}

\begin{proof}[Proof of (Theorem~\ref{theo:JAK} $\Rightarrow$ Theorem~\ref{theo:QL})]
It suffices to prove part (i) for $d=4,$ by identifying the eigenfunctions on $\mathbb{T}^d$ with the eigenfunctions on $\mathbb{T}^{d+l}$ all of whose non-zero frequencies lie in the subspace $\{(x_1,\ldots,x_{d+l}):x_{d+1}=\ldots=x_{d+l}=0\}\subseteq \mathbb{Z}^{d+l}$. 

Consider the sequence of $L^{2}$-normalised eigenvectors on $\mathbb{T}^4$ that arise by choosing
$$
a_{\xi} =
\begin{cases}
\sqrt{\frac{2^{k}}{2^k-1}}\cdot \frac{1}{(2^jr(a_j^2))^{1/2}}\,\,\,&\text{if $\xi=(X,Y,b_j^{(k)},c_j^{(k)})$ for some $j$ and $X^2+Y^2=a_j^2$,} \\
0\,\,\,&\text{otherwise.}
\end{cases} 
$$
Fix $i$ and suppose $k\geq i$. Given two non-zero coefficients $a_{\xi},a_{\xi'}$, corresponding to vectors of the form $\xi=(X,Y,b_i^{(k)},c_i^{(k)})$ and $\xi'=(X',Y',b_i^{(k)},c_i^{(k)}),$ we see the difference vector is 
\begin{equation*}
\xi-\xi' = (X-X',Y-Y',0,0),
\end{equation*}
and the norm of this vector is $\leq 2a_i$ by the triangle inequality. From (\ref{eq:fourier3}), it follows that if we sum $b_{\tau}(k)$ over all $|\tau|\leq 2a_i$ then we pick up all such differences. There are $r(a_i^2)^2$ of them, leading to 
\begin{equation*}
\sum_{\substack{\tau \in \mathbb{Z}^d \\ |\tau| \leq 2a_i}}|b_{\tau}(k)|^{2-\epsilon} \geq \bigg(\frac{2^k}{2^k-1}\bigg)^{2-\epsilon} \cdot \frac{(2^{i}r(a_i^2))^{\epsilon}}{4^{i}}.
\end{equation*}
Taking the limit as $k\rightarrow \infty$ we conclude that 
\begin{equation*}
\sum_{\substack{\tau \in \mathbb{Z}^d \\ |\tau| \leq 2a_i}}|c_{\tau}|^{2-\epsilon} \geq \frac{(2^ir(a_i^2))^{\epsilon}}{4^i}  \geq \frac{(2^ir(a_i))^{\epsilon}}{4^i}   \gg \frac{(2^iF(i))^{\epsilon}}{4^i}.
\end{equation*}
Now we can choose $F=F_{\epsilon}$ so that the expression on the right hand side is unbounded as $i\rightarrow \infty.$ It follows that $\sum_{\tau}|c_{\tau}|^{2-\epsilon}$ doesn't converge, proving part~(i).

For part~(ii), fix $d\geq 5.$ We proceed as in part (i), except this time because $d\geq 5$ we have the lower bound $r_{d-2}(a_i^2) \gg_{d} a_i^{d-4}$. We remark that to obtain this bound for $d\in \{5,6\}$ we are using property (2) given by Theorem~\ref{theo:JAK}. For $d\geq 7$ the bound holds without this extra assumption on our sequence. 

Now consider the sequence of eigenvectors on $\mathbb{T}^d$ with densities
$$
a_{\xi} =
\begin{cases}
\sqrt{\frac{2^{k}}{2^k-1}}\cdot \frac{1}{(2^jr_{d-2}(a_j^2))^{1/2}}\,\,\,&\parbox[t]{.6\textwidth}{if $\xi=(X_1,\ldots,X_{d-2},b_j^{(k)},c_j^{(k)})$ for some $j$ 
and $X_1^2+\ldots X_{d-2}^2=a_j^2$,} \\
0\,\,\,&\text{otherwise.}
\end{cases} 
$$
Fix $i$ and suppose $k\geq i$. As above we can conclude
\begin{equation*}
\sum_{\substack{\tau \in \mathbb{Z}^d \\ |\tau| \leq 2a_i}}|b_{\tau}(k)| \geq \bigg(\frac{2^k}{2^k-1}\bigg) \cdot \frac{r_{d-2}(a_i^2)}{2^{i}}.
\end{equation*}
Taking the limit as $k\rightarrow \infty$ we conclude that 
\begin{equation*}
\sum_{\substack{\tau \in \mathbb{Z}^d \\ |\tau| \leq 2a_i}}|c_{\tau}| \geq \frac{r_{d-2}(a_i^2)}{2^i} \gg_{d} \frac{a_{i}^{d-4}}{2^i}.
\end{equation*}
It follows that
\begin{equation*}
\frac{\sum_{\substack{|\tau| \leq 2a_i}}|c_{\tau}|}{(2a_i)^{d-4-\epsilon}} \gg_{d} \frac{(2a_{i})^{\epsilon}}{2^i} \gg_d \frac{(2F(i))^{\epsilon}}{2^i}.
\end{equation*}
Choosing $F=F_{\epsilon}$ appropriately and letting $i\rightarrow \infty,$ we see that for this choice of quantum limit we have
\begin{equation*}
\limsup_{\rho\rightarrow \infty} \frac{\Sigma(\rho)}{\rho^{d-4-\epsilon}}=+\infty,
\end{equation*}
where $\Sigma(\rho)$ is defined as in (\ref{eq:Sigma}). This proves part (ii).
\end{proof}

\section{Notation} We will use both Landau and Vinogradov asymptotic notation throughout the paper. $N$ will denote a large integer, and all asymptotic notation is to be understood as referring to the limit as $N\rightarrow \infty.$ Any dependencies of the implied constants on other parameters $A$ will be denoted by a subscript, for example $X\ll_{A} Y$ or $X=O_{A}(Y),$ unless stated otherwise. We let $\epsilon$ denote a small positive constant, and we adopt the convention it is allowed to change at each occurrence, and even within a line.

We will denote the non-trivial Dirichlet character $(\text{mod}\,\,4)$ by $\chi_4,$ and we may omit the subscript and simply write $\chi$. As usual, we let $\varphi(n)$ denote the Euler-Totient function, $\tau_r(n)$ denote the number of ways of writing $n$ as the product of $r$ natural numbers, $\mu(n)$ denote the M\"obius function, and $r_d(n)$ denote the number of representations of $n$ as the sum of $d$ squares. For the rest of the paper we will write $r(n)$ when $d=2.$ For integers $a,b$ we let $(a,b)$ denote their highest common factor, and $[a,b]$ denote their lowest common multiple.

We define the Ramanujan-Landau constant
\begin{equation}\label{eq:A}
A= \frac{1}{\sqrt{2}}\prod_{p\equiv 3\,\,(\text{mod}\,\,4)}\bigg(1-\frac{1}{p^2}\bigg)^{-\frac{1}{2}}
\end{equation}
which will appear in many of our results.


\section{Preliminaries}\label{section:APs}

In this section, we formalise some of the notions discussed in Section~2, and state a few key estimates that will be required later in the sieve calculations.

\subsection{A pigeonhole argument} The following proposition allows us to go from the set-up in Theorem~\ref{theo:intro1.1} to the model problem discussed in Section~2.

\begin{prop}[Pigeonhole argument for infinite bin set-up]\label{prop:pigeonhole}
\sloppy Fix $\mathcal{A}\subseteq \mathbb{N}$ and a set $\mathcal{H}^{*}=\{h_1,h_2,\ldots\}$ of integers. Suppose that there exists a partition $\mathcal{H}^{*}=B_1\cup B_2\cup\ldots$ where each bin $B_i$ is a fixed, finite size, such that for every $M\geq1,$ there exists infinitely many $n$ and $M$ translates $n+h_{i,M}\in \mathcal{A}$ with $h_{i,M}\in B_i$ for $1\leq i \leq M.$ Then there exists increasing sequences $a_j$ and $n_k$ such that for every $k\geq 1$ we have $n_k+h_{a_j} \in \mathcal{A}$ for $1\leq j\leq k$ and moreover $h_{a_j}\in B_j$ for all $j$. 
\end{prop}
\begin{proof}
With the above set-up, obtain translates $n+h_{i,M}\in \mathcal{A}$ with $h_{i,M}\in B_i$ for $1\leq i \leq M$ for each $M\geq 1.$  Record this process in the following infinite table:
\begin{center}
\begin{tabular}{l |c| c |c |c |c|c r}
$B_1$ & $B_2$ & $B_3$ & $\ldots$ & $B_M$ & $B_{M+1}$ & $\ldots$ \\
\hline
$h_{1,1}$ & $$ &$$ & $\ldots$ & $$ &  \\
$h_{1,2}$ & $h_{2,2}$ &$$ & $\ldots$ & $$ &$$    \\
$h_{1,3}$ & $h_{2,3}$ &$h_{3,3}$ & $\ldots$ & $$ &$$    \\
$\vdots$ & $\vdots$ &$\vdots$ &$\vdots$ &$$ &$$  \\
$h_{1,M}$ & $h_{2,M}$ &$h_{3,M}$ &$\ldots$ &$h_{M,M}$ &$$   \\
$h_{1,M+1}$ & $h_{2,M+1}$ &$h_{3,M+1}$ &$\ldots$ &$h_{M,M+1}$ &$h_{M+1,M+1}$   \\
$\vdots$ & $\vdots$ &$\vdots$ &$\ldots$ &$\vdots$ &$\vdots$ &$\vdots$  
\end{tabular}
\end{center}
Look at the first column. By the pigeonhole principle, since $B_1$ is finite, there must exist an element $h_{a_1}\in B_1$ which appears infinitely many times. Choose the smallest such $h_{a_1},$ and choose any $n_1\in \mathbb{N}$ for which $n_1+h_{a_1}\in \mathcal{A}$. Now erase all the rows that do not start with $h_{a_1},$ and look at the remaining (infinite) table. Again, since $B_2$ is finite, some element $h_{a_2}\in B_2$ must occur infinitely many times in the second column. Choose the smallest such $h_{a_2},$ and choose any $n_2>n_1$ such that $n_2+h_{a_2}\in \mathcal{A}$ (which we can do because there are infinitely many such $n_2$). By construction this $n_2$ will be such that $n_2+h_{a_1}\in \mathcal{A}.$ Now erase all rows that don't start with $h_{a_1},h_{a_2},$ and repeat this process for $B_3,$ and so on. We will end up with increasing sequences $a_j$ and $n_k$ which by construction satisfy the required conditions.
\end{proof}

\subsection{A second moment estimate} As discussed in Section~2, our work will require input about the joint distribution of the bins, which we will achieve via concentration of measure arguments. The following second moment estimate will suffice for our purposes.

\begin{prop}[Second moment estimate]\label{prop:2ndmoment}
Fix $\mathcal{A}\subset \mathbb{N}$ and $\mathcal{H}=\{h_1,\ldots,h_k\}$ a set of integers. Suppose we have a partition $\mathcal{H}=B_1\cup\ldots\cup B_M.$ Let $\mu_i, t_i\geq 1$ be real numbers for $1\leq i\leq M.$ Let $\rho_{\mathcal{A}}$ be a non-negative function supported on $\mathcal{A}$ and $w(n)$ be non-negative weights for each integer $n$. If 
\begin{equation}\label{eq:2nd}
\sum_{\substack{N\leq n<2N}}\bigg[\min_{j=1,\ldots,M}\frac{\mu_j^2}{t_j^2}-\sum_{i=1}^{M}\bigg(\frac{\sum_{h\in B_i}\rho_{\mathcal{A}}(n+h)-\mu_i}{t_i}\bigg)^2\bigg]w(n)>0,
\end{equation}
then there exists an $n\in [N,2N)$ and elements $h_{a_i}\in B_i$ such that $n+h_{a_i}\in \mathcal{A}$ for $1\leq i\leq M.$ 
\end{prop}
\begin{proof}
By positivity we deduce the existence of an $n\in [N,2N)$ such that 
\begin{equation*}\label{eq:2ndmoment}
\sum_{i=1}^{M}\bigg(\frac{\sum_{h\in B_i}\rho_{\mathcal{A}}(n+h)-\mu_i}{t_i}\bigg)^2<\min_{j=1,\ldots,M}\frac{\mu_j^2}{t_j^2}.
\end{equation*}
If $n+h\notin \mathcal{A}$ for all $h\in B_i,$ then by assumption on the support of $\rho_{\mathcal{A}}$ the left hand side of the above expression is $\geq \mu_i^2/t_i^2,$ a contradiction. 
\end{proof}

Thus, if the second moment estimate (\ref{eq:2nd}) holds for all $M\geq 1$ and sufficiently large $N$, then we are in a situation where the hypotheses of Proposition~\ref{prop:pigeonhole} are satisfied.

\subsection{Estimates in arithmetic progressions} We require an understanding of how $\rho(n)$ and $\rho(n)\rho(n+h)$ behaves in arithmetic progressions for our sieve calculations. Essentially this reduces down to understanding the corresponding sums for $r(n)$ and $r(n)r(n+h),$ where the estimates we need are known with power-saving error terms. This means that the error terms in the sieve calculations can be bounded trivially (cf. with the case of primes~\cite{Maynard, GPY}, where we have to use equi-distribution results such as the Bombieri-Vinogradov theorem to bound the error terms that arise). 

We have the following lemmas. We note that the functions $g_1,\ldots,g_7$ defined in this section will be used frequently throughout the rest of the paper.

\begin{lemma}\label{lemma:1} Suppose $(a,q)=(d,q)=1$ where $d,q$ are square-free and odd, of size $\ll N^{O(1)}.$ Then we have
\begin{equation*}
\sum_{\substack{n\leq N\\ n\equiv a\,\,(\text{mod}\,\,q) \\ n\equiv 1\,\,(\text{mod}\,\,4) \\ d|n}}r(n) = \frac{g_1(q)g_2(d)}{2qd}\pi N+R_1(N;d,q),
\end{equation*}
where $g_2$ is defined as in (\ref{eq:psi}), $g_1$ is the multiplicative function defined on primes by $g_1(p)=1-\chi(p)/p,$ and 
\begin{equation*}
R_1(N;d,q)\ll_{\epsilon} ((qd)^{\frac{1}{2}}+N^{\frac{1}{3}})d^{\frac{1}{2}}N^{\epsilon}
\end{equation*}
\end{lemma}

\begin{lemma}\label{lemma:2}
Suppose that $(a,q)=(a+h,q)=(d_1,q)=(d_2,q)=(d_1,d_2)=1$ and $4|h,$ where $d_1,d_2,q$ are square-free and odd, of size $\ll N^{O(1)}.$ Moreover suppose $h>0$ is fixed such that $p|h\Rightarrow p|2q$. Then we have 
$$\sum_{\substack{n\leq N\\ n\equiv a\,\,(\text{mod}\,\,q) \\ n\equiv 1\,\,(\text{mod}\,\,4) \\ d_1|n \\ d_2|n+h}}r(n)r(n+h) = \frac{g_1(q)^2\Gamma(d_1,d_2,q)}{q}\pi^2N+R_2(N;d_1,d_2,q),$$
where
$$\Gamma(d_1,d_2,q) = \frac{g_2(d_1)g_2(d_2)}{d_1d_2}\sum_{\substack{(r,2q)=1}}\frac{\mu(r)(d_1,r)(d_2,r)\chi[(d_1^2,r)]\chi[(d_2^2,r)]}{r^2},
$$
and
$$
R_2(N;d_1,d_2,q)\ll_{\epsilon} q^{\frac{1}{2}}d_1d_2N^{\frac{3}{4}+\epsilon}+d_1^{\frac{1}{2}}d_2^{\frac{1}{2}}N^{\frac{5}{6}+\epsilon}.
$$
\end{lemma}

\begin{lemma}\label{lemma:3}
Let $(a,q)=(d,q)=1$ where $d,q$ are square-free and odd, of size $\ll N^{O(1)}.$ Then we have 
\begin{align*}
\sum_{\substack{n\leq N \\ n\equiv a\,\,(\text{mod}\,\,q) \\ n\equiv 1\,\,(\text{mod}\,\,4) \\ d|n}}r^2(n) &= 
\frac{g_3(q)g_4(d)}{qd}\bigg(\log{N}+A_2+2\sum_{p|q}g_5(p)-2\sum_{p|d}g_6(p)\bigg)N+O_{\epsilon}(qN^{\frac{3}{4}+\epsilon}),
\end{align*}
where
$$A_2=2\gamma-1+2\frac{L'(1,\chi_4)}{L(1,\chi_4)}-2\frac{\zeta'(2)}{\zeta(2)}+\frac{4}{3}\log{2}.$$
Here $g_3,g_4$ are the multiplicative functions defined on primes by 
\begin{align*}
g_3(p)
&=
\begin{cases}
\frac{(p-1)^2}{p(p+1)}\,\,&\text{if $p\equiv 1\,\,(\text{mod}\,\,4),$} \\
g_1(p)\,\,&\text{if $p\equiv 3\,\,(\text{mod}\,\,4),$} 
\end{cases} \,\,\,\,\,\,\,\,\,\,\,\,\,
g_4(p)
=
\begin{cases}
\frac{4p^2-3p+1}{p(p+1)}\,\,&\text{if $p\equiv 1\,\,(\text{mod}\,\,4),$} \\
g_2(p)\,\,&\text{if $p\equiv 3\,\,(\text{mod}\,\,4),$}
\end{cases}
\end{align*}
and $g_5(p),g_6(p)$ are defined by
\begin{align*}
g_5(p)
&=
\begin{cases}
\frac{(2p+1)\log{p}}{p^2-1}\,\,&\text{if $p\equiv 1\,\,(\text{mod}\,\,4),$} \\
\frac{\log{p}}{p^2-1}\,\,&\text{if $p\equiv 3\,\,(\text{mod}\,\,4),$} 
\end{cases} \,\,\,\,\,\,\,\,\,\,\,\,\,
g_6(p)
=
\begin{cases}
\frac{(p-1)^2(2p+1)\log{p}}{(p+1)(4p^2-3p+1)}\,\,&\text{if $p\equiv 1\,\,(\text{mod}\,\,4),$} \\
\log{p}\,\,&\text{if $p\equiv 3\,\,(\text{mod}\,\,4).$}
\end{cases}
\end{align*}

%
\end{lemma}

We will prove each of these results in Appendix A. Lemma~\ref{lemma:1} follows from two known results. Lemma~\ref{lemma:2} follows by adapting the method used by Plaksin in~\cite{Plak}, where a similar sum is considered. Finally Lemma~\ref{lemma:3} can be shown using standard Perron's formula arguments, together with a fourth moment estimate for Dirichlet L-functions. 

When finding the corresponding estimates for $\rho(n)$ the following sums naturally appear (for the definitions of $W,W_1$ and $D_0$ see (\ref{eq:W}) below the fold):
\begin{align} \label{eq:X}
X_{N,W}&=\sum_{\substack{a\leq v \\ (a,W)=1 \\ p|a\Rightarrow p\equiv 1\,\,(\text{mod}\,\,4)}} \frac{\mu(a)}{a}\log{\frac{v}{a}},\\ \label{eq:Y}
Y_{N,W}&=\sum_{\substack{a,b\leq v\\ (a,W)=(b,W)=1 \\ (a,b)=1 \\ p|a,b\Rightarrow p\equiv 1\,\,(\text{mod}\,\,4)}} \frac{\mu(a)\mu(b)}{g_7(a)g_7(b)}\log{\frac{v}{a}}\log{\frac{v}{b}}, \\ \label{eq:Z1}
Z_{N,W}^{(1)}&=\sum_{\substack{a,b\leq v\\ (a,W)=(b,W)=1 \\ p|a,b\Rightarrow p\equiv 1\,\,(\text{mod}\,\,4)}}\frac{\mu(a)\mu(b)g_4([a,b])}{g_2(a)g_2(b)[a,b]}\log{\frac{v}{a}}\log{\frac{v}{b}}, \\ \label{eq:Z2}
Z_{N,W}^{(2)}&=\sum_{\substack{a,b\leq v\\ (a,W)=(b,W)=1 \\ p|a,b\Rightarrow p\equiv 1\,\,(\text{mod}\,\,4)}}\frac{\mu(a)\mu(b)g_4([a,b])}{g_2(a)g_2(b)[a,b]}\log{\frac{v}{a}}\log{\frac{v}{b}}\sum_{p|[a,b]}g_6(p).
\end{align}
Here $g_7$ is the multiplicative function defined on primes by $g_7(p)=p+1.$ The following lemma evaluates the auxiliary sums above.

\begin{lemma}[Auxiliary estimates for $\rho(n)$]\label{lemma:aux}
We have 
\begin{align*}
X_{N,W} &= (1+o(1))\frac{8A\log^{\frac{1}{2}}{v}}{\pi g_1(W_1)}, \\
Y_{N,W} &= (1+o(1))\frac{64A^2\log{v}}{\pi^2g_1(W_1)^2}, \\
Z_{N,W}^{(1)} &=(1+o(1))\frac{32A^3\log^{\frac{1}{2}}{v}}{\pi^2g_1(W_1)^3}, \\
Z_{N,W}^{(2)} &=-(1+o(1))\frac{16A^3\log^{\frac{3}{2}}{v}}{\pi^2g_1(W_1)^3},
\end{align*}
where $A$ is defined as in (\ref{eq:A}). In each case one may take the $o(1)$ term to be $O(D_0^{-1}).$
\end{lemma}

We remark that the estimate for $X_{N,1}$ appears in Hooley's work (after correcting a misprint - cf.~\hbox{\cite[Lemma 5]{Hooley}} and note his slightly different definition of $A$). We prove Lemma \ref{lemma:aux} in Appendix B. Each sum can be evaluated by the Selberg-Delange method.


\section{The sieve set-up}\label{section:sieveresults}

We now state our sieve results and use them to deduce Theorem~\ref{theo:intro1.1}. For the rest of the paper $k$ is fixed, $\mathcal{H}=\{h_1,\ldots,h_k\}$ is a fixed admissible set such that $4|h_i$ for each $i,$ and $N$ is sufficiently large in terms of any fixed quantity. We allow any of the constants hidden in the Landau notation to depend on $k,$ without explicitly specifying so. 

We will employ a $4W$-trick in our sieve calculations. Let 
\begin{equation}\label{eq:W}
W=\prod_{2<p\leq D_0} p,
\end{equation}
where $D_0=(\log\log{N})^3,$ so that $W\ll (\log{N})^{2(\log\log{N})^2} \ll_{\epsilon} N^{\epsilon}$ for any fixed $\epsilon>0$ by the prime number theorem. It will prove useful to define 
\begin{equation}
W_1=\prod_{\substack{p\leq D_0 \\ p\equiv 1\,\,(\text{mod}\,\,4)}}p,\,\,\,\,\,\,\hspace{10mm}W_3=\prod_{\substack{p\leq D_0 \\ p\equiv 3\,\,(\text{mod}\,\,4)}}p,
\end{equation}
so that $W=W_1W_3.$ 

By admissibility of $\mathcal{H}$ there exists a fixed residue class $v_0\,\,(\text{mod}\,\,W)$ such that $(v_0+h_i,W)=1$ for each $i$. Fix $1\leq m,l\leq k$ with $m\neq l.$ We consider four types of sums:

\begin{align}
S_{1} &= \sum_{\substack{N\leq n< 2N \\ n\equiv v_0\,\,(\text{mod}\,\,W) \\ n\equiv 1\,\,(\text{mod}\,\,4)}}\bigg(\sum_{\substack{d_1,\ldots,d_k \\d_i|n+h_i \,\,\forall i}}\lambda_{d_1,\ldots,d_k}\bigg)^{2}, \\
S_{2}^{(m)}&= \sum_{\substack{N\leq n< 2N \\ n\equiv v_0\,\,(\text{mod}\,\,W) \\ n\equiv 1\,\,(\text{mod}\,\,4)}} \rho(n+h_m) \bigg(\sum_{\substack{d_1,\ldots,d_k \\d_i|n+h_i \,\,\forall i}}\lambda_{d_1,\ldots,d_k}\bigg)^{2}, \\
S_{3}^{(m,l)}&= \sum_{\substack{N\leq n< 2N \\ n\equiv v_0\,\,(\text{mod}\,\,W) \\ n\equiv 1\,\,(\text{mod}\,\,4)}} \rho(n+h_m)\rho(n+h_l) \bigg(\sum_{\substack{d_1,\ldots,d_k \\d_i|n+h_i \,\,\forall i}}\lambda_{d_1,\ldots,d_k}\bigg)^{2}, \\
S_4^{(m)}&= \sum_{\substack{N\leq n< 2N \\ n\equiv v_0\,\,(\text{mod}\,\,W) \\ n\equiv 1\,\,(\text{mod}\,\,4)}} \rho^{2}(n+h_m) \bigg(\sum_{\substack{d_1,\ldots,d_k \\d_i|n+h_i \,\,\forall i}}\lambda_{d_1,\ldots,d_k}\bigg)^{2}. 
\end{align}
Because $k$ is fixed, we may assume that $D_0$ is sufficiently large so that
\begin{equation}
p|h_i-h_j \Rightarrow p|2W.
\end{equation}
\begin{remark}
For the second moment estimate, it proves important to control the residue classes of the translates $n+h\,\,(\text{mod}\,\,4),$ hence the condition $n\equiv 1\,\,(\text{mod}\,\,4)$ in our sieve sums and also the assumption $4|h$ for our admissible set. This is because of the inherent bias numbers representable as a sum of two squares have modulo 4. 
\end{remark}


Our first Proposition evaluates these sums for general half-dimensional Maynard-Tao sieve weights. Fix $0<\theta_1<1/18$ in the definition of $\rho(n)$ (see (\ref{eq:HooleysRho})). We also define the normalisation constant
\begin{equation}\label{eq:B}
B= \frac{A}{\Gamma(1/2)\sqrt{L(1,\chi_4)}}\cdot\frac{\varphi(W_3)(\log{R})^{\frac{1}{2}}}{W_3} = \frac{2A}{\pi}\cdot\frac{\varphi(W_3)(\log{R})^{\frac{1}{2}}}{W_3}.
\end{equation}

\begin{prop}[Half-dimensional Maynard-Tao sieve estimates]\label{prop:sieve}
Let $R=N^{\theta_2/2}$ for some small fixed positive constant $\theta_2$ such that $0<\theta_1+\theta_2<1/18$. Let $\lambda_{d_1,\ldots,d_k}$ be defined in terms of a fixed smooth function $F$ by 
$$\lambda_{d_1,\ldots,d_k} = \bigg(\prod_{i=1}^{k}\mu(d_i)d_i\bigg)\sum_{\substack{r_1,\ldots,r_k \\ d_i|r_i \forall i \\ (r_i,W)=1\forall i \\ p|r_i \Rightarrow p\equiv 3\,\,(\text{mod}\,\,4)\forall i}} \frac{\mu(\prod_{i=1}^{k}r_i)^{2}}{\prod_{i=1}^{k}\varphi(r_i)}F\bigg(\frac{\log{r_1}}{\log{R}},\ldots,\frac{\log{r_k}}{\log{R}}\bigg),$$
whenever $\prod_{i=1}^{k}d_i\leq R$ is squarefree, $(\prod_{i=1}^{k}d_i,W)=1$ and $p|\prod_{i=1}^{k}d_i\Rightarrow p\equiv 3\,\,(\text{mod}\,\,4),$ and let $\lambda_{d_1,\ldots,d_k}=0$ otherwise. Moreover let $F$ be supported on $R_{k}=\{(x_1,\ldots,x_k)\in[0,1]^{k}: \sum_{i=1}^{k}x_i \leq 1\}.$ Then we have 
\begin{align*}
S_{1}  &=(1+o(1)) \frac{B^{k}N}{4W}L_{k}(F), \\
S_2^{(m)} &= (1+o(1)) \frac{4\sqrt{\frac{\log{R}}{\log{v}}}B^kN}{\pi W}L_{k;m}(F), \\
 S_{3}^{(m,l)} &= (1+o(1)) \frac{64(\frac{\log{R}}{\log{v}})B^{k}N}{\pi^2 W}L_{k;m,l}(F), \\
 S_{4}^{(m)} &= (1+o(1))\frac{2\sqrt{\frac{\log{R}}{\log{v}}}(\frac{\log{N}}{\log{v}}+1)B^{k}N}{\pi W}L_{k;m}(F)
\end{align*}
provided $L_k(F),L_{k;m}(F)$ and $L_{k;m,l}(F)$ are non-zero, where
\begin{align*}
L_{k}(F)&= \int_0^1\ldots \int_0^1 \bigg[F(x_1,\ldots,x_k)\bigg]^2 \prod_{i=1}^{k}\frac{\mathrm{d}x_i}{\sqrt{x_i}} \\
L_{k;m}(F)&= \int_0^1\ldots \int_0^1 \bigg[\int_0^1F(x_1,\ldots,x_k)\frac{\mathrm{d}x_m}{\sqrt{x_m}}\bigg]^2 \prod_{\substack{i=1 \\ i\neq m}}^{k}\frac{\mathrm{d}x_i}{\sqrt{x_i}}, \\
L_{k;m,l}(F)&= \int_0^1\ldots \int_0^1 \bigg[\int_0^1\bigg(\int_0^1F(x_1,\ldots,x_k)\frac{\mathrm{d}x_m}{\sqrt{x_m}}\bigg)\frac{\mathrm{d}x_l}{\sqrt{x_l}}\bigg]^2  \prod_{\substack{i=1 \\ i\neq m,l}}^{k}\frac{\mathrm{d}x_i}{\sqrt{x_i}}. 
\end{align*}
\end{prop}

From this, one can deduce the corresponding results for the modification of the Maynard-Tao sieve described in Section~2.

\begin{prop}[Modified Maynard-Tao sieve estimates]\label{prop:sievemod}
Suppose in addition to the hypotheses of Proposition~\ref{prop:sieve} we have a partition $\mathcal{H}=\{h_1,\ldots,h_k\}=B_1\cup\ldots\cup B_M$ into bins $B_i$ of fixed and finite size $k_i.$ Write $B_i=\{h_{k_0+\ldots+k_{i-1}+1},\ldots,h_{k_0+\ldots+k_i}\}$ with the convention that $k_0=0$. Suppose further we have a corresponding factorisation 
\begin{equation*}
F(x_1,\ldots,x_k)=\prod_{i=1}^{M}F_i(x_{k_0+\ldots+k_{i-1}+1},\ldots,x_{k_0+\ldots+k_i}),
\end{equation*}
where each $F_i$ is smooth and supported on the simplex 
\begin{equation*}
R_{B_i,\beta_i}=\{(x_{k_0+\ldots+k_{i-1}+1},\ldots,x_{k_0+\ldots+k_i})\in[0,1]^{k_i}:0\leq \sum_{j=k_0+\ldots+k_{i-1}+1}^{k_0+\ldots+k_i}x_j\leq \beta_i\}.
\end{equation*}
Here $(\beta_i)_{i=1}^{\infty}$ is a sequence of real numbers such that $\sum_{i=1}^{\infty}\beta_i\leq 1.$ Then for $h_m,h_l\in B_j$ we have
\begin{align*}
S_{1}  &=(1+o(1)) \frac{B^{k}N}{4W}\bigg(\prod_{i=1}^{M}L_{|B_i|}(F)\bigg), \\
S_2^{(m)} &= (1+o(1)) \frac{4\sqrt{\frac{\log{R}}{\log{v}}}B^kN}{\pi W}\bigg(\prod_{\substack{i=1}}^{M}L_{|B_i|}(F_i)\bigg)\frac{L_{|B_j|;m}(F_j)}{L_{|B_j|}(F_j)},  \\
 S_{3}^{(m,l)} &= (1+o(1)) \frac{64(\frac{\log{R}}{\log{v}})B^{k}N}{\pi^2W}\bigg(\prod_{\substack{i=1}}^{M}L_{|B_i|}(F_i)\bigg)\frac{L_{|B_j|;m,l}(F_j)}{L_{|B_j|}(F_j)}, \\
 S_{4}^{(m)} &= (1+o(1))\frac{2\sqrt{\frac{\log{R}}{\log{v}}}(\frac{\log{N}}{\log{v}}+1)B^{k}N}{\pi W}\bigg(\prod_{\substack{i=1}}^{M}L_{|B_i|}(F_i)\bigg)\frac{L_{|B_j|;m}(F_j)}{L_{|B_j|}(F_j)}. 
\end{align*}
\end{prop}
\begin{proof}
The hypotheses imply $F=\prod_{i=1}^{M}F_i$ is also smooth and supported on $R_{k},$ and hence the results of Proposition~\ref{prop:sieve} apply. It suffices to to show the functionals factorise in the forms stated. Because our set-up ensures that $\text{supp}(F)= \text{supp}(F_1)\times\ldots\times\text{supp}(F_M)$ one can easily check that if $h_m,h_l\in B_j$ then
\begin{align*}
L_k(F) &= \prod_{i=1}^{M} L_{|B_i|}(F_i),\\
L_{k;m}(F) &= \bigg(\prod_{\substack{i=1}}^{M}L_{|B_i|}^{(0)}(F_i)\bigg)\frac{L_{|B_j|:m}(F_j)}{L_{|B_j|}(F_j)}, \\
L_{k;m,l}(F) &= \bigg(\prod_{\substack{i=1}}^{M}L_{|B_i|}^{(0)}(F_i)\bigg)\frac{L_{|B_j|:m,l}(F_j)}{L_{|B_j|}(F_j)}. 
\end{align*}
\end{proof}

With the following lemma we will be in a position to prove Theorem \ref{theo:intro1.1}.

\begin{lemma}[Evaluation of sieve functionals]\label{lemma:functiondef}
Let $F(t_1,\ldots,t_k)=\prod_{i=1}^{k}g(kt_i)$ where 
$$
g(t)=
\begin{cases}
\frac{1}{1+\frac{t}{\beta}},\,\,\,&\text{if $t\leq \beta$,} \\
0,\,\,\,&\text{otherwise.}
\end{cases}
$$
Then for any $m,l$ we have
\begin{align*}
\frac{L_{k;m}(F)}{L_k(F)} &= \frac{\pi^2}{\pi+2}\cdot \sqrt{\frac{\beta}{k}}, \\
\frac{L_{k:m,l}(F)}{L_k(F)} &= \bigg(\frac{\pi^2}{\pi+2}\bigg)^2\cdot \frac{\beta}{k}.  \\
\end{align*}
\end{lemma}

\begin{proof}
The definition implies $F$ is supported on the cube $[0,\frac{\beta}{k}]^k\subseteq R_{k,\beta}.$ In this case the functionals factorise completely and the lemma follows from the fact
\begin{align*}
\int_0^{\frac{\beta}{k}}\frac{\mathrm{d}t}{\sqrt{t}(1+\frac{kt}{\beta})^2}&=\frac{\pi+2}{4}\cdot \sqrt{\frac{\beta}{k}}, \\
\int_0^{\frac{\beta}{k}}\frac{\mathrm{d}t}{\sqrt{t}(1+\frac{kt}{\beta})}&=\frac{\pi}{2}\cdot \sqrt{\frac{\beta}{k}}.
\end{align*}
\end{proof}

\begin{remark}
We have restricted the support of our functions to the cube $[0,\frac{\beta}{k}]^k\subseteq R_{k,\beta}$ so that the integrals can be evaluated exactly. This essentially means we are using weights of similar strength to the original GPY weights~(cf. \cite{GPY}). For the half-dimensional case one can show that for large $k$ these weights are essentially optimal. (In particular, following a similar optimisation process as in~\hbox{\cite[Section 7]{Maynard}} one arrives at the same results as above.)
\end{remark}

We are now in a position to prove Theorem \ref{theo:intro1.1}.
\begin{proof}[Proof of Theorem \ref{theo:intro1.1}] Let $\mathcal{H}=\{h_1,h_2,\ldots\}$ be a fixed admissible set. Fix real numbers $\theta_1,\theta_2$ subject to $0<\theta_1+\theta_2 < 1/18$ and define the constant
\begin{equation*}
\Delta= \Delta(\theta_1,\theta_2) = \frac{\sqrt{2}(\pi+2)}{32\pi}\frac{1+\theta_1}{\sqrt{\theta_1\theta_2}}.
\end{equation*}
With notation as above, consider a partition $\mathcal{H}=B_1\cup B_2\cup \ldots$ where $k_1>2\Delta^3$ and for $i\geq 2$ we choose $B_{i}$ such that $k_i > 2^{7i}$. By Proposition \ref{prop:2ndmoment} we will be done if we can show, for every $M\geq 1,$ the inequality
\begin{equation}\label{eq:variance}
\sum_{\substack{N\leq n<2N \\ n\equiv v_0\,\,(\text{mod}\,\,W) \\ n\equiv 1\,\,(\text{mod}\,\,4)}}\bigg[\min_{j=1,\ldots,M}\frac{\mu_j^2}{t_j^2}-\sum_{i=1}^{M}\bigg(\frac{\sum_{h\in B_i}\rho(n+h)-\mu_i}{t_i}\bigg)^2\bigg]\bigg(\sum_{\substack{d_1,\ldots,d_k \\ d_i|n+h_i\,\,\forall i}}\lambda_{d_1,\ldots,d_k}\bigg)^2 >0
\end{equation}
for all sufficiently large $N$ and some choice of real numbers $\mu_i,t_i \geq 1.$ Choose weights $\lambda_{d_1,\ldots,d_k}$ as in Proposition \ref{prop:sieve}, and let $F(x_1,\ldots,x_k)=\prod_{i=1}^{M}F_i(x_{k_0+\ldots+k_{i-1}+1},\ldots,x_{k_0+\ldots+k_i})$ where each $F_i$ is supported on $R_{B_i,2^{-i}}.$ Let $F_i(x_{k_0+\ldots+k_{i-1}+1},\ldots,x_{k_0+\ldots+k_i}) = \prod_{j=k_0+\ldots+k_{i-1}+1}^{k_0+\ldots+k_i}g(k_ix_j)$ where for $j\in \{b_{2i-1},\ldots,{b_{2i}}\}$ we define
\begin{equation*}
g(x_j)=
\begin{cases}
\frac{1}{1+2^ix_j},\,\,\,&\text{if $x_j\leq 2^{-i}$} \\
0,\,\,\,&\text{otherwise}
\end{cases}
\end{equation*}
Expanding out (\ref{eq:variance}), we have the evaluate the expression
\begin{equation*}
\bigg(\min_{j=1,\ldots,M}\frac{\mu_j^2}{t_j^2}\bigg) S_1 - \sum_{i=1}^{M} \frac{1}{t_i^2}\bigg[\sum_{\substack{h,h'\in B_i \\ h\neq h'}}S_3^{(h,h')} + \sum_{h\in B_i} S_4^{(h)} -  2\mu_i\sum_{h\in B_i}S_2^{(h)}+\mu_i^2S_1\bigg]
\end{equation*}
(where by abuse of notation we have written $S_2^{(h)}$ for $S_2^{(i)}$ where $h=h_i$ say). A convenient choice of $\mu_i, \lambda_i$ is 
\begin{align*}
\mu_i &=c\bigg(\frac{k_i}{2^i}\bigg)^{\frac{1}{2}},\,\,\,\,\,t_i =c\bigg(\frac{k_i}{2^i}\bigg)^{\frac{1}{3}},
\end{align*}
where
\begin{equation*}
c = c(\theta_1,\theta_2) = \frac{16\sqrt{\theta_2/2\theta_1}}{\pi}\bigg(\frac{\pi^2}{\pi+2}\bigg).
\end{equation*}
Evaluating these sums using Proposition \ref{prop:sievemod} and Lemma \ref{lemma:functiondef} we see that this is asymptotically
\begin{align*} 
&\frac{B^k N}{4W}\bigg(\prod_{i=1}^{M}L_{|B_i|}^{(0)}(F)\bigg) \bigg\{ \bigg(\frac{k_1}{2}\bigg)^{\frac{1}{3}} - \sum_{i=1}^{M}\bigg[\Delta\bigg(\frac{2^i}{k_i}\bigg)^{\frac{1}{6}} -\frac{1}{2^i} \bigg(\frac{2^i}{k_i}\bigg)^{\frac{2}{3}}\bigg]  \bigg\}.
\end{align*}
Hence (\ref{eq:variance}) will be satisfied for all sufficiently large $N$ provided
\begin{equation}\label{eq:rand10}
\Delta \sum_{i=1}^{M}\bigg(\frac{2^i}{k_i}\bigg)^{\frac{1}{6}} < \bigg(\frac{k_1}{2}\bigg)^{\frac{1}{3}}. 
\end{equation}
But now our choice of bins ensures that 
\begin{equation*}
\Delta \sum_{i=1}^{M}\bigg(\frac{2^i}{k_i}\bigg)^{\frac{1}{6}}  \leq \Delta \sum_{i=1}^{M}\frac{1}{2^i} \leq \Delta < \bigg(\frac{k_1}{2}\bigg)^{\frac{1}{3}},
\end{equation*}
and so (\ref{eq:rand10}) is satisfied for all $M \geq 1.$
\end{proof}

It remains to prove Proposition \ref{prop:sieve}. Each sum can be treated similarly. The following lemma handles all of them at once. First, given a function $F$ satisfying the hypotheses of Proposition~\ref{prop:sieve}, we define
\begin{equation}
F_{\text{max}} = \sup_{(t_1,\ldots,t_k)\in[0,1]^k}\bigg(|F(t_1,\ldots,t_k)|+\sum_{i=1}^{k}|\frac{\partial F}{\partial t_i}(t_1,\ldots,t_k)|\bigg).
\end{equation}

The lemma can now be stated as follows.

\begin{lemma}[General sieve lemma]\label{lemma:generalsieve}
Let $J\subseteq\{1,\ldots,k\}$ (possibly empty) and $p_1,p_2\in \mathbb{P}\cup\{1\}$ be fixed. Write $I=\{1,\ldots,k\}\backslash J.$ Define the sieve sum $S_{J,p_1,p_2,m}=S_{J,p_1,p_2,m,f,g}$ by
\begin{equation*}\label{eq:1}
S_{J,p_1,p_2,m}=\sum_{\substack{d_1,\ldots,d_k \\ e_1,\ldots,e_k \\ W,[d_1,e_1],\ldots,[d_k,e_k]\text{ coprime} \\ p_1|d_m,p_2|e_m}}\lambda_{d_1,\ldots,d_k}\lambda_{e_1,\ldots,e_k} \prod_{i\in I}f([d_i,e_i])\prod_{j\in J}g([d_j,e_j]),
\end{equation*}
with weights $\lambda_{d_1,\ldots,d_k}$ defined as in Proposition~\ref{prop:sieve}. If $J=\emptyset$ we define $f(p)=1/p$ (and there is no dependence on $g$ in the sum). Otherwise, $f$ and $g$ are non-zero multiplicative functions defined on primes by 
\begin{equation*}
f(p) = \frac{1}{p}+O\bigg(\frac{1}{p^2}\bigg),\,\,\,g(p) = \frac{1}{p^2}+O\bigg(\frac{1}{p^3}\bigg),
\end{equation*}
and moreover we assume that $f(p)\neq 1/p.$ We write $S_{J}$ for $S_{J,1,1,m}$. Suppose $\lambda_{d_1,\ldots,d_k}$ satisfy the same hypotheses as in Proposition \ref{prop:sieve}. Then for $|J|\in\{0,1,2\}$ we have the following:
\begin{enumerate}[label=(\roman*)]
	\item If $m\in J$ then 
	\begin{equation*}
	S_{J,p_1,p_2,m} \ll \frac{F_{\text{max}}^2B^{k+|J|}(\log\log{R})^2}{(p_1p_2/(p_1,p_2))^2} .
	\end{equation*}
	\item If $m\notin J$ then 
	\begin{equation*}
	S_{J,p_1,p_2,m} \ll \frac{F_{\text{max}}^2B^{k+|J|}(\log\log{R})^2}{p_1p_2/(p_1,p_2)} .
	\end{equation*}
	\item We have
	\begin{equation*}
	S_{J} = (1+o(1))B^{k+|J|}L_J(F),
	\end{equation*}
	where the integral operators are defined by Proposition \ref{prop:sieve} above, and we write $L_J(F)$ as shorthand for $L_{k;j\in J}(F).$ 
\end{enumerate}
\end{lemma}

We now show how this implies Proposition \ref{prop:sieve}.

\begin{proof}[Proof of (Lemma \ref{lemma:generalsieve} $\Rightarrow$ Proposition \ref{prop:sieve})]
We consider each sum in turn. First we note that using the definition of $\lambda_{d_1,\ldots,d_k},$ the exact same calculation as in \cite[p.~394]{Maynard} gives
\begin{equation*}
\sup_{d_1,\ldots,d_k}|\lambda_{d_1,\ldots,d_k}| \ll F_{\text{max}} \sum_{\substack{u\leq R \\ p|u\Rightarrow p\equiv 3\,\,(\text{mod}\,\,4)}}\frac{\mu^2(u)\tau_k(u)}{\varphi(u)} \ll F_{\text{max}} (\log{R})^{\frac{k}{2}},
\end{equation*}
and so we have a trivial bound\footnote{By Rankin's trick we have 
$$\sum_{\substack{d\leq R \\ p|d\Rightarrow p\equiv 3\,\,(\text{mod}\,\,4) }} \tau_k(d) \leq R\sum_{\substack{d\leq R \\ p|d\Rightarrow p\equiv 3\,\,(\text{mod}\,\,4) }} \frac{\tau_k(d)}{d} \ll R(\log{R})^{k/2}.$$
}
\begin{align} \nonumber
\sum_{\substack{d_1,\ldots,d_k \\ e_1,\ldots, e_k}}|\lambda_{d_1,\ldots,d_k}\lambda_{e_1,\ldots,e_k}| &\ll \bigg(\sup_{d_1,\ldots,d_k}|\lambda_{d_1,\ldots,d_k}|\bigg)^2 \bigg(\sum_{\substack{d\leq R \\ p|d\Rightarrow p\equiv 3\,\,(\text{mod}\,\,4) }} \tau_k(d)\bigg)^2 \\ \label{eq:trivial}
&\ll F_{\text{max}}^2R^2(\log{R})^{2k} \ll_{\epsilon} F_{\text{max}}^2N^{\theta_2+\epsilon}.
\end{align}
As mentioned in Section~4, because we can obtain power-saving in the error terms for the formulae stated there, this trivial bound will suffice for our purposes.
\\
\begin{enumerate}[label=(\roman*), wide, labelwidth=!, labelindent=0pt]
	\item Rewrite $S_1$ in the form
	\begin{equation*}
	S_1 = \sum_{\substack{d_1,\ldots,d_k \\ e_1,\ldots, e_k}}\lambda_{d_1,\ldots,d_k}\lambda_{e_1,\ldots,e_k}\sum_{\substack{N\leq n<2N \\ n\equiv v_0\,\,(\text{mod}\,\,W) \\ n\equiv 1\,\,(\text{mod}\,\,4) \\ n\equiv -h_i\,\,(\text{mod}\,\,[d_i,e_i])}}1.
	\end{equation*}
	We may assume $W,[d_1,e_1],\ldots,[d_k,e_k]$ are pairwise coprime, as otherwise the inner sum is empty. In this case, by the Chinese Remainder Theorem, these congruences are equivalent to a single congruence (mod $q$) where $q=4W\prod_{i=1}^{k}[d_i,e_i]$. The inner sum evaluates to 
	$$\frac{N}{q}+O(1).$$
	The error term contributes $O_{\epsilon}(F_{\text{max}}^2N^{\theta_2+\epsilon})$ which is negligible. The main term is 	\begin{equation*}
	\frac{N}{4W} \sum_{\substack{d_1,\ldots,d_k \\ e_1,\ldots, e_k \\ W,[d_1,e_1],\ldots,[d_k,e_k]\text{ coprime}}}\lambda_{d_1,\ldots,d_k}\lambda_{e_1,\ldots,e_k} \prod_{i=1}^{k}\frac{1}{[d_i,e_i]}.
	\end{equation*}
	This is of the form $S_{J}$ where $|J|=0$. Evaluating it according to Lemma~\ref{lemma:generalsieve} we obtain
	\begin{equation*}
	S_1 = (1+o(1))\frac{B^kN}{4W}L_k^{(0)}(F).
	\end{equation*}
	\\ 
	\item Rewrite $S_2^{(m)}$ in the form
	\begin{equation*}
	S_2^{(m)} = \sum_{\substack{d_1,\ldots,d_k \\ e_1,\ldots, e_k}}\lambda_{d_1,\ldots,d_k}\lambda_{e_1,\ldots,e_k} \sum_{\substack{N\leq n < 2N \\ n\equiv v_0\,\,(\text{mod}\,\,W) \\ n\equiv 1\,\,(\text{mod}\,\,4) \\ n\equiv -h_i\,\,(\text{mod}\,\,[d_i,e_i])\,\,\forall i}}\rho(n+h_m).
	\end{equation*}
	By definition of $\rho(n+h_m)$ this is equal to
	\begin{equation*}
	\frac{1}{\log{v}}\sum_{\substack{d_1,\ldots,d_k \\ e_1,\ldots, e_k}}\lambda_{d_1,\ldots,d_k}\lambda_{e_1,\ldots,e_k} \sum_{\substack{a\leq v \\ p|a\Rightarrow p\equiv 1\,\,(\text{mod}\,\,4)}} \frac{\mu(a)}{g_2(a)}\log{\frac{v}{a}}\sum_{\substack{N\leq n < 2N \\ n\equiv v_0\,\,(\text{mod}\,\,W) \\ n\equiv 1\,\,(\text{mod}\,\,4) \\ n\equiv -h_i\,\,(\text{mod}\,\,[d_i,e_i])\,\,\forall i \\ n\equiv -h_m\,\,(\text{mod}\,\,a)}}r(n+h_m).
	\end{equation*}
	From considering the support of $\lambda_{d_1,\ldots,d_k}$ we see that for non-zero contribution we may assume $W,[d_1,e_1],\ldots,[d_k,e_k]$ and $a$ are pairwise coprime. In this case the inner sum can be evaluated according to Lemma~\ref{lemma:1}, taking $q=W\prod_{i\neq m}[d_i,e_i]$ and $d=a[d_m,e_m]$. As $q \ll WR^2 \ll_{\epsilon} N^{\theta_2+\epsilon}$ and $d \ll vR^2 \ll N^{\theta_1+\theta_2}$ we see that the inner sum evaluates to 
	$$\frac{g_1(q)g_2(d)}{2qd}\pi N +O_{\epsilon}(N^{\frac{1}{3}+\frac{1}{2}(\theta_1+\theta_2)+\epsilon}).$$
	Bounding the sum over $a$ trivially by $v\log{v}$ and using (\ref{eq:trivial}), we see the error term contributes $O_{\epsilon}(N^{\frac{1}{3}+\frac{3}{2}(\theta_1+\theta_2)+\epsilon})$ which is negligible. We obtain a main term 
	\begin{align*}
	&\frac{X_{N,W}g_1(W)\pi N}{2W\log{v}}\sum_{\substack{d_1,\ldots,d_k \\ e_1,\ldots, e_k \\ W,[d_1,e_1],\ldots,[d_k,e_k]\text{ coprime}}}\lambda_{d_1,\ldots,d_k}\lambda_{e_1,\ldots,e_k}\prod_{\substack{i\neq m}}\frac{g_1([d_i,e_i])}{[d_i,e_i]} \frac{1}{[d_m,e_m]^2},
	\end{align*}
	where we have defined
	\begin{equation*}
	X_{N,W}=\sum_{\substack{a\leq v \\ (a,W)=1 \\ p|a\Rightarrow p\equiv 1\,\,(\text{mod}\,\,4)}} \frac{\mu(a)}{a}\log{\frac{v}{a}}
	\end{equation*}
	as in (\ref{eq:X}). The sieve sum above is of the form $S_{J}$ with $|J|=1.$ We can evaluate this by Lemma~\ref{lemma:generalsieve} to obtain
	\begin{equation*}
	S_2^{(m)} = (1+o(1))\frac{X_{N,W}g_1(W)\pi B^{k+1}N}{2W\log{v}}L_{k;m}^{(1)}(F).
	\end{equation*}
	Recalling the definition of $B$ in (\ref{eq:B}), evaluating $X_{N,W}$ according to Lemma~\ref{lemma:aux}, and using the fact
	\begin{equation*}
	\frac{g_1(W_3)\varphi(W_3)}{W_3} = \prod_{\substack{p|W_3}}\bigg(1-\frac{1}{p^2}\bigg) = \frac{1}{2A^2}+O(D_0^{-1}),
	\end{equation*}
	we obtain
	\begin{equation*}
	S_2^{(m)} = (1+o(1))\frac{4 \sqrt{\frac{\log{R}}{\log{v}}}B^kN}{\pi W}L_{k;m}^{(1)}(F).
	\end{equation*}
	\\ 
	\item Rewrite  $S_3^{(m,l)}$ in the form
	\begin{equation*}
	S_3^{(m,l)}=\sum_{\substack{d_1,\ldots,d_k \\ e_1,\ldots, e_k}}\lambda_{d_1,\ldots,d_k}\lambda_{e_1,\ldots,e_k} \sum_{\substack{N\leq n <2N \\ n\equiv v_0\,\,(\text{mod}\,\,W) \\ n\equiv 1\,\,(\text{mod}\,\,4) \\ n\equiv -h_i\,\,(\text{mod}\,\,[d_i,e_i])\,\,\forall i}}\rho(n+h_{m})\rho(n+h_l).
	\end{equation*}

Expanding out the definition of $\rho$ this is 

\begin{align*}
\frac{1}{\log^2{v}}\sum_{\substack{d_1,\ldots,d_k \\ e_1,\ldots, e_k}}\lambda_{d_1,\ldots,d_k}&\lambda_{e_1,\ldots,e_k}\sum_{\substack{a,b\leq v \\ p|a,b\Rightarrow p\equiv 1\,\,(\text{mod}\,\,4)}} \frac{\mu(a)\mu(b)}{g_2(a)g_2(b)}\log{\frac{v}{a}}\log{\frac{v}{b}} \\
&\cdot\sum_{\substack{N\leq n <2N \\ n\equiv v_0\,\,(\text{mod}\,\,W) \\ n\equiv 1\,\,(\text{mod}\,\,4) \\ n\equiv -h_i\,\,(\text{mod}\,\,[d_i,e_i])\,\,\forall i \\ n\equiv -h_{m}\,\,(\text{mod}\,\,a) \\ n\equiv -h_{l}\,\,(\text{mod}\,\,b)}}r(n+h_{m})r(n+h_l).
\end{align*}

Similarly to the above, for non-zero contribution we may restrict to the case $W,[d_1,e_1],\ldots,[d_k,e_k],a,b$ are pairwise coprime (note that the last two congruences are solvable if and only if $(a,b)|h_l-h_m,$ and in the case $(a,2W)=(b,2W)=1$ this is true if and only if $(a,b)=1$). We evaluate the inner sum according to Lemma~\ref{lemma:2}, taking $q=W\prod_{i\neq m,l}[d_i,e_i],$ $d_1 = a[d_m,e_m]$ and $d_2=b[d_l,e_l].$ We note that 
$q\ll_{\epsilon} N^{\theta_2+\epsilon}$ and $d_1,d_2\ll N^{\theta_1+\theta_2}.$ Using the fact $\theta_1+\theta_2<1/18,$ we see the second error term in the definition of $R_2(N;d_1,d_2,q)$ dominates, and so the inner sum evaluates to 
$$\frac{g_1(q)^2\Gamma(d_1,d_2,q)}{q}\pi^2N+ O_{\epsilon}(N^{\frac{5}{6}+\theta_1+\theta_2+\epsilon}).$$
Bounding the rest of the sum trivially, we obtain a total error of size $O_{\epsilon}(N^{\frac{5}{6}+3\theta_1+2\theta_2+\epsilon})$ which, again, is negligible in the range $\theta_1+\theta_2<1/18.$ We obtain a main term
\begin{align*}
&\frac{g_1(W)^2\pi^2N}{W\log^{2}{v}}\sum_{\substack{a,b\leq v \\ p|a,b\Rightarrow p\equiv 1\,\,(\text{mod}\,\,4)}} \frac{\mu(a)\mu(b)}{g_2(a)g_2(b)}\log{\frac{v}{a}}\log{\frac{v}{b}} \\ 
&\cdot\sum_{\substack{d_1,\ldots,d_k \\ e_1,\ldots, e_k \\ W,[d_1,e_1],\ldots,[d_k,e_k]\text{ coprime}}}\lambda_{d_1,\ldots,d_k}\lambda_{e_1,\ldots,e_k}\prod_{i\neq m,l}\frac{g_1([d_i,e_i])^2}{[d_i,e_i]} \Gamma([d_m,e_m]a,[d_l,e_l]b,W\prod_{i\neq m,l}[d_i,e_i]).
\end{align*}
For arbitrary (square-free) moduli $d_1,d_1$ and $q,$ we can write $\Gamma(d_1,d_2,q)$ as a product over primes (cf.~ the definition of $\Gamma(d_1,d_2,q)$ given in Lemma~\ref{lemma:2} and note that we are summing a multiplicative function). By considering the Euler-product and the various support restrictions on the variables $d_i,e_i,a,b$, one can write $\Gamma([d_m,e_m]a,[d_l,e_l]b,W\prod_{i\neq m,l}[d_i,e_i])$ in the form
\begin{align*} 
\prod_{p\nmid 2W}\bigg(1-\frac{1}{p^2}\bigg)^{-1}\frac{g_2(a)g_2(b)}{g_7(a)g_7(b)}\prod_{i\neq m,l} \frac{[d_i,e_i]}{g_1([d_i,e_i]\varphi([d_i,e_i])} \prod_{j=m,l} \frac{1}{[d_j,e_j]\varphi([d_j,e_j])},
\end{align*}
leaving us with a main term
\begin{align*}
&\prod_{p\nmid 2W}\bigg(1-\frac{1}{p^2}\bigg)^{-1}\frac{g_1(W)^2Y_{N,W}\pi^2N}{W\log^{2}{v}} \\
&\cdot\sum_{\substack{d_1,\ldots,d_k \\ e_1,\ldots, e_k \\ W,[d_1,e_1],\ldots,[d_k,e_k]\text{ coprime}}}\lambda_{d_1,\ldots,d_k}\lambda_{e_1,\ldots,e_k}\prod_{i\neq m,l}\frac{g_1([d_i,e_i])}{\varphi([d_i,e_i])}\prod_{j=m,l} \frac{1}{[d_j,e_j]\varphi([d_{j},e_{j}])}.
\end{align*}
Here we have defined
\begin{equation*}	
Y_{N,W}=\sum_{\substack{a,b\leq v\\ (a,W)=(b,W)=1 \\ (a,b)=1 \\ p|a,b \Rightarrow p\equiv 1\,\,(\text{mod}\,\,4)}} \frac{\mu(a)\mu(b)}{g_7(a)g_7(b)}\log{\frac{v}{a}}\log{\frac{v}{b}}
\end{equation*}
as in (\ref{eq:Y}). The main term is of the form $S_J$ for $|J|=2.$ By Lemma~\ref{lemma:generalsieve} it can be evaluated as
\begin{equation*}
S_{3}^{(l,m)} = (1+o(1))\frac{Y_{N,W}g_1(W)^2\pi^2B^{k+2}N}{W\log^2{v}}L_{k;m,l}^{(2)}(F),
\end{equation*}
where we have written 
\begin{equation*}
\prod_{p\nmid 2W}\bigg(1-\frac{1}{p^2}\bigg)^{-1}=1+O(D_0^{-1}).
\end{equation*}
Evaluating $Y_{N,W}$ as in Lemma \ref{lemma:aux}, this simplifies to 
\begin{equation*}
S_{3}^{(l,m)} = (1+o(1))\frac{64(\frac{\log{R}}{\log{v}}) B^{k}N}{\pi^2 W}L_{k;m,l}^{(2)}(F).
\end{equation*}
\\
	\item Rewrite $S_4^{(m)}$ in the form
	\begin{equation*}
	\sum_{\substack{d_1,\ldots,d_k \\ e_1,\ldots,e_k}}\lambda_{d_1,\ldots,d_k}\lambda_{e_1,\ldots,e_k}\sum_{\substack{N\leq n <2N \\ n\equiv v_0\,\,(\text{mod}\,\,W) \\ n\equiv 1\,\,(\text{mod}\,\,4) \\ n\equiv -h_i\,\,(\text{mod}\,\,[d_i,e_i])\,\,\forall i}}\rho^2(n+h_m).
	\end{equation*}
Expanding out the definition of $\rho^2(n)$ we see this is equal to 
\begin{equation*}
\frac{1}{\log^2{v}}\sum_{\substack{d_1,\ldots,d_k \\ e_1,\ldots,e_k}}\lambda_{d_1,\ldots,d_k}\lambda_{e_1,\ldots,e_k}\sum_{a,b\leq v}\frac{\mu(a)\mu(b)}{g_2(a)g_2(b)}\log{\frac{v}{a}}\log{\frac{v}{b}}\sum_{\substack{N\leq n <2N \\ n\equiv v_0\,\,(\text{mod}\,\,W) \\ n\equiv 1\,\,(\text{mod}\,\,4) \\ n\equiv -h_i\,\,(\text{mod}\,\,[d_i,e_i])\,\,\forall i \\ n\equiv -h_m\,\,(\text{mod}\,\,[a,b])}}r^2(n+h_m).
\end{equation*}
Again, we may restrict to the case $W,[d_1,e_1],\ldots,[d_k,e_k],[a,b]$ are pairwise coprime. In this case the inner sum can be evaluated according to Lemma \ref{lemma:3}, taking $q=W\prod_{i\neq m}[d_i,e_i]$ and $d=[a,b][d_m,e_m].$ We note that $q\ll_{\epsilon} N^{\theta_2+\epsilon}$ and $d\ll N^{2\theta_1+\theta_2},$ and so the inner sum becomes
\begin{equation*}
\frac{g_3(q)g_4(d)}{qd}\bigg(\log{N}+A_2+2\sum_{p|q}g_5(p)-2\sum_{p|d}g_6(p)\bigg)N+O_{\epsilon}(N^{\frac{3}{4}+\theta_2+\epsilon}).
\end{equation*}
Bounding the rest of the sum trivially, we see the error term contributes $O_{\epsilon}(N^{\frac{3}{4}+2(\theta_1+\theta_2)+\epsilon})$ which is small. 
For the main term, let
\begin{align*}
Z_{N,W}^{(1)} &= \sum_{a,b\leq v}\frac{\mu(a)\mu(b)g_4([a,b])}{g_2(a)g_2(b)[a,b]}\log{\frac{v}{a}}\log{\frac{v}{b}},\\
Z_{N,W}^{(2)} &= \sum_{a,b\leq v}\frac{\mu(a)\mu(b)g_4([a,b])}{g_2(a)g_2(b)[a,b]}\log{\frac{v}{a}}\log{\frac{v}{b}} \sum_{p|[a,b]}g_6(p)
\end{align*}
be as in~(\ref{eq:Z1}) and~(\ref{eq:Z2}). We can express $S_4^{(m)}=\Lambda_1+\Lambda_2+\Lambda_3+\Lambda_4$ where
\begin{align*} 
\Lambda_1 &= \frac{g_3(W)Z_{N,W}^{(1)}N}{W\log^2{v}}\bigg(\log{N}+A_2+2\sum_{p|W}g_5(p)\bigg)T,\\
\Lambda_2 &= \frac{2g_3(W)Z_{N,W}^{(1)}N}{W\log^2{v}} \sum_{i\neq m} \sum_{\substack{D_0<p\leq v \\ p\equiv 3\,\,(\text{mod}\,\,4)}}g_5(p)T^{(p,i)}, \\
\Lambda_3 &= -\frac{2g_3(W)Z_{N,W}^{(1)}N}{W\log^2{v}}\sum_{\substack{D_0<p\leq v \\ p\equiv 3\,\,(\text{mod}\,\,4)}}g_6(p)T^{(p,m)}, \\
\Lambda_4 &=-\frac{2g_3(W)Z_{N,W}^{(2)}N}{W\log^2{v}}T 
\end{align*}
and
\begin{align*}
T &= \sum_{\substack{d_1,\ldots,d_k \\ e_1,\ldots,e_k \\ W,[d_1,e_1],\ldots,[d_k,e_k]\text{ coprime}}}\lambda_{d_1,\ldots,d_k}\lambda_{e_1,\ldots,e_k}\prod_{i\neq m}\frac{g_1([d_i,e_i])}{[d_i,e_i]}\frac{1}{[d_m,e_m]^2}, \\
T^{(p,i)} &= \sum_{\substack{d_1,\ldots,d_k \\ e_1,\ldots,e_k \\ p|[d_i,e_i] \\ W,[d_1,e_1],\ldots,[d_k,e_k]\text{ coprime}}}\lambda_{d_1,\ldots,d_k}\lambda_{e_1,\ldots,e_k}\prod_{i\neq m}\frac{g_1([d_i,e_i])}{[d_i,e_i]}\frac{1}{[d_m,e_m]^2}.
\end{align*}
$T$ is of the form $S_{J}$ for $|J|=1,$ and so by Lemma~\ref{lemma:generalsieve} it can be evaluated as
\begin{equation*}
T = (1+o(1))B^{k+1}L_{k;m}^{(1)}(F).
\end{equation*}
To evaluate $T^{(p,i)},$ note by inclusion-exclusion we can write it as 
\begin{equation*}
\bigg(\sideset{}{'}\sum_{\substack{d_1,\ldots,d_k \\ e_1,\ldots,e_k \\ p|d_i }}+\sideset{}{'}\sum_{\substack{d_1,\ldots,d_k \\ e_1,\ldots,e_k \\ p|e_i}}-\sideset{}{'}\sum_{\substack{d_1,\ldots,d_k \\ e_1,\ldots,e_k \\ p|d_i,e_i}}\bigg)\lambda_{d_1,\ldots,d_k}\lambda_{e_1,\ldots,e_k}\prod_{i\neq m}\frac{g_1([d_i,e_i])}{[d_i,e_i]}\frac{1}{[d_m,e_m]^2},
\end{equation*}
where $\sideset{}{'}\sum$ denotes the condition $W,[d_1,e_1],\ldots,[d_k,e_k]$ are pairwise coprime. Thus we see it is of the form $S_{J,p,1,i}+S_{J,1,p,i}-S_{J,p,p,i}$ for $|J|=1.$ By Lemma~\ref{lemma:generalsieve} we conclude
\begin{equation*}
T^{(p,i)} \ll
\begin{cases}
\frac{F_{\text{max}}^2B^{k+1}(\log\log{R})^2}{p},\,\,&\text{if $i\neq m$} \\
\frac{F_{\text{max}}^2B^{k+1}(\log\log{R})^2}{p^2},\,\,&\text{if $i=m$}
\end{cases}
\end{equation*}
Now we note that 
\begin{align*}
\sum_{\substack{D_0<p\leq v \\ p\equiv 3\,\,(\text{mod}\,\,4)}}\frac{g_5(p)}{p} &\ll \sum_{p>D_0} \frac{\log{p}}{p^2} \ll \frac{\log{D_0}}{D_0},\\
\sum_{\substack{D_0<p\leq v \\ p\equiv 3\,\,(\text{mod}\,\,4)}}\frac{g_6(p)}{p^2}&\ll \sum_{p>D_0} \frac{\log{p}}{p^2} \ll \frac{\log{D_0}}{D_0},
\end{align*}
and so, with our choice $D_0 = (\log\log{N})^3,$ the contributions from $\Lambda_2$ and $\Lambda_3$ are negligible. Because
\begin{align*}
\sum_{\substack{p|W}}g_5(p) &\ll \sum_{p<D_0} \frac{\log{p}}{p} \ll \log{D_0},
\end{align*}
we see that the only contribution to the main term comes from the $\Lambda_1$ term corresponding to $\log{N},$ and $\Lambda_4,$ leaving us with
\begin{align*}
S_4^{(m)} = (1+o(1)) \frac{g_3(W)B^{k+1}N}{W\log^2{v}}\bigg[Z_{N,W}^{(1)}\log{N}-2Z_{N,W}^{(2)}\bigg]L_{k;m}^{(1)}(F).
\end{align*}
Evaluating these according to Lemma \ref{lemma:aux}, and using the fact
\begin{align*}
\frac{g_3(W_1)}{g_1(W_1)^3}=\prod_{\substack{p<D_0 \\ p\equiv1\,\,(\text{mod}\,\,4)}}\bigg(1-\frac{1}{p^2}\bigg)^{-1} = \frac{3\zeta(2)}{8A^2}+O(D_0^{-1}) = \frac{\pi^2}{16A^2}+O(D_0^{-1}),
\end{align*}
we obtain
\begin{align*}
S_4^{(m)} = (1+o(1)) \frac{2 \sqrt{\frac{\log{R}}{\log{v}}}(\frac{\log{N}}{\log{v}}+1)B^k N}{\pi W}L_{k;m}^{(1)}(F).
\end{align*}

\end{enumerate}
\vspace{2mm}
This finishes the proof of Proposition \ref{prop:sieve}.
\end{proof}

Thus it remains to establish Lemma~\ref{lemma:generalsieve}. First we require a few technical sieve lemmas. We list these in the following section.

\section{Technical sieve sums lemmas}\label{section:sievetechnical}

In the various sieve calculations that appear in the proof of Lemma~\ref{lemma:generalsieve}, we will frequently encounter sums of the form
\begin{equation*}
\sum_{\substack{n\leq X \\ p|n\Rightarrow p\equiv 3\,\,(\text{mod}\,\,4)}}\mu^2(n)f(n),
\end{equation*}
where $f$ is a multiplicative function satisfying $f(p)=O(1/p).$ We can evaluate sums of this type with the following lemmas.

\begin{lemma}[Technical sieve sum lemma]\label{lemma:technical}
Let $A_1,A_2,L>0.$ Let $\gamma$ be a multiplicative function satisfying the sieve axioms 
$$0\leq \frac{\gamma(p)}{p} \leq 1-A_1,$$
and 
$$-L\leq \sum_{w\leq p \leq z} \frac{\gamma(p)\log{p}}{p} - \frac{1}{2}\log{\frac{z}{w}} < A_2$$
for any $2\leq w\leq z.$ Let $g$ be the totally multiplicative function defined on primes by $g(p) = \frac{\gamma(p)}{p-\gamma(p)}$. Finally, let $G:[0,1]\rightarrow \mathbb{R}$ be a piecewise differentiable function, and let $G_{\text{max}}=\sup_{t\in [0,1]} (|G(t)|+|G'(t)|)$. Then 
\begin{equation*}
\sum_{d<z}\mu^2(d)g(d)G\bigg(\frac{\log{d}}{\log{z}}\bigg) = c_{\gamma} \frac{(\log{z})^{\frac{1}{2}}}{\Gamma(1/2)} \int_0^1 G(x)\frac{\mathrm{d}x}{\sqrt{x}}+O_{A_1,A_2}(c_{\gamma}LG_{\text{max}}(\log{z})^{-\frac{1}{2}}),
\end{equation*}
where
\begin{equation*}
c_{\gamma} = \prod_{p}\bigg(1-\frac{\gamma(p)}{p}\bigg)^{-1}\bigg(1-\frac{1}{p}\bigg)^{\frac{1}{2}}.
\end{equation*}
Here, the implied constant in the Landau notation is independent of $G$ and $L$.
\end{lemma}
\begin{proof}
This is \cite[Lemma 4]{GPY2} with slight changes to notation.
\end{proof}

To use this lemma in practice, we need to be able to evaluate the singular series $c_{\gamma}$ which appears. In the next lemma we do this for a function $\gamma(p)$ which covers the cases of interest to us.

\begin{lemma}[Evaluation of singular series]\label{lemma:singseries}
Let 
\begin{equation*}
\gamma(p)=
\begin{cases}
1+O(1/p)\,\,&\text{ if $p\nmid W, p\equiv 3\,\,(\text{mod}\,\,4)$,} \\
0\,\,&\text{otherwise,}
\end{cases}
\end{equation*}
With the notation of Lemma \ref{lemma:technical}, we have
\begin{align*}
c_{\gamma} &= \frac{A}{\sqrt{L(1)}}\cdot\frac{\varphi(W_3)}{W_3}(1+O(D_0^{-1}))
\end{align*}
where $A$ is the Ramanujan-Landau constant defined in~(\ref{eq:A}).
\end{lemma}

\begin{proof}
Let $\gamma(p)=1+\alpha(p)$ where $\alpha(p)=O(1/p).$ Define the auxiliary function
\begin{equation*}
\delta(p)=
\begin{cases}
1\,\,&\text{if $p\equiv 3\,\,(\text{mod}\,\,4)$,} \\
0\,\,&\text{otherwise.}
\end{cases}
\end{equation*}
One can easily show $c_{\delta} = A/\sqrt{L(1)}.$ The result follows because
\begin{equation*}
c_{\gamma} = c_{\delta}  \prod_{\substack{p|W \\ p\equiv 3\,\,(\text{mod}\,\,4)}}\bigg(1-\frac{1}{p}\bigg)\prod_{\substack{p\nmid W \\ p\equiv 3\,\,(\text{mod}\,\,4)}}\bigg(1-\frac{\alpha(p)}{p-1}\bigg)^{-1}.
\end{equation*}
The latter product is $1+O(D_0^{-1})$ by our assumption $\alpha(p)=O(1/p).$
\end{proof}

The next lemma collects both of these results together. First we recall the definition of the normalising constant from (\ref{eq:B}):
\begin{equation}\label{eq:B2}
B= \frac{2A \varphi(W_3) (\log{R})^{\frac{1}{2}}}{\pi W_3}.
\end{equation}
\begin{lemma}[Evaluation of sieve sums]\label{key}
Let $f$ be a multiplicative function such that 
$$f(p)=\frac{1}{p}+O\bigg(\frac{1}{p^2}\bigg).$$
Then for any piece-wise smooth function $G$ we have
\begin{align*}
&\sum_{\substack{d\leq R \\ (d,W)=1 \\ p|d \Rightarrow p\equiv 3\,\,(\text{mod}\,\,4)}}\mu^2(d)f(d)G\bigg(\frac{\log{d}}{\log{x}}\bigg)= B \int_0^1G(x)\frac{\mathrm{d}x}{\sqrt{x}} +O\bigg(\frac{G_{\text{max}}B}{D_0}\bigg).
\end{align*}
\end{lemma}
\begin{proof}
Let $f(p)=1/p+g(p)$ where $g(p)=O(1/p^2),$ and consider the function $\gamma$ defined on primes by
\begin{equation*}
\gamma(p)=
\begin{cases}
1-\frac{1}{p+1+pg(p)}\,\,&\text{if $p\nmid W, p\equiv 3\,\,(\text{mod}\,\,4)$,} \\
0,\,\,&\text{otherwise.}
\end{cases}
\end{equation*}
With this choice of $\gamma(p)$ we have
\begin{equation*}
\frac{\gamma(p)}{p-\gamma(p)} = f(p).
\end{equation*}
Note that 
\begin{align*}
\sum_{w\leq p \leq z} \frac{\gamma(p)\log{p}}{p} &= \sum_{\substack{w\leq p \leq z \\ p\equiv 3\,\,(\text{mod}\,\,4) \\ p\nmid W}} \frac{\log{p}}{p} +O\bigg( \sum_{\substack{w\leq p \leq z \\ p\equiv 3\,\,(\text{mod}\,\,4) \\ p\nmid W}} \frac{\log{p}}{p^2}\bigg)\\
&=\sum_{\substack{w\leq p \leq z \\ p\equiv 3\,\,(\text{mod}\,\,4)}} \frac{\log{p}}{p}+O\bigg(\sum_{\substack{w\leq p \leq z \\ p\equiv 3\,\,(\text{mod}\,\,4) \\ p|W}} \frac{\log{p}}{p}\bigg)+O(1) \\
&=\frac{1}{2}\log{\frac{z}{w}} + O(\log{D_0})+O(1).
\end{align*}
Therefore we can apply Lemma~\ref{lemma:technical} with $\gamma(p),$ taking $L\ll 1+\log{D_0}$ and $A_2$ a suitable constant. We obtain 
\begin{equation*}
\sum_{\substack{d\leq R \\ (d,W)=1 \\ p|d\Rightarrow p\equiv 3\,\,(\text{mod}\,\,4)}}\mu^2(d)f(d)G\bigg(\frac{\log{d}}{\log{x}}\bigg) = \frac{c_{\gamma}(\log{R})^{\frac{1}{2}}}{\Gamma(1/2)}\int_0^1G(x)\frac{\mathrm{d}x}{\sqrt{x}} +O\bigg(\frac{G_{\text{max}}c_{\gamma}(1+\log{D_0})}{(\log{R})^{\frac{1}{2}}}\bigg).
\end{equation*}
We can evaluate $c_{\gamma}$ by Lemma~\ref{lemma:singseries} to find
\begin{equation*}
c_{\gamma} = \frac{A}{\sqrt{L(1)}}\cdot\frac{\varphi(W_3)}{W_3}(1+O(D_0^{-1}))
\end{equation*}
When we substitute this back into our expression we see the error incurred here contributes $O(G_{\text{max}}B/D_0)$ and dominates. The result follows as $\Gamma(1/2)\sqrt{L(1)} = \pi/2.$
\end{proof}

We highlight the following two results, the first of which follows immediately from Lemma~\ref{key}, and the second of which is trivial.
\\
\begin{enumerate}[wide, labelwidth=!, labelindent=0pt]
	\item For multiplicative functions $f$ satisfying $f(p)=1/p+O(1/p^2)$ we have the upper bound 
\begin{equation}
\sum_{\substack{d\leq R \\ (d,W)=1 \\ p|d\Rightarrow p\equiv 3\,\,(\text{mod}\,\,4)}}\mu^2(d)f(d) \ll B.
\end{equation}
	\item For multiplicative functions $g$ satisfying $g(p)=O(1/p^2)$ we have the upper bound
	\begin{equation}
\sum_{\substack{d\leq R \\ (d,W)=1 \\ p|d\Rightarrow p\equiv 3\,\,(\text{mod}\,\,4)}}\mu^2(d)g(d) \ll 1.
\end{equation}
\end{enumerate}

These sums will appear frequently in our calculations, and we will use these bounds without comment in the arguments which follow.

\section{Establishing Lemma \ref{lemma:generalsieve}}\label{section:sieveproofs}
Our attention now turns to establishing  Lemma \ref{lemma:generalsieve}. We follow the combinatorial arguments used by Maynard - the steps which follow mirror those found in \cite{Maynard}. 
\subsection{Change of variables.} Our first step to evaluating the sums appearing in Lemma \ref{lemma:generalsieve} is to make a change of variables. We do so in the following proposition.
\begin{prop}[Diagonalising the sieve sum]\label{prop:2}
With notation as in Lemma \ref{lemma:generalsieve}, denote by $f^*,g^*$ the convolutions
\begin{equation*}
f^*=\mu*\frac{1}{f},\,\,\,\,\,g^*=\mu*\frac{1}{g}.
\end{equation*}
Define the diagonalising vectors $y_{r_1,\ldots,r_k}^{(J,p,m)} = y_{r_1,\ldots,r_k}^{(J,p,m,f,g)}$ by
\begin{equation*}
y_{r_1,\ldots,r_k}^{(J,p,m)}= \bigg(\prod_{i\in I}\mu(r_i)f^*(r_i)\bigg)\bigg(\prod_{j\in J}\mu(r_j)g^*(r_j)\bigg) \sum_{\substack{d_1,\ldots, d_k \\ r_i|d_i\,\,\forall i \\ p|d_m}}\lambda_{d_1,\ldots,d_k}\prod_{i\in I}f(d_i)\prod_{j\in J}g(d_j).
\end{equation*}
Let $y_{\text{max}}^{(J,p,m)}=\sup_{r_1,\ldots,r_k} |y_{r_1,\ldots,r_k}^{(J,p,m)}|$ and $\tilde{y}_{\text{max}}^{(J,p,m)}=\sup_{\substack{r_1,\ldots,r_k \\ (r_m,p)=1}} |y_{r_1,\ldots,r_k}^{(J,p,m)}|.$ (Note that these coincide if $p=1.$) Then we have
\begin{align*}
S_{J,p_1,p_2,m} &= \sum_{\substack{u_1,\ldots,u_k \\ (u_m,p_1p_2)=1 \\ u_j=1\,\,\forall j\in J}}\frac{(y_{u_1,\ldots,u_k}^{(J,p_1,m)})(y_{u_1,\ldots,u_k}^{(J,p_2,m)})}{\prod_{i\in I}f^*(u_i)\prod_{j\in J}g^*(u_j)}+E.
\end{align*}
If $m\in J$ then the (error) term $E$ satisfies
\begin{align*}
E \ll B^{|I|}\bigg[&\frac{(\tilde{y}_{\text{max}}^{(J,p_1,m)})(\tilde{y}_{\text{max}}^{(J,p_2,m)})}{D_0}+\frac{(y_{\text{max}}^{(J,p_1,m)})(y_{\text{max}}^{(J,p_2,m)})}{(p_1p_2/(p_1,p_2))^2} \\
&+\frac{(y_{\text{max}}^{(J,p_1,m)})(\tilde{y}_{\text{max}}^{(J,p_2,m)})}{p_1^2}+\frac{(\tilde{y}_{\text{max}}^{(J,p_1,m)})(y_{\text{max}}^{(J,p_2,m)})}{p_2^2}\bigg].
\end{align*}
If $m\notin J$ then $E$ satisfies a similar estimate, namely that which is obtained upon replacing all occurrences of $p_i^2$ with $p_i$ in the above expression, for $i\in \{1,2\}$. Moreover, in both of these cases, we adopt the convention that if $p_i=1$ then any term in our expression for $E$ involving $p_i$ in the denominator may be omitted.
\end{prop}


\begin{proof}
Recall the definition of $S_{J,p_1,p_2,m}$ given in Lemma~\ref{lemma:generalsieve}:
\begin{equation}
S_{J,p_1,p_2,m}=\sum_{\substack{d_1,\ldots,d_k \\ e_1,\ldots,e_k \\ W,[d_1,e_1],\ldots,[d_k,e_k]\text{ coprime} \\ p_1|d_m,p_2|e_m}}\lambda_{d_1,\ldots,d_k}\lambda_{e_1,\ldots,e_k} \prod_{i\in I}f([d_i,e_i])\prod_{j\in J}g([d_j,e_j]).
\end{equation}
We can write this in the form
\begin{equation}\label{eq:rand50}
S_{J,p_1,p_2,m}=\sum_{\substack{d_1,\ldots,d_k \\ e_1,\ldots,e_k \\ W,[d_1,e_1],\ldots,[d_k,e_k]\text{ coprime} \\ p_1|d_m, p_2|e_m}}\lambda_{d_1,\ldots,d_k}\lambda_{e_1,\ldots,e_k} \prod_{i\in I}\frac{f(d_i)f(e_i)}{f((d_i,e_i))}\prod_{j\in J}\frac{g(d_j)g(e_j)}{g((d_j,e_j))},
\end{equation}
using multiplicativity of the functions $f$ and $g,$ together with the fact $[d_i,e_i]$ is square-free for each $i$. We remark that because $f,g$ are non-zero, the functions $1/f,1/g$ are well-defined. We note the convolution identities
\begin{equation*}
\frac{1}{f((d_i,e_i))} = \sum_{u_i|d_i,e_i}f^*(u_i),\,\,\,\,\,\,\,\hspace{7mm}\frac{1}{g((d_j,e_j))} = \sum_{u_j|d_j,e_j}g^*(u_j)
\end{equation*}
for $f^*$ and $g^*$. Substituting these into (\ref{eq:rand50}) and swapping the order of summation, we obtain
\begin{align*}
&\sum_{\substack{u_1,\ldots,u_k}}\bigg(\prod_{i\in I}f^*(u_i)\bigg)\bigg(\prod_{j\in J}g^*(u_j)\bigg)\\
&\cdot\sum_{\substack{d_1,\ldots,d_k \\ e_1,\ldots,e_k \\ W,[d_1,e_1],\ldots,[d_k,e_k]\text{ coprime} \\ u_i|d_i,e_i\,\,\forall i\in I \\ u_j|d_j,e_j\,\,\forall j\in J \\ p_1|d_m,p_2|e_m}}\lambda_{d_1,\ldots,d_k}\lambda_{e_1,\ldots,e_k} \prod_{i\in I}f(d_i)f(e_i)\prod_{j\in J}g(d_j)g(e_j).
\end{align*}
From the support of the $\lambda_{d_1,\ldots,d_k},$ we see the only restriction coming from the pairwise coprimality of $W,[d_1,e_1],\ldots,[d_k,e_k]$ is from the possibility $(d_i,e_j)\neq 1$ for $i\neq j$. We can take care of this constraint by M\"obius inversion: multiplying by $\sum_{s_{i,j}|d_i,e_j}\mu(s_{i,j})$ for all $i\neq j,$ we obtain
\begin{align} \nonumber
&\sum_{\substack{u_1,\ldots,u_k}}\bigg(\prod_{i\in I}f^*(u_i)\bigg)\bigg(\prod_{j\in J}g^*(u_j)\bigg) \sum_{s_{1,2},\ldots,s_{k-1,k}}\bigg(\prod_{\substack{1\leq i,j\leq k \\ i\neq j}}\mu(s_{i,j})\bigg) \\ \label{eq:rand11}
&\cdot\sum_{\substack{d_1,\ldots,d_k \\ e_1,\ldots,e_k \\ u_i|d_i,e_i\,\,\forall i\in I \\ u_j|d_j,e_j\,\,\forall j\in J \\ s_{i,j}|d_i,e_j\,\,\forall i\neq j \\ p_1|d_m, p_2|e_m}}\lambda_{d_1,\ldots,d_k}\lambda_{e_1,\ldots,e_k} \prod_{i\in I}f(d_i)f(e_i)\prod_{j\in J}g(d_j)g(e_j).
\end{align}
We may restrict to the case where $s_{i,j}$ is coprime to $s_{i,a},s_{b,j}$ and $u_i,u_j,$ for $a\neq j$ and $b\neq i,$ because the vectors $\lambda_{d_1,\ldots,d_k}$ are supported on square-free integers $d=\prod_{i=1}^{k}d_i.$ Denote the sum over $s_{i,j}$ with these conditions by $\sideset{}{'}\sum$. Define the diagonalising vectors 
\begin{equation}\label{eq:diagvec}
y_{r_1,\ldots,r_k}^{(J,p,m)}= \bigg(\prod_{i\in I}\mu(r_i)f^*(r_i)\bigg)\bigg(\prod_{j\in J}\mu(r_j)g^*(r_j)\bigg) \sum_{\substack{d_1,\ldots, d_k \\ r_i|d_i\,\,\forall i \\ p|d_m}}\lambda_{d_1,\ldots,d_k}\prod_{i\in I}f(d_i)\prod_{j\in J}g(d_j).
\end{equation}
From the support of $\lambda_{d_1,\ldots,d_k}$ we see that $y_{r_1,\ldots,r_k}^{(J,p,m)}$ is also supported on $r_1,\ldots,r_k$ with $r=\prod_{i=1}^{k}r_i$ square-free, $(r,W)=1$ and $p|r\Rightarrow p\equiv 3\,\,(\text{mod}\,\,4).$ We claim this change of variables is invertible. Indeed, from the definition (\ref{eq:diagvec}), for $d_1,\ldots,d_k$ with $\prod_{i=1}^{k}d_i$ square-free, we have
\begin{align} \nonumber
\sum_{\substack{r_1,\ldots,r_k \\ d_i|r_i\,\,\forall i}}\frac{y_{r_1,\ldots,r_k}^{(J,p,m)}}{\prod_{i\in I}f^*(r_i)\prod_{j\in J}g^*(r_j)} &=\sum_{\substack{r_1,\ldots,r_k \\ d_i|r_i\,\,\forall i }}\prod_{i=1}^{k}\mu(r_i) \sum_{\substack{e_1,\ldots,e_k \\ r_i|e_i\,\,\forall i \\ p|e_m}}\lambda_{e_1,\ldots,e_k}\prod_{i\in I}f(e_i)\prod_{j\in J}g(e_j) \\ \nonumber
&= \sum_{\substack{e_1,\ldots,e_k }}1_{p|e_m} \lambda_{e_1,\ldots,e_k}\prod_{i\in I}f(e_i)\prod_{j\in J}g(e_j)   \sum_{\substack{r_1,\ldots,r_k \\ r_i|e_i\,\,\forall i \\ d_i|r_i\,\,\forall i}}\prod_{i=1}^{k}\mu(r_i) \\ \label{eq:inversion}
&= 1_{p|d_m}\lambda_{d_1,\ldots,d_k}\prod_{i\in I}\mu(d_i)f(d_i)\prod_{j\in J}\mu(d_j)g(d_j).
\end{align}
With this transformation our sum (\ref{eq:rand11}) becomes
\begin{align} \nonumber
&\sum_{\substack{u_1,\ldots,u_k}}\bigg(\prod_{i\in I}f^*(u_i)\bigg)\bigg(\prod_{j\in J}g^*(u_j)\bigg) \\ \nonumber
&\cdot\sideset{}{'}\sum_{s_{1,2},\ldots,s_{k-1,k}}\bigg(\prod_{\substack{1\leq i,j\leq k \\ i\neq j}}\mu(s_{i,j})\bigg) \bigg(\prod_{i\in I}\frac{\mu(a_i)\mu(b_i)}{f^*(a_i)f^*(b_i)}\bigg)\bigg(\prod_{j\in J}\frac{\mu(a_j)\mu(b_j)}{g^*(a_j)g^*(b_j)}\bigg)y_{a_1,\ldots,a_k}^{(J,p_1,m)}y_{b_1,\ldots,b_k}^{(J,p_2,m)},
\end{align}
where we have defined $a_i = u_i \prod_{j\neq i}s_{i,j}$ and $b_j = u_j\prod_{i\neq j}s_{i,j}.$ Because of our constraints on the $s_{i,j}$ variables, we can use multiplicativity to write this as
\begin{align} \nonumber
&\sum_{\substack{u_1,\ldots,u_k}}\bigg(\prod_{i\in I}\frac{\mu^2(u_i)}{f^*(u_i)}\bigg)\bigg(\prod_{j\in J}\frac{\mu^2(u_j)}{g^*(u_j)}\bigg) \\ \nonumber
&\sideset{}{'}\sum_{s_{1,2},\ldots,s_{k-1,k}}\bigg(\prod_{\substack{i,j\in I \\ i\neq j}}\frac{\mu(s_{i,j})}{f^*(s_{i,j})^2}\bigg)\bigg(\prod_{\substack{i,j\in J \\ i\neq j}}\frac{\mu(s_{i,j})}{g^*(s_{i,j})^2}\bigg)\\
&\bigg(\prod_{\substack{i\in I, j\in J}}\frac{\mu(s_{i,j})\mu(s_{j,i})}{f^*(s_{i,j})f^*(s_{j,i})g^*(s_{i,j})g^*(s_{j,i})}\bigg)y_{a_1,\ldots,a_k}^{(J,p_1,m)}y_{b_1,\ldots,b_k}^{(J,p_2,m)}.
\end{align}
We now wish to reduce to the case when $(a_m,p_1)=(b_m,p_2)=1.$ Indeed, we will show the contribution from the alternative cases is negligible. Of course, depending on whether of not $p_1=1$ and/or $p_2=1,$ some (or all) of the analysis which follows is not necessary, and this accounts for the convention we assert in the statement of the proposition. First let us note the estimate
\begin{equation}\label{eq:rand55}
\frac{1}{f^*(p)} = \frac{f(p)}{1-f(p)}=\frac{\frac{1}{p}+O(\frac{1}{p^2})}{1-\frac{1}{p}+O(\frac{1}{p^2})}=\frac{1}{p}+O\bigg(\frac{1}{p^2}\bigg),
\end{equation}
and similarly 
\begin{equation}
\frac{1}{g^{*}(p)}=\frac{1}{p^2}+O\bigg(\frac{1}{p^3}\bigg).
\end{equation}
Now, there are three cases to consider.

\begin{enumerate}
	\item Suppose that $p_1|a_m$ and $p_2|b_m.$ This occurs if and only if $p_1|u_m$ or $p_1|s_{m,j}$ for some $j\neq m,$ and $p_2|u_m$ or $p_2|s_{i,m}$ for some $i\neq m.$ Suppose, for example, that $p_1|u_m$ and $p_2|u_m.$ Moreover let us assume that $m\in J.$ Then one can bound the contribution as follows:
\begin{align} \nonumber
&\ll (y_{\text{max}}^{(J,p_1,m)})(y_{\text{max}}^{(J,p_2,m)})\bigg(\sum_{\substack{u\leq R \\ (u,W)=1 \\ p|u\Rightarrow p\equiv 3\,\,(\text{mod}\,\,4)}}\frac{\mu^2(u)}{f^*(u)}\bigg)^{|I|}\\ \nonumber
&\cdot \bigg(\sum_{\substack{u\leq R \\ (u,W)=1 \\ p|u\Rightarrow p\equiv 3\,\,(\text{mod}\,\,4)}}\frac{\mu^2(u)}{g^*(u)}\bigg)^{|J|-1} \bigg(\sum_{\substack{u_m\leq R \\ (u_m,W)=1 \\ p_1,p_2|u_m \\ p|u_m\Rightarrow p\equiv 3\,\,(\text{mod}\,\,4)}}\frac{\mu^2(u_m)}{g^*(u_m)}\bigg)\\ \nonumber
&\bigg(\sum_{s\geq 1}\frac{\mu^2(s)}{f^*(s)^2}\bigg)^{|I|^2-|I|}\bigg(\sum_{s\geq 1}\frac{\mu^2(s)}{g^*(s)^2}\bigg)^{|J|^2-|J|}\bigg(\sum_{s\geq 1}\frac{\mu^2(s)}{f^*(s)g^*(s)}\bigg)^{2|I||J|} \\ \nonumber
&\ll \frac{(y_{\text{max}}^{(J,p_1,m)})(y_{\text{max}}^{(J,p_2,m)})B^{|I|}}{(p_1p_2/(p_1,p_2))^2}.
\end{align}
It is easy to see that this bound also holds in any of the other possible cases in which $p_1|a_m$ and $p_2|b_m$ and $m\in J.$ If instead $m\notin J,$ then again it is easy to see the contribution is 
$$\ll \frac{(y_{\text{max}}^{(J,p_1,m)})(y_{\text{max}}^{(J,p_2,m)})B^{|I|}}{p_1p_2/(p_1,p_2)},$$
in all possible cases.

	\item Suppose that $p_1|a_m$ and $p_2\nmid b_m.$ If $m\in J,$ then similarly to the above, one can bound the contribution by 
	$$\ll \frac{(y_{\text{max}}^{(J,p_1,m)})(\tilde{y}_{\text{max}}^{(J,p_2,m)})B^{|I|}}{p_1^2}.$$
	If $m\notin J$ then likewise one obtains a contribution
	$$\ll \frac{(y_{\text{max}}^{(J,p_1,m)})(\tilde{y}_{\text{max}}^{(J,p_2,m)})B^{|I|}}{p_1}.$$
	
	\item Finally, the case $p_1\nmid a_m$ and $p_2|b_m$ proceeds as above, interchanging the roles of $p_1$ and $p_2.$
\end{enumerate}

Thus, we may now suppose that $(a_m,p_1)=(b_m,p_2)=1.$ From the support of the $y_{a_1,\ldots,a_k}^{(J,p_1,m)}$ we see there is no contribution from $(s_{i,j},W)\neq 1$ and so either $s_{i,j}=1$ or $s_{i,j}>D_0.$ The contribution from $s_{i,j}>D_0$ with $i,j\in I$ is 
\begin{align} \nonumber
&\ll (\tilde{y}_{\text{max}}^{(J,p_1,m)})(\tilde{y}_{\text{max}}^{(J,p_2,m)})\bigg(\sum_{\substack{u\leq R \\ (u,W)=1 \\ p|u\Rightarrow p\equiv 3\,\,(\text{mod}\,\,4)}}\frac{\mu^2(u)}{f^*(u)}\bigg)^{|I|}\bigg(\sum_{\substack{u\leq R \\ (u,W)=1 \\ p|u\Rightarrow p\equiv 3\,\,(\text{mod}\,\,4)}}\frac{\mu^2(u)}{g^*(u)}\bigg)^{|J|} \\ \nonumber
&\bigg(\sum_{\substack{s_{i,j}\geq D_0 }}\frac{\mu^2(s_{i,j})}{f^*(s_{i,j})^2}\bigg) \bigg(\sum_{s\geq 1}\frac{\mu^2(s)}{f^*(s)^2}\bigg)^{|I|^2-|I|-1}\bigg(\sum_{s\geq 1}\frac{\mu^2(s)}{g^*(s)^2}\bigg)^{|J|^2-|J|}\bigg(\sum_{s\geq 1}\frac{\mu^2(s)}{f^*(s)g^*(s)}\bigg)^{2|I||J|} \\ \nonumber
&\ll \frac{(\tilde{y}_{\text{max}}^{(J,p_1,m)})(\tilde{y}_{\text{max}}^{(J,p_2,m)})B^{|I|}}{D_0},
\end{align}
This contribution will be negligible. The cases $i,j\in J$ and $i\in I,j\in J$ can be treated the same way. This leaves us with a main term
\begin{equation*}
\sum_{\substack{u_1,\ldots,u_k \\ (u_m,p_1p_2)=1}}\frac{(y_{u_1,\ldots,u_k}^{(J,p_1,m)})(y_{u_1,\ldots,u_k}^{(J,p_2,m)})}{\prod_{i\in I}f^*(u_i)\prod_{j\in J}g^*(u_i)}
\end{equation*}
To finish, we claim the contribution from $u_j>1$ is small whenever $j\in J$. Indeed if $u_j>1$ then it must be divisible by a prime $p>D_0$ (with $p\equiv 3\,\,(\text{mod}\,\,4)$). So suppose $|J| \geq 1$ and let $j\in J.$ If $u_j>1$ we get a contribution 
\begin{align} \nonumber
&\ll (\tilde{y}_{\text{max}}^{(J,p_1,m)})(\tilde{y}_{\text{max}}^{(J,p_2,m)}) \bigg(\sum_{\substack{u\leq R \\ (u,W)=1}}\frac{\mu^2(u)}{f^*(u)}\bigg)^{|I|}\bigg(\sum_{\substack{u\leq R}}\frac{\mu^2(u)}{g^*(u)}\bigg)^{|J|-1}\sum_{\substack{p>D_0 \\ p\equiv 3\,\,(\text{mod}\,\,4)}}\bigg(\sum_{\substack{u_j<R \\ p|u_j}}\frac{\mu^2(u_j)}{g^*(u_j)}\bigg)\\ \nonumber
&\ll (\tilde{y}_{\text{max}}^{(J,p_1,m)})(\tilde{y}_{\text{max}}^{(J,p_2,m)})B^{|I|} \sum_{\substack{p>D_0 \\ p\equiv 3\,\,(\text{mod}\,\,4)}}\frac{1}{g^*(p)}\bigg(\sum_{\substack{u_j}}\frac{\mu^2(u_j)}{g^*(u_j)}\bigg) \\ \nonumber
&\ll \frac{(\tilde{y}_{\text{max}}^{(J,p_1,m)})(\tilde{y}_{\text{max}}^{(J,p_2,m)})B^{|I|}}{D_0},
\end{align}
which is small. Putting all of these facts together establishes Proposition~\ref{prop:2}.
\end{proof}

\subsection{Transformation for $y_{r_1,\ldots,r_k}^{(J,p,m)}$ and proof of Lemma~\ref{lemma:generalsieve} parts (i) and (ii)}

Define
\begin{equation*}
y_{r_1,\ldots,r_k}= \bigg(\prod_{i=1}^k\mu(r_i)\varphi(r_i)\bigg)\sum_{\substack{d_1,\ldots, d_k \\ r_i|d_i\,\,\forall i}}\frac{\lambda_{d_1,\ldots,d_k}}{\prod_{i=1}^{k}d_i}.
\end{equation*}
and let $y_{\text{max}}=\sup_{r_1,\ldots,r_k}|y_{r_1,\ldots,r_k}|.$ By the inversion formula (\ref{eq:inversion}), our definition of $\lambda_{d_1,\ldots,d_k}$ in Proposition~\ref{prop:sieve} is equivalent to taking
\begin{equation}
y_{r_1,\ldots,r_k}=F\bigg(\frac{\log{r_1}}{\log{R}},\ldots,\frac{\log{r_k}}{\log{R}}\bigg).
\end{equation}
We now wish to relate the more complicated diagonalisation vectors $y_{r_1,\ldots,r_k}^{(J,p,m)}$ to these simpler vectors. We first deal with the case when $J=\emptyset$, which is straightforward. By inspecting the proof of Proposition~\ref{prop:sieve}, it is clear that we only need to understand this case when $f(p)=1/p.$ 

\begin{lemma}[Relating $y_{r_1,\ldots,r_k}^{(J,p_1,m)}$ to $y_{r_1,\ldots,r_k}$ when $J=\emptyset$]\label{lemma:easycase}
Suppose $y_{r_1,\ldots,r_k}^{(J,p_1,m)}\neq 0, J=\emptyset$ and $m\in \{1,\ldots,k\}.$ Moreover suppose $f(p)=1/p$ for all primes $p$. Then the following hold.
\begin{enumerate}
	\item If $p_1|r_m$ then 
	$$y_{r_1,\ldots,r_k}^{(J,p_1,m)} = y_{r_1,\ldots,r_k}.$$
	\item If $p_1\nmid r_m$ then 
	$$y_{r_1,\ldots,r_k}^{(J,p_1,m)} = \frac{y_{r_1,\ldots,p_1r_m,\ldots,r_k}}{\mu(p_1)\varphi(p_1)}.$$
\end{enumerate}
\end{lemma}
\begin{proof}
If $f(p)=1/p$ then $f^{*}(p) = p-1 = \varphi(p).$ Hence we are assuming 
$$y_{r_1,\ldots,r_k}^{(J,p,m)} = \bigg(\prod_{i=1}^{k} \mu(r_i)\varphi(r_i)\bigg) \sum_{\substack{d_1,\ldots,d_k \\ r_i|d_i\,\,\forall i \\ p|d_m}}\frac{\lambda_{d_1,\ldots,d_k}}{\prod_{i=1}^{k}d_i}.$$
The result then follows by comparing with the definition of $y_{r_1\ldots,r_k}$ given above.
\end{proof}

Thus, proceeding, we may suppose that $|J| \geq 1.$ We may further suppose that $f(p)\neq 1/p$ as this case is of no interest to us (again, this is clear by inspecting the proof of Proposition~\ref{prop:sieve}). The following proposition gives the result in full.

\begin{prop}[Relating $y_{r_1,\ldots,r_k}^{(J,p_1,m)}$ to $y_{r_1,\ldots,r_k}$ when $|J| \geq 1$]\label{prop:newvecoldvec} Suppose $y_{r_1,\ldots,r_k}^{(J,p_1,m)}\neq 0, |J| \geq 1$ and $f(p)\neq 1/p$. Then the following hold:
	\begin{enumerate}[label=(\roman*)]
	\item If $m\in J$ and $(r_m,p)=1$ we have
	\begin{align*} \nonumber
	y_{r_1,\ldots,r_k}^{(J,p_1,m)} &= \frac{\mu(p_1)p_1g(p_1)}{\varphi(p_1)}\bigg(\prod_{j\in J}\frac{r_jg(r_j)g^{*}(r_j)}{g^{**}(r_j)}\bigg) \\
	&\cdot \sum_{\substack{e_1,\ldots,e_m',\ldots,e_k \\ r_i|e_i\,\,\forall i\neq m \\ r_m|e_m' \\ e_i=r_i\,\,\forall i\in I}}y_{e_1,\ldots,p_1e_m',\ldots,e_k}\bigg(\prod_{\substack{j\in J \\ j\neq m}}\frac{g^{**}(e_j)}{\varphi(e_j)}\bigg)\bigg(\frac{g^{**}(e_m')}{\varphi(e_m')}\bigg)+O\bigg(\frac{y_{\text{max}}B^{|J|}\log\log{R}}{D_0p_1^2}\bigg).
\end{align*}
	\item If $m\notin J$ and $(r_m,p)=1$ we have
	\begin{align*}
	y_{r_1,\ldots,r_k}^{(J,p_1,m)} &= \frac{\mu(p_1)p_1f(p_1)}{\varphi(p_1)}\bigg(\prod_{j\in J}\frac{r_jg(r_j)g^{*}(r_j)}{g^{**}(r_j)}\bigg)  \\
	&\cdot\sum_{\substack{e_1,\ldots,e_m',\ldots,e_k \\ r_i|e_i\,\,\forall i\neq m \\ e_i=r_i\,\,\forall i\in I\backslash\{m\} \\ e'_m = r_m}}y_{e_1,\ldots,p_1e_m',\ldots,e_k}\bigg(\prod_{j\in J}\frac{g^{**}(e_j)}{\varphi(e_j)}\bigg)+O\bigg(\frac{y_{\text{max}}B^{|J|}\log\log{R}}{D_0p_1}\bigg).
\end{align*}
	\item If $p_1|r_m$ we have 
	\begin{align*}
	y_{r_1,\ldots,r_k}^{(J,p_1,m)} &= \bigg(\prod_{j\in J}\frac{r_jg(r_j)g^{*}(r_j)}{g^{**}(r_j)}\bigg)  \\
	&\cdot\sum_{\substack{e_1,\ldots,e_k \\ r_i|e_i\,\,\forall i \\ e_i=r_i\,\,\forall i\in I}}y_{e_1,\ldots,e_k}\bigg(\prod_{j\in J}\frac{g^{**}(e_j)}{\varphi(e_j)}\bigg)+O\bigg(\frac{y_{\text{max}}B^{|J|}\log\log{R}}{D_0}\bigg).
\end{align*}	
	\end{enumerate}
Here $f^{**}$ and $g^{**}$ are defined by the convolutions
$$f^{**}= \iota \mu f *1,\,\,\,\,\,g^{**} = \iota \mu g*1,$$
where $\iota$ is the identity function, $\iota(p)=p$.
\end{prop}
%
%

\begin{proof}
We prove (i), with the rest proved in exactly the same way. Directly from the definition~(\ref{eq:diagvec}) we have
\begin{align}\label{eq:rand12}
\frac{y_{r_1,\ldots,r_k}^{(J,p_1,m)}}{\prod_{i\in I}\mu(r_i)f^*(r_i)\prod_{j\in J}\mu(r_j)g^*(r_j)} &=\sum_{\substack{d_1,\ldots, d_k \\ r_i|d_i\,\,\forall i \\ p_1|d_m}}\lambda_{d_1,\ldots,d_k}\prod_{i\in I}f(d_i)\prod_{j\in J}g(d_j). 
\end{align}
From the inversion formula~(\ref{eq:inversion}) and the definition of $y_{r_1,\ldots,r_k},$ we see that the right hand side of~(\ref{eq:rand12}) equals
\begin{align*}
&\sum_{\substack{d_1,\ldots, d_k \\ r_i|d_i\,\,\forall i \\ p_1|d_m}}\bigg(\prod_{i\in I}\mu(d_i)d_if(d_i)\bigg)\bigg(\prod_{j\in J}\mu(d_j)d_jg(d_j)\bigg) \sum_{\substack{e_1,\ldots,e_k \\ d_i|e_i\,\,\forall i }}\frac{y_{e_1,\ldots,e_k}}{\prod_{i=1}^{k}\varphi(e_i)}. 
\end{align*}
Swapping sums we obtain
\begin{align}\label{eq:rand13}
 \sum_{\substack{e_1,\ldots,e_k \\ r_i|e_i\,\,\forall i \\ p_1|e_m}}\frac{y_{e_1,\ldots,e_k}}{\prod_{i=1}^{k}\varphi(e_i)}\sum_{\substack{d_1,\ldots, d_k \\ r_i|d_i\,\,\forall i \\ d_i|e_i\,\,\forall i\\ p_1|d_m}}\bigg(\prod_{i\in I}\mu(d_i)d_if(d_i)\bigg)\bigg(\prod_{j\in J}\mu(d_j)d_jg(d_j)\bigg).
\end{align}
We can evaluate the inner sums using the convolution identities
\begin{equation*}
f^{**}(n)=\sum_{d|n}\mu(d)df(d),\,\,\,\,\,g^{**}(n)=\sum_{d|n}\mu(d)dg(d).
\end{equation*}
We note that
\begin{equation}\label{eq:rand51}
f^{**}(p) = 1-pf(p)=O\bigg(\frac{1}{p}\bigg),
\end{equation}
and similarly
\begin{equation}\label{eq:rand52}
g^{**}(p) = 1-pg(p)=1+O\bigg(\frac{1}{p}\bigg).
\end{equation}
With our assumption $f(p)\neq 1/p,$ we may suppose both of these functions are non-zero. Now, recall that we are assuming $m\in J.$ Using these identities transforms ~(\ref{eq:rand13}) into
\begin{align*}
\frac{\mu(p_1)p_1g(p_1)}{g^{**}(p_1)} &\bigg(\prod_{i\in I}\frac{\mu(r_i)r_if(r_i)}{f^{**}(r_i)}\bigg)\bigg(\prod_{j\in J}\frac{\mu(r_j)r_jg(r_j)}{g^{**}(r_j)}\bigg)  \\
&\cdot\sum_{\substack{e_1,\ldots,e_k \\ r_i|e_i\,\,\forall i \\ p_1|e_m}}y_{e_1,\ldots,e_k}\bigg(\prod_{i\in I}\frac{f^{**}(e_i)}{\varphi(e_i)}\bigg)\bigg(\prod_{j\in J}\frac{g^{**}(e_j)}{\varphi(e_j)}\bigg).
\end{align*}
Here we are using the fact $y_{e_1,\ldots,e_k}$ is supported on square-free integers $e=\prod_{i=1}^{k}e_i.$ Hence, from~(\ref{eq:rand12}), it follows that 
\begin{align*} \nonumber
y_{r_1,\ldots,r_k}^{(J,p_1,m)} &= \frac{\mu(p_1)p_1g(p_1)}{\varphi(p_1)}  \bigg(\prod_{i\in I}\frac{r_if(r_i)f^{*}(r_i)}{f^{**}(r_i)}\bigg)\bigg(\prod_{j\in J}\frac{r_jg(r_j)g^{*}(r_j)}{g^{**}(r_j)}\bigg) \\
&\cdot \sum_{\substack{e_1,\ldots,e_m',\ldots e_k \\ r_i|e_i\,\,\forall i\neq m \\ r_m|e_m'}}y_{e_1,\ldots,p_1e_m',\ldots e_k}\bigg(\prod_{i\in I}\frac{f^{**}(e_i)}{\varphi(e_i)}\bigg)\bigg(\prod_{\substack{j\in J \\ j\neq m}}\frac{g^{**}(e_j)}{\varphi(e_j)}\bigg)\frac{g^{**}(e_m')}{\varphi(e_m')}.
\end{align*}
Here we have substituted $e_m= p_1 e_m'$ and used the fact $(r_m,p)=1.$ Now, since
\begin{equation*}
\frac{pf(p)f^{*}(p)}{\varphi(p)}= \frac{p}{p-1}(1-f(p)) = \frac{p}{p-1}\bigg[1-\frac{1}{p}+O\bigg(\frac{1}{p^2}\bigg)\bigg]= 1+O\bigg(\frac{1}{p^2}\bigg),
\end{equation*}
it follows that
\begin{equation}\label{eq:estimate1}
\sup_{r_1,\ldots,r_k}\prod_{j\in J}\frac{r_j f(r_j)f^{*}(r_j)}{\varphi(r_j)} \ll 1.
\end{equation}
Similarly, since 
$$\frac{p g(p)g^{*}(p)}{\varphi(p)} = \frac{p}{p-1}(1-g(p)) = \frac{p}{p-1}\bigg[1+O\bigg(\frac{1}{p^2}\bigg)\bigg],$$
we have the bound
\begin{align}\label{eq:estimate2}
\sup_{r_1,\ldots,r_k}\prod_{j\in J}\frac{r_j g(r_j)g^{*}(r_j)}{\varphi(r_j)} &\ll \sup_{r_1,\ldots,r_k}\prod_{j\in J}\frac{r_j}{\varphi(r_j)} \ll \log\log{R}.
\end{align}
Here we have used the fact $r = \prod_{i=1}^{k}r_i \leq R$ and the standard estimate
$$\frac{n}{\varphi(n)} \ll \log\log{n}.$$
Now, if $i\neq m$ then either $e_i=r_i$ or $e_i>D_0 r_i$. Suppose $e_{i_0}>D_0r_{i_0}$ for some $i_0\in I.$ By first using multiplicativity of the sum over the $e_i$ variables, and then using estimates~(\ref{eq:estimate1}) and~(\ref{eq:estimate2}), we see that that this gives a contribution 
\begin{align}\nonumber
&\ll \frac{y_{\text{max}}\log\log{R}}{p_1^2} \bigg(\sum_{\substack{e_{i_0}> D_0 \\ (e_{i_0},W)=1}}\frac{\mu^2(e_{i_0})f^{**}(e_{i_0})}{\varphi(e_{i_0})}\bigg) \\ \nonumber
&\cdot\bigg(\sum_{\substack{u\leq R \\ (u,W)=1 \\ p|u\Rightarrow p\equiv 3\,\,(\text{mod}\,\,4)}}\frac{\mu^2(u)f^{**}(u)}{\varphi(u)}\bigg)^{|I|-1}\bigg(\sum_{\substack{u\leq R \\ (u,W)=1 \\ p|u\Rightarrow p\equiv 3\,\,(\text{mod}\,\,4)}}\frac{\mu^2(u)g^{**}(u)}{\varphi(u)}\bigg)^{|J|} \\ \nonumber
&\ll \frac{y_{\text{max}}B^{|J|}\log\log{R}}{D_0p_1^2},
\end{align}
where we have used the fact $f^{**}(p)/\varphi(p)=O(1/p^2)$ and $g^{**}(p)/\varphi(p)=1/p+O(1/p^2),$ which follows from (\ref{eq:rand51}) and (\ref{eq:rand52}) respectively. This is small. Thus 
\begin{align*}
y_{r_1,\ldots,r_k}^{(J,p_1,m)} &= \frac{\mu(p_1)p_1g(p_1)}{\varphi(p_1)}\bigg(\prod_{i\in I}\frac{r_if(r_i)f^{*}(r_i)}{\varphi(r_i)}\bigg)\bigg(\prod_{j\in J}\frac{r_jg(r_j)g^{*}(r_j)}{g^{**}(r_j)}\bigg)\\
&\cdot \sum_{\substack{e_1,\ldots,e_m',\ldots,e_k \\ r_i|e_i\,\,\forall i\neq m \\ r_m|e_m' \\ e_i=r_i\,\,\forall i\in I}}y_{e_1,\ldots,p_1e_m',\ldots e_k}\bigg(\prod_{\substack{j\in J \\ j\neq m}}\frac{g^{**}(e_j)}{\varphi(e_j)}\bigg)\bigg(\frac{g^{**}(e_m')}{\varphi(e_m')}\bigg) +O\bigg(\frac{y_{\text{max}}B^{|J|}\log\log{R}}{D_0p_1^2}\bigg).
\end{align*}
From the support restrictions on $y_{r_1,\ldots,r_k}^{(J,p_1,m)}$ we necessarily have $(r_i,W)=1.$ From the above, we see the first product may be replaced by $1+O(D_0^{-1}).$ This incurs an acceptable error 
\begin{align*}
\ll \frac{y_{\text{max}}\log\log{R}}{D_0p_1^2} \bigg(\sum_{\substack{u\leq R \\ (u,W)=1 \\ p|u\Rightarrow p \equiv 3\,\,(\text{mod}\,\,4)}}\frac{\mu^2(u) g^{**}(u)}{\varphi(u)}\bigg)^{|J|} \ll \frac{y_{\text{max}}B^{|J|}\log\log{R}}{D_0p_1^2},
\end{align*}
which gives the result stated. 
\end{proof}


We record the following useful corollary.

\begin{corollary}\label{corollary:errorterms}
With notation as in the statement of Proposition~\ref{prop:2}, the following estimates hold.
\begin{enumerate}
	\item If $m\in J$ then
	$$\tilde{y}_{\text{max}}^{(J,p_1,m)} \ll \frac{y_{\text{max}}B^{|J|}\log\log{R}}{p_1^2}.$$
	\item If $m\notin J$ then
	$$\tilde{y}_{\text{max}}^{(J,p_1,m)} \ll \frac{y_{\text{max}}B^{|J|}\log\log{R}}{p_1}.$$	
	\item We have
	$$y_{\text{max}}^{(J,p,m)} \ll y_{\text{max}}B^{|J|}\log\log{R}.$$
\end{enumerate}
\end{corollary}
\begin{proof}
This follows easily from  Lemma~\ref{lemma:easycase} and Proposition~\ref{prop:2}. (We note that in the special case $J=\emptyset,$ in items (2) and (3) we could drop the extra $\log\log{R}$ factor if required.)
\end{proof}

We are now in a position to prove the first two parts of Lemma~\ref{lemma:generalsieve}. 

\begin{proof}[Proof of Lemma~\ref{lemma:generalsieve} parts (i) and (ii)]
We prove part (i), in the case $m\in J.$ The rest of the argument proceeds along the same lines. From Proposition~\ref{prop:newvecoldvec} we have
\begin{equation*}
S_{J,p_1,p_2,m} = \sum_{\substack{u_1,\ldots,u_k \\ (u_m,p_1p_2)=1 \\ u_j=1\,\,\forall j\in J}}\frac{(y_{u_1,\ldots,u_k}^{(J,p_1,m)})(y_{u_1,\ldots,u_k}^{(J,p_2,m)})}{\prod_{i\in I}f^*(u_i)\prod_{j\in J}g^*(u_j)}+O\bigg(\frac{y_{\text{max}}^2B^{k+|J|}(\log\log{R})^2}{(p_1p_2/(p_1,p_2))^2}\bigg).
\end{equation*}
Here we have used Corollary~\ref{corollary:errorterms} to control the various error terms in the statement of the proposition. We have also used the fact $|I|+|J|=k.$ It follows that
\begin{align*}
S_{J,p_1,p_2,m} &\ll (\tilde{y}_{\text{max}}^{(J,p_1,m)})(\tilde{y}_{\text{max}}^{(J,p_2,m)})\bigg(\sum_{\substack{u\leq R \\ (u,W)=1 \\ p|u\Rightarrow p\equiv 3\,\,(\text{mod}\,\,4)}} \frac{\mu^2(u)}{f^*(u)}\bigg)^{|I|} +\frac{y_{\text{max}}^2B^{k+|J|}(\log\log{R})^2}{(p_1p_2/(p_1,p_2))^2} \\
&\ll \frac{y_{\text{max}}^2B^{k+|J|}(\log\log{R})^2}{(p_1p_2/(p_1,p_2))^2},
\end{align*}
again using Corollary~\ref{corollary:errorterms}. This gives the result, as required.
\end{proof}

From now on we are only interested in the sums $S_{J} = S_{J,1,1,m},$ and in particular the cases \hbox{$|J|\in \{0,1, 2\}$.} For ease of notation we let $y_{r_1,\ldots,r_k}^{(J)} = y_{r_1,\ldots,r_k}^{(J,1,m)}$ and $y_{\text{max}}^{(J)} = \sup_{r_1,\ldots,r_k} |y_{r_1,\ldots,r_k}^{(J)}|.$ For future reference we note the bound
\begin{equation}\label{eq:ymaxJ}
y_{\text{max}}^{(J)} \ll y_{\text{max}} B^{|J|}\log\log{R},
\end{equation}
proved above.

\subsection{Relating vectors to functionals and proof of Lemma~\ref{lemma:generalsieve} part (iii)} We first prove the following corollary that follows from Proposition~\ref{prop:newvecoldvec}.

\begin{corollary}[Relating $y_{r_1,\ldots,r_k}^{(J)}$ to integral operators]\label{cor:2}
Let $y_{r_1,\ldots,r_k}$ be defined in terms of a fixed, smooth function $F$, supported on ${R}_k=\{\vec{x}\in[0,1]^k:\sum_{i=1}^{k}x_i\leq 1\}$, by
\begin{equation*}
y_{r_1,\ldots,r_k}=F\bigg(\frac{\log{r_1}}{\log{R}},\ldots,\frac{\log{r_k}}{\log{R}}\bigg).
\end{equation*}
Let
\begin{equation*}
F_{\text{max}} = \sup_{(t_1,\ldots,t_k)\in[0,1]^k}|F(t_1,\ldots,t_k)|+\sum_{i=1}^{k}|\frac{\partial F}{\partial t_i}(t_1,\ldots,t_k)|.
\end{equation*}
Define the integral operators
\begin{align*}
I_{r_1,\ldots,r_k; m}(F) &= \int_0^1F\bigg(\frac{\log{r_1}}{\log{R}},\ldots,x_m,\ldots,\frac{\log{r_k}}{\log{R}}\bigg)\frac{\mathrm{d}x_m}{\sqrt{x_m}}, \\
I_{r_1,\ldots,r_k; m,l}(F) &= \int_0^1\bigg(\int_0^1F\bigg(\frac{\log{r_1}}{\log{R}},\ldots,x_l,\ldots,x_m,\ldots,\frac{\log{r_k}}{\log{R}}\bigg)\frac{\mathrm{d}x_m}{\sqrt{x_m}}\bigg)\frac{\mathrm{d}x_l}{\sqrt{x_l}}.
\end{align*}
Then
\begin{enumerate}[label=(\roman*)]
	\item if $J=\{m\}$ we have
	$$y_{r_1,\ldots,r_k}^{(J)}\bigg|_{r_j=1\,\,\forall j\in J} = B\bigg(\prod_{i\in I}\frac{\varphi(r_i)}{r_i}\bigg)I_{r_1,\ldots,r_k; m}(F)+ 	O\bigg(\frac{F_{\text{max}}B}{D_0}\bigg).$$
	\item if $J=\{m,l\}$ we have
	$$y_{r_1,\ldots,r_k}^{(J)}\bigg|_{r_j=1\,\,\forall j\in J} = B^2\bigg(\prod_{i\in I}\frac{\varphi(r_i)}{r_i}\bigg)^2I_{r_1,\ldots,r_k; m,l}(F)+ O\bigg(\frac{F_{\text{max}}B^2}{D_0}\bigg).$$
\end{enumerate}
\end{corollary}

\begin{proof}
First suppose $J=\{m\}.$ From Proposition~\ref{prop:newvecoldvec} we have
\begin{equation}\label{eq:rand56}
y_{r_1,\ldots,r_k}^{(J)}\big|_{r_j=1\,\,\forall j\in J} = \sum_{\substack{e_m}}\frac{g^{**}(e_m)}{\varphi(e_m)}F\bigg(\frac{\log{r_1}}{\log{R}},\ldots,\frac{\log{e_m}}{\log{R}},\ldots\frac{\log{r_k}}{\log{R}}\bigg)+O\bigg(\frac{F_{\text{max}}B}{D_0}\bigg)
\end{equation}
Consider the function
\begin{equation*}\label{eq:rand14}
\gamma_1(p)=
\begin{cases}
\frac{p}{1+\frac{\varphi(p)}{g^{**}(p)}}\,\,&\text{if $p\nmid W\prod_{i\in I}r_i$ and $p\equiv 3\,\,(\text{mod}\,\,4),$} \\
0\,\,&\text{otherwise.}
\end{cases}
\end{equation*}
With this choice of $\gamma_1(p)$ we have
\begin{equation*}
\frac{\gamma_1(p)}{p-\gamma_1(p)} = \frac{g^{**}(p)}{\varphi(p)}
\end{equation*}
and one can easily check $\gamma_1(p) = 1+O(1/p).$ By an argument identical to the proof of Lemma~\ref{key}, we can evaluate the sum in (\ref{eq:rand56}) as
\begin{align*} 
y_{r_1,\ldots,r_k}^{(J)}\big|_{r_j=1\,\,\forall j\in J} &= B\bigg(\prod_{i\in I}\frac{\varphi(r_i)}{r_i}\bigg)\int_0^1 F\bigg(\frac{\log{r_1}}{\log{R}},\ldots,t_m,\ldots,\frac{\log{r_k}}{\log{R}}\bigg)\frac{\mathrm{d}t_m}{\sqrt{t_m}}+O\bigg(\frac{F_{\text{max}}B}{D_0}\bigg),
\end{align*}
which proves (i). If $J=\{m,l\}$ then, again using Proposition~\ref{prop:newvecoldvec}, we find
\begin{align} \nonumber
y_{r_1,\ldots,r_k}^{(J)}\big|_{r_j=1\,\,\forall j\in J} &= \sum_{\substack{e_{m},e_{l}}}\bigg(\prod_{j=m,l}\frac{g^{**}(e_j)}{\varphi(e_j)}\bigg)F\bigg(\frac{\log{r_1}}{\log{R}},\ldots,\frac{\log{e_{m}}}{\log{R}},\ldots,\frac{\log{e_{l}}}{\log{R}},\ldots,\frac{\log{r_{k}}}{\log{R}}\bigg)\\ \label{eq:rand40}
&+O\bigg(\frac{F_{\text{max}}B^2}{D_0}\bigg).
\end{align}
Take the sum over $e_l$ first. Consider the function 
\begin{equation*}\label{eq:rand15}
\gamma_2(p)=
\begin{cases}
\frac{p}{1+\frac{\varphi(p)}{g^{**}(p)}}\,\,&\text{if $p\nmid W\prod_{i\in I}r_i e_m$ and $p\equiv 3\,\,(\text{mod}\,\,4),$} \\
0\,\,&\text{otherwise.}
\end{cases}
\end{equation*}
Reasoning as above, the sum on the right hand side of (\ref{eq:rand40}) becomes
\begin{align*}
 B\bigg(\prod_{i\in I}\frac{\varphi(r_i)}{r_i}\bigg)&\sum_{e_{m}}\frac{g^{**}(e_{m})}{e_{m}}\int_0^1 F\bigg(\frac{\log{r_1}}{\log{R}},\ldots,\frac{\log{e_{m}}}{\log{R}},\ldots,t_{l},\ldots,\frac{\log{r_k}}{\log{R}}\bigg)\frac{\mathrm{d}t_{l}}{\sqrt{t_{l}}}\\
 &+O\bigg(\frac{F_{\text{max}}B^2}{D_0}\bigg).
\end{align*}
We can evaluate this sum in much the same way, this time using the function 
\begin{equation*}
\gamma_3(p)=
\begin{cases}
\frac{p}{1+\frac{p}{g^{**}(p)}}\,\,&\text{if $p\nmid W\prod_{i\in I}r_i$ and $p\equiv 3\,\,(\text{mod}\,\,4),$} \\
0\,\,&\text{otherwise,}
\end{cases}
\end{equation*}
to get the stated result.
\end{proof}

If we define the (identity) operator 
\begin{equation*}
I_{r_1,\ldots,r_k}(F)=F\bigg(\frac{\log{r_1}}{\log{R}},\ldots,\frac{\log{r_k}}{\log{R}}\bigg),
\end{equation*}
then the results of Corollary~\ref{cor:2} can be concisely written as 
\begin{equation*}
y_{r_1,\ldots,r_k}^{(J)}\bigg|_{r_j=1\,\,\forall j\in J} = B^{|J|}\bigg(\prod_{i\in I}\frac{\varphi(r_i)}{r_i}\bigg)^{|J|}I_{r_1,\ldots,r_k;J}(F)+ O\bigg(\frac{F_{\text{max}}B^{|J|}}{D_0}\bigg)
\end{equation*}
for $|J|\in \{0,1,2\}.$ Here we have used $I_{r_1,\ldots,r_k;J}(F)$ to denote $I_{r_1,\ldots,r_k;j:j\in J}(F).$ We are now in a position to prove the remaining part of Lemma~\ref{lemma:generalsieve}.

\begin{proof}[Proof of Lemma~\ref{lemma:generalsieve} part (iii)] We can write the operators in the statement of Lemma~\ref{lemma:generalsieve} as follows:
\begin{align*} 
L(F)&= \int_0^1\ldots \int_0^1 \bigg[I(F)\bigg]^2  \prod_{\substack{i=1}}^{k}\frac{\mathrm{d}x_i}{\sqrt{x_i}}, \\
L_{m}(F)&= \int_0^1\ldots \int_0^1 \bigg[I_{m}(F)\bigg]^2  \prod_{\substack{i=1 \\ i\neq m}}^{k}\frac{\mathrm{d}x_i}{\sqrt{x_i}}, \\ 
L_{m,l}(F)&= \int_0^1\ldots \int_0^1 \bigg[I_{m,l}(F)\bigg]^2  \prod_{\substack{i=1 \\ i\neq m,l}}^{k}\frac{\mathrm{d}x_i}{\sqrt{x_i}},
\end{align*}
where we have defined
\begin{align*}
I(F) &= F(x_1,\ldots,x_k), \\
I_m(F) &= \int_0^1F(x_1,\ldots,x_k)\frac{\mathrm{d}x_m}{\sqrt{x_m}}, \\
I_{m,l}(F) &= \int_0^1 \bigg(\int_0^1F(x_1,\ldots,x_k)\frac{\mathrm{d}x_m}{\sqrt{x_m}}\bigg)\frac{\mathrm{d}x_l}{\sqrt{x_l}}.
\end{align*}
From Proposition \ref{prop:2} we see that
\begin{equation}\label{eq:rand53}
S_J = \sum_{\substack{u_1,\ldots,u_k \\ u_j=1\,\,\forall j\in J}}\frac{(y_{u_1,\ldots,u_k}^{(J)})^2}{\prod_{i\in I}f^*(u_i)}+O\bigg(\frac{(y_{\text{max}}^{(J)})^2B^{|I|}}{D_0}\bigg).
\end{equation}
From Corollary~\ref{cor:2}, for $|J|\in\{0,1,2\},$ we have 
\begin{equation*}
(y_{r_1,\ldots,r_k}^{(J)})^2\bigg|_{r_j=1\,\,\forall j\in J} = B^{2|J|}\bigg(\prod_{i\in I}\frac{\varphi(r_i)}{r_i}\bigg)^{2|J|}\bigg[I_{r_1,\ldots,r_k;J}(F)\bigg]^2+ O\bigg(\frac{F_{\text{max}}^2B^{2|J|}}{D_0}\bigg).
\end{equation*}
Substituting this into (\ref{eq:rand53}), and using (\ref{eq:ymaxJ}), yields
\begin{align*} \nonumber
S_J&= B^{2|J|} \sum_{\substack{u_1,\ldots,u_k \\ u_j=1\,\,\forall j\in J \\(u_i,u_j)=1\,\,\forall i\neq j \\ (u_i,W)=1\,\,\forall i \\ p|u_i\Rightarrow p\equiv 3\,\,(\text{mod}\,\,4)}}\prod_{i\in I}\frac{\mu^2(u_i)\varphi(u_i)^{2|J|}}{f^*(u_i)u_i^{2|J|}}\bigg[I_{u_1,\ldots,u_k;J}(F)\bigg]^2 \\
&+O\bigg(\frac{F_{\text{max}}^2B^{2|J|}}{D_0}\sum_{\substack{u_1,\ldots,u_k \\ u_j=1\,\,\forall j\in J \\ (u_i,W)=1\,\,\forall i \\ p|u_i\Rightarrow p\equiv 3\,\,(\text{mod}\,\,4)}}\prod_{i\in I}\frac{\mu^2(u_i)}{f^*(u_i)}\bigg)+ O\bigg(\frac{F_{\text{max}}^2B^{k+|J|}(\log\log{R})^2}{D_0}\bigg).
\end{align*}
The first error contributes
\begin{equation*}
\ll \frac{F_{\text{max}}^2B^{2|J|}}{D_0}\bigg(\sum_{\substack{u\leq R \\ (u,W)=1 \\ p|u\Rightarrow p\equiv 3\,\,(\text{mod}\,\,4)}}\frac{\mu^2(u_i)}{f^*(u_i)}\bigg)^{|I|} \ll \frac{F_{\text{max}}^2B^{k+|J|}}{D_0}.
\end{equation*}
For the main term, if $(u_i,u_j)\neq 1$ then they must be divisible by a prime $q>D_0$ with $q\equiv 3\,\,(\text{mod}\,\,4)$. In this case we get a contribution 
\begin{align*} \nonumber
&\ll F_{\text{max}}^2B^{2|J|} \sum_{\substack{q>D_0 \\ q\equiv 3\,\,(\text{mod}\,\,4)}}\sum_{\substack{u_1,\ldots,u_k \\ u_j=1\,\,\forall j\in J \\ (u_i,W)=1\,\,\forall i \\ p|u_i \Rightarrow p\equiv 3\,\,(\text{mod}\,\,4) \\ q|u_i,u_j}}\prod_{i\in I}\frac{\mu^2(u_i)\varphi(u_i)^{2|J|}}{f^*(u_i)u_i^{2|J|}} \\ 
&\ll F_{\text{max}}^2B^{2|J|} \bigg(\sum_{\substack{u\leq R \\ (u,W)=1 \\ p|u\Rightarrow p\equiv 3\,\,(\text{mod}\,\,4)}}\frac{\mu^2(u)\varphi(u)^{2|J|}}{f^*(u)u^{2|J|}}\bigg)^{|I|} \sum_{q>D_0}\frac{\varphi(q)^{4|J|}}{f^*(q)^2q^{4|J|}} \ll \frac{F_{\text{max}}^2B^{k+|J|}}{D_0}.
\end{align*}
Thus this constraint can be removed at the cost of a negligible error and we are left with
\begin{equation*}
S_J=B^{2|J|} \sum_{\substack{u_1,\ldots,u_k \\ u_j=1\,\,\forall j\in J \\ (u_i,W)=1\,\,\forall i \\ p|u_i\Rightarrow p\equiv 3\,\,(\text{mod}\,\,4)}}\prod_{i\in I}\frac{\mu^2(u_i)\varphi(u_i)^{2|J|}}{f^*(u_i)u_i^{2|J|}}\bigg[I_{u_1,\ldots,u_k;J}(F)\bigg]^2+O\bigg(\frac{F_{\text{max}}^2B^{k+|J|}(\log\log{R})^2}{D_0}\bigg).
\end{equation*}
Now, since
$$\frac{\varphi(p)^{2|J|}}{f^*(p)p^{2|J|}} = \frac{1}{p}+O\bigg(\frac{1}{p^2}\bigg),$$
we can evaluate this multidimensional sum by applying Lemma~\ref{key} $|I|=k-|J|$ times. We obtain 
\begin{equation*}
S_J=B^{k+|J|}L_J(F)+O\bigg(\frac{F_{\text{max}}^2B^{k+|J|}(\log\log{R})^2}{D_0}\bigg).
\end{equation*}
This completes the proof of Lemma \ref{lemma:generalsieve}.
\end{proof}
%


\appendix
\appendixpage

\section{Estimates for $r(n)$ and $r(n)r(n+h)$.}
We sketch proofs for Lemmas \ref{lemma:1}, \ref{lemma:2} and \ref{lemma:3}. The first follows immediately from the following two results, due to Tolev~\hbox{\cite[Theorem]{Tolev}} and Plaksin  (cf. discussion just before~\hbox{\cite[Lemma 4]{Plak}}) respectively.

\begin{lemma}\label{theo:Tol}
We have 
\begin{equation}\label{eq:app1}
\sum_{\substack{n\leq N\\ n\equiv \Delta\,\,(\text{mod}\,\,Q)}}r(n) = \frac{\eta(Q,\Delta)}{Q^2}\pi N+O((Q^{\frac{1}{2}}+N^{\frac{1}{3}})(\Delta,Q)^{\frac{1}{2}}\tau^4(Q)\log^4{N}),
\end{equation}
where
\begin{equation*}
\eta(Q,\Delta) = \#\{1\leq \alpha,\beta \leq Q: \alpha^2+\beta^2 \equiv \Delta\,\,(\text{mod}\,\,Q)\}.
\end{equation*}
\end{lemma}

\begin{lemma}\label{theo:Plak}
We have
\begin{equation}\label{eq:app2}
\sum_{\substack{n\leq N\\ n\equiv \Delta\,\,(\text{mod}\,\,Q)}}r(n) = \frac{A(Q,\Delta)}{Q}\pi N+O(P(N;Q,\Delta)),
\end{equation}
where
\begin{equation}\label{eq:app3}
A(Q,\Delta) = \bigg[1+\sum_{\substack{k \\ 4|2^k|(Q,4\Delta)}}\chi\bigg(\frac{4\Delta}{2^k}\bigg)\bigg]\sum_{r|Q}\frac{\chi(r)}{r}c_r(\Delta),
\end{equation}
and $P(N;Q,\Delta)$ satisfies
\begin{equation*}
\int_1^N P^2(y;Q,\Delta)\mathrm{d}y \ll_{\epsilon} (QN)^{\epsilon}(N^{\frac{3}{2}}+QN).
\end{equation*}
Here $c_r(\Delta)=\sum_{(a,r)=1}e^{\frac{2\pi i a \Delta}{r}}$ denotes the Ramanujan sum.
\end{lemma}
\begin{proof}[Proof of Lemma~\ref{lemma:1}]
Equating (\ref{eq:app1}) and (\ref{eq:app2}), dividing through by $\pi N,$ and letting $N\rightarrow \infty$ we see that
\begin{equation*}
\frac{\eta(Q,\Delta)}{Q^2}=\frac{A(Q,\Delta)}{Q}
\end{equation*}
holds for any fixed $Q,\Delta.$ Thus we may write 
\begin{equation}\label{eq:tolplak}
\sum_{\substack{n\leq N\\ n\equiv \Delta\,\,(\text{mod}\,\,Q)}}r(n) = \frac{A(Q,\Delta)}{Q}+O((Q^{\frac{1}{2}}+N^{\frac{1}{3}})(\Delta,Q)^{\frac{1}{2}}\tau^4(Q)\log^4{N}).
\end{equation}
Let $Q=4qd$ and $\Delta\,\,(\text{mod}\,\,Q)$ be the solution to the congruence system $\Delta \equiv a\,\,(\text{mod}\,\,q),$ $\Delta \equiv 1\,\,(\text{mod}\,\,4),$ and $\Delta \equiv 0\,\,(\text{mod}\,\,d),$
where $(a,q)=(d,q)=1$ and $q,d$ are square-free and odd. Then by multiplicativity of the Ramanujan sum and our assumptions about $\Delta,q$ and $d$ from (\ref{eq:app3}) we obtain
\begin{align}
A(Q,\Delta) = 2 \sum_{r|q}\frac{\chi(r)\mu(r)}{r}\sum_{r|d}\frac{\chi(r)\varphi(r)}{r} = 2g_1(q)g_2(d),
\end{align}
and so we obtain the result of Lemma \ref{lemma:1} on dividing through by $Q=4qd$. 
\end{proof}
For Lemma \ref{lemma:2} we have the following result.

\begin{lemma}\label{theo:plak}
Suppose that $(a,q)=(a+h,q)=(d_1,q)=(d_2,q)=(d_1,d_2)=1,4|h$ where $d_1,d_2,q$ are square-free and odd, of size $\ll N^{O(1)}$. Then for $0<h<N^{\frac{3}{4}}$ we have 
\begin{equation}\label{eq:app4}
\sum_{\substack{n\leq N\\ n\equiv a\,\,(\text{mod}\,\,q) \\ n\equiv 1\,\,(\text{mod}\,\,4) \\ d_1|n, d_2|n+h}}r(n)r(n+h) = \frac{g_1(q)^2}{q}\Gamma(h,d_1,d_2,q)\pi^2N+R_2(N,d_1,d_2,q),
\end{equation}
where
$$R_2(N,d_1,d_2,q)\ll_{\epsilon}q^{\frac{1}{2}}d_1d_2N^{\frac{3}{4}+\epsilon}+d_1^{\frac{1}{2}}d_2^{\frac{1}{2}}N^{\frac{5}{6}+\epsilon},$$
where
$$\Gamma(h,d_1,d_2,q) = \frac{g_2(d_1)g_2(d_2)}{d_1d_2}\sum_{\substack{(r,2q)=1}}\frac{c_r(h)}{r^2}\frac{(d_1,r)(d_2,r)}{\Psi(d_1,r)\Psi(d_2,r)}\chi[(d_1^2,r)]\chi[(d_2^2,r)]
$$
and $g_1$ is the multiplicative function defined on primes by $g_1(p) =1- \frac{\chi(p)}{p}$ and $\Psi(d_1,r)=g_2((d_1,r/(r,d_1))).$
\end{lemma}

\begin{proof}[Proof of Lemma~\ref{lemma:2}]
Under the additional assumption that $p|h \Rightarrow p|2q,$ then we see for $r$ considered in the sum $c_r(h)=\mu(r)$. But now restricting to square-free $r$, since $d_1,d_2$ are square-free $\Psi(d_1,r)=\Psi(d_2,r)=1.$ Thus in this case
\begin{equation}
\Gamma(h,d_1,d_2,q) = \frac{g_2(d_1)g_2(d_2)}{d_1d_2}\sum_{\substack{(r,2q)=1}}\frac{\mu(r)(d_1,r)(d_2,r)}{r^2}\chi[(d_1^2,r)]\chi[(d_2^2,r)].
\end{equation}
This is the form stated in Lemma~\ref{lemma:2}.
\end{proof}

Now we briefly outline the proof of Lemma~\ref{theo:plak}. In~\cite[Lemma 4]{Plak} a similar sum to (\ref{eq:app4}) is considered, this time under the hypotheses $p|d_1,d_2\Rightarrow p\equiv 3\,\,(\text{mod}\,\,4),$ $q=1$ and the congruence $n\equiv 1\,\,(\text{mod}\,\,4)$ is omitted. The proof of Lemma \ref{theo:plak} is similar to the proof found there, with few minor changes. The key point is to note that, using the convolution identity $r=4(\chi*1)$ and complete multiplicativity of $\chi,$ for $n\equiv 1\,\,(\text{mod}\,\,4)$ we can write
\begin{align} \nonumber
\frac{r(n)}{4} &= \sum_{m|n}\chi(m) = \sum_{\substack{m|n \\ m\leq \sqrt{N}}}\chi(m)+\sum_{\substack{m|n \\ m>\sqrt{N}}}\chi(m) =\sum_{\substack{m|n \\ m\leq \sqrt{N}}}\chi(m)+\sum_{\substack{l|n \\ l <\frac{n}{\sqrt{N}}}}\chi(l) \\
&=2\sum_{\substack{m|n \\ m\leq \sqrt{N}}}\chi(m)-\sum_{\substack{m|n \\ \frac{n}{\sqrt{N}}\leq m \leq \sqrt{N}}}\chi(m). 
\end{align}
Using this we may expand out $r(n)$ in the sum (\ref{eq:app4}). We obtain (after swapping the order of summation)
\begin{align*}
\sum_{\substack{n\leq N\\ n\equiv a\,\,(\text{mod}\,\,q) \\ n\equiv 1\,\,(\text{mod}\,\,4) \\ d_1|n, d_2|n+h}}r(n)r(n+h)  = 4 \sum_{m\leq \sqrt{N}}\chi(m)\bigg[2\sum_{\substack{n\leq N\\ n\equiv a\,\,(\text{mod}\,\,q) \\ n\equiv 1\,\,(\text{mod}\,\,4) \\ d_1|n, d_2|n+h \\ m|n}}r(n+h)-\sum_{\substack{n\leq m\sqrt{N}\\ n\equiv a\,\,(\text{mod}\,\,q) \\ n\equiv 1\,\,(\text{mod}\,\,4) \\ d_1|n, d_2|n+h \\ m|n}}r(n+h)\bigg].
\end{align*}
These congruences have a solution if and only if $(m,q)=1$ and $(m,d_2)|h$. In this case we can use the Chinese Remainder Theorem and write
\begin{align}\nonumber
\sum_{\substack{n\leq N\\ n\equiv a\,\,(\text{mod}\,\,q) \\ n\equiv 1\,\,(\text{mod}\,\,4) \\ d_1|n, d_2|n+h}}r(n)r(n+h)  = 4\sum_{\substack{m\leq \sqrt{N}\\ (m,q)=1 \\ (m,d_2)|h}}\chi(m) &\bigg[2\sum_{\substack{n\leq N \\ n\equiv \Delta\,\,(\text{mod}\,\,Q)}}r(n)-\sum_{\substack{n\leq m\sqrt{N} \\ n\equiv \Delta\,\,(\text{mod}\,\,Q)}}r(n)\bigg] \\ \label{eq:app5}
&+O_{\epsilon}((N^{\frac{1}{2}}+h)N^{\epsilon}),
\end{align}
where $Q=4q[m,d_1,d_2]$ and $\Delta\,\,(\text{mod}\,\,Q)$ satisfies the congruence system $\Delta \equiv a+h\,\,(\text{mod}\,\,q), \Delta \equiv 1\,\,(\text{mod}\,\,4), \Delta \equiv h\,\,(\text{mod}\,\,m), \Delta \equiv h\,\,(\text{mod}\,\,d_1)$ and $\Delta \equiv 0\,\,(\text{mod}\,\,d_2)$. The error term arises by estimating the intervals of length $h$ left over: taking absolute values, using the divisor bound $r(n) \ll n^{\epsilon}$, and noting that in this regime $h\ll N^{3/4},$ we have to estimate sums of type 
$$N^{\epsilon} \sum_{m\leq \sqrt{N}} \sum_{\substack{X < n \leq X+h \\ n\equiv h\,\,(\text{mod}\,\,m)}}1.$$
Applying the standard estimate $h/m+O(1)$ to the inner sum and then carrying out the summation over $m$ yields the desired estimate (after redefining our choice of $\epsilon$).
%

Now the inner sums in (\ref{eq:app5}) consist of estimating $r(n)$ in arithmetic progressions. It proves convenient to proceed using formula (\ref{eq:tolplak}). We obtain
\begin{equation}\label{eq:app6}
4\pi\sum_{\substack{m\leq \sqrt{N} \\ (m,q)=1 \\ (m,d_2)|h}}\chi(m)\bigg[\frac{2N-m\sqrt{N}}{Q}\bigg]A(Q,\Delta)+R_2(N;d_1,d_2,q)+O_{\epsilon}((N^{\frac{1}{2}}+h)N^{\epsilon}),
\end{equation}
where $R_2$ is an error term. Using the fact $(\Delta,Q)\leq d_1d_2(m,h)$ and $Q\ll  mqd_1d_2$ one can show $R_2$ contributes
\begin{align*} 
R_2(N;d_1,d_2,q) &\ll q^{\frac{1}{2}}d_1d_2N^{\frac{3}{4}+\epsilon}+d_1^{\frac{1}{2}}d_2^{\frac{1}{2}}N^{\frac{5}{6}+\epsilon},
\end{align*}
and so this error term dominates. Now one can estimate the main term in (\ref{eq:app6}) using the same techniques found in~\hbox{\cite[Lemma 4]{Plak}}.
\\ \\
For Lemma \ref{lemma:3} we have the following result.

\begin{lemma}\label{theo:Perron}
Let $(a,q)=(d,q)=1$ where $d,q$ are square-free and odd, of size $\ll N^{O(1)}$. Then we have 
\begin{equation}
\sum_{\substack{n\leq x \\ n\equiv a\,\,(\text{mod}\,\,q) \\ n\equiv 1\,\,(\text{mod}\,\,4) \\ d|n}}r^2(n) = \frac{16H(1,d,4q)}{\varphi(4q)}\bigg[\log{x}+2\gamma-1+\frac{H'(1,d,4q)}{H(1,d,4q)}\bigg]x
+ O_{\epsilon}(qx^{\frac{3}{4}+\epsilon}),
\end{equation}
where
\begin{align*}
H(s,d,4q)&=\frac{L^2(s,\chi_4)G(s,4q)A(s,d)}{\zeta(2s)}, \\
G(s,4q)&= \prod_{p|4q}\bigg(1-\frac{1}{p^s}\bigg)\bigg(1+\frac{1}{p^s}\bigg)^{-1}\bigg(1-\frac{\chi_4(p)}{p^s}\bigg)^{2}, \\
A(s,d) &= \prod_{\substack{p|d \\ p\equiv 3\,\,(\text{mod}\,\,4)}}\frac{1}{p^{2s}}\prod_{\substack{p|d \\ p\equiv 1\,\,(\text{mod}\,\,4)}}\frac{4p^{2s}-3p^s+1}{p^{2s}(p^s+1)}.
\end{align*}
\end{lemma}
\begin{proof}[Proof of Lemma~\ref{lemma:3}]
One obtains the main term as stated in Lemma \ref{lemma:3} by taking the logarithmic derivative of $H(s,d,4q)$ (taking an appropriate branch-cut) and evaluating at $s=1.$ 
\end{proof}

We now outline the proof of Lemma~\ref{theo:Perron}. The proof uses a standard application of Perron's formula. The following lemma, which combines \cite[Theorem II.8.20]{Tenen} and \cite[Theorem II.8.22]{Tenen}, will prove necessary. 

\begin{lemma}\label{lemma:appendixtech} Let $\mathcal{L}= \log{(|t|+Q+1)}$ and let $\chi_Q\,\,(\text{mod}\,\,Q)$ be a Dirichlet character. For $\sigma \geq 1$ we have the following
\begin{enumerate}
	\item If $\chi_{Q}^2$ is complex then $L(s,\chi_Q^2)^{-1} \ll \mathcal{L}^7.$
	\item If $\chi_Q^2$ is real, non-trivial, then there exists an absolute constant $c_0>0$ such that
	$$
L(s,\chi_Q^2)^{-1} \ll 
\begin{cases}
\mathcal{L}^6 (\mathcal{L}+1/|t|)\,\,&\text{if $|t| > c_0 Q^{-\frac{1}{2}}(\log{2Q})^{-2}$,} \\
Q^{\frac{1}{2}}\,\,&\text{if $|t| \leq c_0 Q^{-\frac{1}{2}}(\log{2Q})^{-2}$.} \\
\end{cases}
$$
\end{enumerate}
\end{lemma} 

Write $Q=4q$ and let $\Delta\,\,(\text{mod}\,\,Q)$ be the unique solution to the congruences $\Delta\equiv a\,\,(\text{mod}\,\,q)$ and $\Delta\equiv 1\,\,(\text{mod}\,\,4)$.
Since $(\Delta,Q)=1$ by character orthogonality we can write
\begin{equation*}
\sum_{\substack{n\leq x \\ n\equiv \Delta\,\,(\text{mod}\,\,Q) \\ d|n}}r^2(n) = \frac{1}{\varphi(Q)}\sum_{\chi_{Q}\,\,(\text{mod}\,\,Q)} \overline{\chi(\Delta)} \sum_{\substack{n\leq x \\ d|n}}r^2(n)\chi_Q(n).
\end{equation*}
We study these sums by considering their generating series
\begin{equation*}
F_d(s,\chi_Q)=\sum_{n=1}^{\infty}\frac{r^2(dn)\chi_Q(dn)}{(dn)^s}
\end{equation*}
Write $F(s,\chi_Q)=F_1(s,\chi_Q).$ Then
\begin{equation*}
F(s,\chi_Q)=16L^2(s,\chi_Q)L^2(s,\chi_4\chi_Q)L(2s,\chi_Q^2)^{-1}.
\end{equation*}
Here $\chi_4$ denotes the unique non-trivial character $(\text{mod}\,\,4),$ and $L(s,\chi_D)$ denotes the L series corresponding to the Dirichlet character $\chi_D\,\,(\text{mod}\,\,D).$ It follows that
\begin{align*} \nonumber
\frac{F_d(s,\chi_Q)}{F(s,\chi_Q)} &= \prod_{p|d} \bigg(\sum_{k=1}^{\infty} \frac{r(p^k)^2\chi_Q(p^k)}{p^{ks}}\bigg)\bigg(\sum_{k=0}^{\infty} \frac{r(p^k)^2\chi_Q(p^k)}{p^{ks}}\bigg)^{-1}\\ 
&=\prod_{\substack{p|d \\ p\equiv 3\,\,(\text{mod}\,\,4)}}\frac{\chi_Q(p)^2}{p^{2s}}\prod_{\substack{p|d \\ p\equiv 1\,\,(\text{mod}\,\,4)}}\frac{\frac{4p^{2s}}{\chi_Q(p)^2}-\frac{3p^s}{\chi_Q(p)}+1}{\frac{p^{2s}}{\chi_Q(p)^2}(\frac{p^s}{\chi_Q(p)}+1)}= A(s,d,Q)
\end{align*}
say, using the fact $d$ is square-free and $(d,Q)=1$. We note the divisor bound $|r^2(n)\chi_q(n)| \leq \tau^2(n) \leq e^{C\frac{\log{n}}{\log\log{n}}}$ for some explicit $C>0.$ We can apply an effective form of Perron's formula (for example \cite[Corollary II.2.4]{Tenen}), averaging over the height $T,$ to obtain
\begin{align} \nonumber
\sum_{\substack{n\leq x \\ n\equiv \Delta\,\,(\text{mod}\,\,Q) \\ d|n}}r^2(n) &= I(T;Q,\Delta,d)+O\bigg(\frac{x(\log{x})^2}{T}+ e^{\frac{2C\log{x}}{\log\log{x}}}\bigg(1+\frac{x\log{T}}{T}\bigg)\bigg),
\end{align}
where
\begin{equation*}
I = \frac{1}{\varphi(Q)}\sum_{\chi_Q\,\,(\text{mod}\,\,Q)}\overline{\chi_Q(\Delta)}\bigg[\frac{1}{T}\int_T^{2T} \bigg(\frac{1}{2\pi i}\int_{c-it_0}^{c+it_0} \frac{16L^2(s,\chi_Q)L^2(s,\chi_4\chi_Q)A(s,d,Q)x^s}{L(2s,\chi_Q^2)s} \mathrm{d}s\bigg)\mathrm{d}t_0\bigg].
\end{equation*}
Here $c=1+1/\log{x}.$ We move the contour to the region defined by $[c-it_0,c+it_0],$ $[1/2+it_0,c+it_0],$ $[1/2-it_0,1/2+it_0]$ and $[1/2-it_0,c+it_0].$ The integrand has a pole of order 2 at $s=1,$ coming from the trivial character $\chi_0\,\,(\text{mod}\,\,Q)$. The residue of this pole is 
\begin{equation*}
H(1,d,Q)\bigg(\log{x}+2\gamma-1+\frac{H'(1,d,Q)}{H(1,d,Q)}\bigg)x,
\end{equation*}
where
\begin{align*}
H(s,d,Q)&=\frac{16L^2(s,\chi_4)G(s,Q)A(s,d)}{\zeta(2s)},\\
G(s,Q)&= \prod_{p|Q}\bigg(1-\frac{1}{p^s}\bigg)\bigg(1+\frac{1}{p^s}\bigg)^{-1}\bigg(1-\frac{\chi_4(p)}{p^s}\bigg)^{2}, \\
A(s,d) &= \prod_{\substack{p|d \\ p\equiv 3\,\,(\text{mod}\,\,4)}}\frac{1}{p^{2s}}\prod_{\substack{p|d \\ p\equiv 1\,\,(\text{mod}\,\,4)}}\frac{4p^{2s}-3p^s+1}{p^{2s}(p^s+1)}.
\end{align*}
We need bounds for the integrand in the region $\sigma \geq 1/2$ and $|t|\leq 2T.$ One can easily check the trivial bound $|A(s,d,Q)|\ll 1$ (uniformly in $d$ and $q$). We require a lower bound for $|L(2s,\chi_Q^2)|$ in this region. If $\chi_Q$ is real then $\chi_Q^2$ is the trivial character, and so we have
\begin{equation*}
L(2s,\chi_Q^2) = \zeta(2s) \prod_{p|Q}\bigg(1-\frac{1}{p^{2s}}\bigg).
\end{equation*}
In the region $\sigma \geq \frac{1}{2}$ we can bound the product from below by $\varphi(Q)/Q.$ Standard bounds for $\zeta(s)$ on the 1-line (see for example \cite[Theorem II.3.9]{Tenen}) then tell us that
\begin{equation*}
L(2s,\chi_Q^2)^{-1} \ll  \frac{Q}{\varphi(Q)} \log{(|t|+2)} \ll_{\epsilon} (QT)^{\epsilon}.
\end{equation*}
If $\chi_{Q}$ is complex then we have to be more careful. We use Lemma~\ref{lemma:appendixtech}, and see there is a minor technical complication where we must bound the contribution from $s=1/2+it$ with $|t|\leq 2c_0$ (say) separately (with this choice of cut-off we can use the bound $L(2s,\chi_Q^2) \ll (QT)^{\epsilon}$ uniformly in $\chi_Q$ for all the other contours). One can easily show the contribution from these values of $s$ is $\ll Q^{\frac{1}{2}},$ which is small. 

To bound the contribution from the other integrals, we can use the fourth-moment estimate for Dirichlet L-functions on the critical line in the form
\begin{equation*}
\frac{1}{\varphi(D)T}\sum_{\substack{\chi \,\,(\text{mod}\,\,D)}} \int_{T}^{2T}|L(1/2+it,\chi)|^{4}\mathrm{d}t \ll (DT)^{\epsilon},
\end{equation*}
(cf.~\cite[Theorem 1]{HBui}) together with the Cauchy-Schwarz inequality. One can show that, with the choice $T=x^{\frac{1}{4}},$ the contours contribute $\ll_{\epsilon}Qx^{\frac{3}{4}+\epsilon}.$ 


\section{Auxiliary estimates for $\rho(n)$}

To evaluate the sums appearing in Lemma~\ref{lemma:aux} we use the Selberg-Delange method.

\begin{lemma}[Selberg-Delange method]\label{lemma:SDL}
Let $F(s)=\sum_{n=1}^{\infty}a_n n^{-s}$ be a Dirichlet series such that the function $G(s;z)=F(s)\zeta(s)^{-z}$ can be analytically continued to the region 
\begin{equation}\label{eq:region}
\sigma > 1-\frac{c_0}{1+\log{(2+|t|)}}
\end{equation}
for some positive $c_0>0$ and $z\in \mathbb{C}$ with $|z|\leq A,$ and moreover satisfies the bounds $|G(s;z)|\ll M(1+|t|)^{1-\delta}$ for some $\delta>0$ in this region. Let $Z_1(s;z)=[\zeta(s)(s-1)]^z, Z_2(s;z)=\frac{Z_1(s;z)}{s}$ (both holomorphic in the disc $|s-1|<1$) and let
\begin{align*}
G(s;z)Z_1(s;z)&=\sum_{k=0}^{\infty}\mu_k(z)(s-1)^k, \\
G(s;z)Z_2(s;z)&=\sum_{k=0}^{\infty}\eta_k(z)(s-1)^k
\end{align*}
be the Taylor series in this region. Then for any $N\geq 0$ and $|z|\leq A$ we have
\begin{equation}
\sum_{n\leq x} \frac{a_n}{n}\log{\frac{x}{n}} = (\log{x})^{z+1}\bigg[\sum_{k=0}^{N}\frac{\mu_k(z)}{\Gamma(z+2-k)(\log{x})^k}+O(MR_N(x))\bigg],
\end{equation}
and moreover if $a_n>0$ then we also have 
\begin{equation}
\sum_{n\leq x} a_n = x(\log{x})^{z-1}\bigg[\sum_{k=0}^{N}\frac{\eta_k(z)}{\Gamma(z-k)(\log{x})^k}+O(MR_N(x))\bigg],
\end{equation}
where
\begin{equation}
R_N(x)=R_N(x;c_1,c_2) = e^{-c_1\sqrt{\log{x}}}+O\bigg(\frac{c_2N+1}{\log{(x+1)}}\bigg)^{N+1}
\end{equation}
for some suitable constants $c_1,c_2>0$. These positive constants, and the implicit constants in the Landau symbol, depend at most on $c_0,\delta$ and $A$.
\end{lemma}
The second statement is exactly \cite[Theorem II.5.2]{Tenen}. The first statement follows by the same proof with a few minor changes.

Lemma~\ref{lemma:aux} follows from suitable applications of Lemma \ref{lemma:SDL}. We sketch the details in each case.

\begin{proof}[Proof of Lemma~\ref{lemma:aux}]
\begin{enumerate}[label=(\roman*), wide, labelwidth=!, labelindent=0pt]
	\item For $X_{N,W}$ recall from (\ref{eq:X}) the definition
\begin{equation}
X_{N,W}=\sum_{\substack{a\leq v \\ (a,W)=1 \\ p|a\Rightarrow p\equiv 1\,\,(\text{mod}\,\,4)}} \frac{\mu(a)}{a}\log{\frac{v}{a}}.
\end{equation}
Consider the generating series, for $\sigma>1$ 
	\begin{equation*}
	h_{1}(W,s) = \sum_{\substack{n=1 \\ (n,W)=1 \\ p|n\Rightarrow p\equiv 1\,\,(\text{mod}\,\,4)}}^{\infty}\frac{\mu(n)}{n^s} =  \prod_{\substack{p\nmid W \\ p\equiv 1\,\,(\text{mod}\,\,4)}}\bigg(1-\frac{1}{p^s}\bigg) = \frac{K_1(s)G_1(W,s)}{(\zeta(s)L(s,\chi_4))^{\frac{1}{2}}},
	\end{equation*} 
	where
	\begin{align*} 
	K_1(s)^2 &= \bigg(1-\frac{1}{2^s}\bigg)^{-1}\prod_{p\equiv 3\,\,(\text{mod}\,\,4)}\bigg(1-\frac{1}{p^{2s}}\bigg)^{-1}, \\
	G_1(W,s) &=  \prod_{\substack{p|W \\ p\equiv 1\,\,(\text{mod}\,\,4)}}\bigg(1-\frac{1}{p^s}\bigg)^{-1}
	\end{align*}
	(taking the positive determination of the square-root in the first instance). Both of these functions are analytic for $\sigma>3/4$ (say). $K_1(s)$ is bounded in this region, and $G_1(W,s)$ satisfies
	\begin{align*}
	|G_{1}(W,s)| &\leq  \prod_{\substack{p|W \\ p\equiv 1\,\,(\text{mod}\,\,4)}}\bigg(1-\frac{1}{p^{\frac{3}{4}}}\bigg)^{-1} = \exp\bigg[- \sum_{\substack{p\leq D_0 \\ p\equiv 1\,\,(\text{mod}\,\,4)}}\log\bigg(1-\frac{1}{p^{3/4}}\bigg)\bigg] \\
	&= \exp\sum_{\substack{p\leq D_0 \\ p\equiv 1\,\,(\text{mod}\,\,4)}}\bigg[\frac{1}{p^{3/4}}+O\bigg(\frac{1}{p^{3/2}}\bigg)\bigg] \ll  \exp\sum_{\substack{p\leq D_0 \\ p\equiv 1\,\,(\text{mod}\,\,4)}}\frac{D_0^{1/4}}{p} \\
	&\ll \exp(D_0^{1/4}\log\log{D_0}).
	\end{align*}
	We remark that with the choice $D_0 = (\log\log{N})^3,$ we certainly have the bound 
	$$\exp(D_0^{1/4}\log\log{D_0}) \ll_{\epsilon} (\log{N})^{\epsilon}$$
	for any fixed $\epsilon>0,$ and using this estimate, it is a simple task to verify that all the error terms which follow are indeed controlled. (Recall $v=N^{\theta}$ for some fixed $\theta>0$.)
	
	From the classical zero-free region for Dirichlet L-functions, $L(s)^{-1}$ can be analytically continued to a region of the form (\ref{eq:region}) for some $c_0$ and moreover satisfies a bound $L(s)^{-1}\ll \log{(2+|t|)}$ in this region (cf. \cite[Chapter 14]{Davenport}). Thus we can apply Lemma~\ref{lemma:SDL} with $M=\exp(D_0^{1/4}\log\log{D_0}), z=-\frac{1}{2},N=0$ and some suitable choices of $c_0,\delta,$ to obtain 
\begin{equation*}
X_{N,W} = \frac{K_1(1)G_{1}(W)}{\Gamma(3/2)\sqrt{L(1,\chi_4)}}(\log{v})^{\frac{1}{2}}+O\bigg(\frac{\exp (D_0^{\frac{1}{4}}\log\log{D_0})}{(\log{v})^{\frac{1}{2}}}\bigg).
\end{equation*}	 
Here we have written $G_{1}(W)=G_{1}(W,1)$ (and this convention will be continued below the fold). This simplifies to the stated result. Note that 
$$G_1(W) = \prod_{\substack{p\leq D_0 \\ p\equiv 1\,\,(\text{mod}\,\,4)}}\bigg(1-\frac{1}{p}\bigg)^{-1} \asymp (\log{D_0})^{\frac{1}{2}}$$
by Mertens' theorem. 
\\ 
	\item For $Y_{N,W}$ recall from (\ref{eq:Y}) the definition
\begin{equation}
Y_{N,W}=\sum_{\substack{a,b\leq v\\ (a,W)=(b,W)=1 \\ (a,b)=1 \\ p|a,b\Rightarrow p\equiv 1\,\,(\text{mod}\,\,4)}} \frac{\mu(a)\mu(b)}{g_7(a)g_7(b)}\log{\frac{v}{a}}\log{\frac{v}{b}}.
\end{equation}
Take the sum over $b$ on the inside. We can evaluate 
	\begin{equation*}
	\sum_{\substack{b\leq v \\ (b,aW)=1 \\ p|b\Rightarrow p\equiv 1\,\,(\text{mod}\,\,4)}}\frac{\mu(b)}{g_7(b)}\log{\frac{v}{b}}
	\end{equation*}
	using the generating function, for $\sigma>1$
	\begin{align} \nonumber
	h_2(aW,s) &= \sum_{\substack{n=1 \\ (n,aW)=1 \\ p|n\Rightarrow p\equiv 1\,\,(\text{mod}\,\,4)}}^{\infty}\frac{\mu(n)n}{n^sg_7(n)} =  h_{1}(aW,s)K_2(s)G_2(aW,s),
	\end{align}
	where 
	\begin{align*} 
	K_2(s) &= \prod_{p\equiv 1\,\,(\text{mod}\,\,4)}\bigg(1+\frac{1}{(p^s-1)(p+1)}\bigg), \\
	G_2(aW,s) &= \prod_{\substack{p|aW\\ p\equiv 1\,\,(\text{mod}\,\,4)}}\bigg(1+\frac{1}{(p^s-1)(p+1)}\bigg)^{-1}.
	\end{align*}
	Arguing as above, Lemma \ref{lemma:SDL} yields
	\begin{equation*}
	Y_{N,W} = T_1+T_2+O(T_3),
	\end{equation*}
	this time taking $M=\exp(D_0^{1/4}\log\log{D_0}), z=-1/2$ and $N=1.$ Here
	\begin{align*}
	T_3 &= \frac{M}{(\log{v})^{\frac{3}{2}}}\sum_{\substack{a\leq v \\ (a,W)=1 \\ p|a \Rightarrow p\equiv 1\,\,(\text{mod}\,\,4)}}\frac{\mu^2(a)\tau(a)}{g_7(a)}\log{\frac{v}{a}} \ll \frac{M}{(\log{v})^{\frac{1}{2}}} \prod_{\substack{p\leq v \\ p\equiv 1\,\,(\text{mod}\,\,4)}}\bigg(1+\frac{2}{p+1}\bigg) \\
	&\ll \frac{M}{(\log{v})^{\frac{1}{2}}} \prod_{\substack{p\leq v \\ p\equiv 1\,\,(\text{mod}\,\,4)}}\bigg(1+\frac{1}{p}\bigg)^2 \ll M (\log{v})^{\frac{1}{2}} = \exp(D_0^{1/4}\log\log{D_0})(\log{v})^{\frac{1}{2}}.
	\end{align*}
	To evaluate $T_2,$ we need a handle on the second Taylor series coefficient. With notation as in the statement of Lemma~\ref{lemma:SDL}, this is labelled $\mu_1(-1/2)$ and is precisely equal to 
	$$G(1)Z_1'(1;-1/2)+G'(1) Z_1(1;-1/2),$$
	where
	$$G(s) = \frac{K_1(s)G_1(aW,s)K_2(s)G_2(aW,s)}{L(s,\chi_4)^{1/2}}$$ 
	(taking an appropriate branch-cut). Here all derivatives are taken with respect to $s$. By a series expansion, one can directly show that $Z_1(1;-1/2)=1$ and $Z_1'(1;-1/2)=-\gamma/2,$ where $\gamma$ is the Euler-Mascheroni constant. We now note the following:
	\begin{align*} 
	G_1(p)G_2(p)&=\frac{g_7(p)}{p}, \\
	G_1'(d)&= G_1(d) \sum_{\substack{p|d \\ p\equiv 1\,\,(\text{mod}\,\,4)}}\frac{\log{p}}{p-1}, \\
	G_2'(d)&= G_2(d) \sum_{\substack{p|d \\ p\equiv 1\,\,(\text{mod}\,\,4)}}\frac{\log{p}}{p(p-1)}.
	\end{align*}
	Here we have defined $G_i'(d)=\frac{\mathrm{d}G_i(d,s)}{\mathrm{d}s}\big|_{s=1}$ for $i\in\{1,2\}$. These results are valid for primes $p\equiv 1\,\,(\text{mod}\,\,4),$ and for any square-free $d.$ Thus one can write $T_2$ as a finite linear combination
	\begin{align*}
	 T_2 &= \frac{1}{(\log{v})^{\frac{1}{2}}}\sum_{i} c_i \alpha_i(W_1) R_i,
	 \end{align*}
	 where $c_i\in \mathbb{R}$ are bounded, 
	 \begin{equation*}
	 \alpha_i(W_1)\in \bigg\{\frac{g_7(W_1)}{W_1}, \,\,\frac{g_7(W_1)}{W_1}\sum_{p|W_1}\frac{\log{p}}{p-1},\,\,\frac{g_7(W_1)}{W_1}\sum_{p|W_1}\frac{\log{p}}{p(p-1)}\bigg\},
	 \end{equation*}
	 and $R_i$ is one of 
	 \begin{align*}
	 &\sum_{\substack{a\leq v \\ (a,W)=1 \\ p|a \Rightarrow p\equiv 1\,\,(\text{mod}\,\,4)}}\frac{\mu(a)}{a}\log{\frac{v}{a}} \ll (\log{D_0})^{1/2}(\log{v})^{\frac{1}{2}} 
	 \end{align*}
	 or
	 \begin{align*}
	&\sum_{\substack{D_0<q<v \\ q\equiv 1\,\,(\text{mod}\,\,4)}} \frac{\log{q}}{q-1}\sum_{\substack{a\leq v \\ (a,W)=1 \\ q|a \\ p|a \Rightarrow p\equiv 1\,\,(\text{mod}\,\,4)}}\frac{\mu(a)}{a}\log{\frac{v}{a}} \ll \exp(D_0^{1/4}\log\log{D_0})(\log{v})^{1/2},
	\end{align*}
	or finally
	\begin{align*}
	&\sum_{\substack{D_0<q<v \\ q\equiv 1\,\,(\text{mod}\,\,4)}} \frac{\log{q}}{q(q-1)}\sum_{\substack{a\leq v \\ (a,W)=1 \\ q|a \\ p|a \Rightarrow p\equiv 1\,\,(\text{mod}\,\,4)}}\frac{\mu(a)}{a}\log{\frac{v}{a}} \ll \frac{(\log{D_0})^{3/2}(\log{v})^{1/2}}{D_0^2}.
	\end{align*}
	Here $q$ denotes a prime variable (a convention that will also be used below the fold). The first estimate follows directly from our expression for $X_{N,W}$ found above. For the last two estimates, we note that the inner sum appearing in both is precisely $X_{N,W}-X_{N,qW},$ and from our work above it follows that 
	\begin{align*}
	X_{N,W} - X_{N,qW} &= \frac{K_1(1)}{\Gamma(3/2)\sqrt{L(1,\chi_4)}}(\log{v})^{\frac{1}{2}}(G_1(W)-G_1(qW)) + O\bigg(\frac{\exp (D_0^{\frac{1}{4}}\log\log{D_0})}{(\log{v})^{\frac{1}{2}}}\bigg) \\
	&\ll \frac{(\log{D_0})^{\frac{1}{2}}(\log{v})^{\frac{1}{2}}}{q} + \frac{\exp (D_0^{\frac{1}{4}}\log\log{D_0})}{(\log{v})^{\frac{1}{2}}}.
	\end{align*}
	Now, standard estimates for sums over primes such as
	$$\sum_{\substack{q>D_0 \\ q\equiv 1\,\,(\text{mod}\,\,4)}} \frac{\log{q}}{q(q-1)} \ll\frac{\log{D_0}}{D_0},$$ 
	yields the results stated.
	
	Now, using the fact 
	$$\sum_{p|W_1}\frac{\log{p}}{p-1}\ll \log{D_0},$$
	and 
	\begin{align*}
	\frac{g_7(W_1)}{W_1} &= \prod_{\substack{p \leq D_0 \\ p\equiv 1\,\,(\text{mod}\,\,4)}}\bigg(1+\frac{1}{p}\bigg) \ll (\log{D_0})^{\frac{1}{2}},
	\end{align*}
	it follows that the total contribution from $T_2$ is
	\begin{align*}
	T_2 \ll  (\log{D_0})^{3/2}\exp(D_0^{1/4}\log\log{D_0})(\log{v})^{1/2}
	\end{align*}
	which is negligible. 
	
	Finally, we can write
	\begin{equation*}
	T_1 = \frac{K_1(1)K_2(1)G_1(W_1)G_2(W_1)(\log{v})^{\frac{1}{2}}}{\Gamma(3/2)\sqrt{L(1,\chi_4)}}\sum_{\substack{a\leq v \\ (a,W)=1 \\ p|a \Rightarrow p\equiv 1\,\,(\text{mod}\,\,4)}}\frac{\mu(a)}{a}\log{\frac{v}{a}}
	\end{equation*}
	This last sum is exactly $X_{N,W}.$ Using our bound from part (i), and also the fact 
	\begin{equation*}
	K_2(1)G_2(W_1)=\prod_{\substack{p>D_0 \\ p\equiv 1\,\,(\text{mod}\,\,4)}}\bigg(1+\frac{1}{p^2-1}\bigg) = 1+O(D_0^{-1}),
	\end{equation*}
	we arrive at a final estimate of
	\begin{align*}
	Y_{N,W} = (1+O(D_0^{-1}))\bigg[&\frac{K_1(1)G_{1}(W)}{\Gamma(3/2)\sqrt{L(1,\chi_4)}}\bigg]^2\log{v}\\
	&+O((\log{D_0})^{3/2}\exp(D_0^{1/4}\log\log{D_0})(\log{v})^{1/2}).
	\end{align*}
	\\ 
	\item For $Z_{N,W}^{(1)}$ recall from (\ref{eq:Z1}) the definition
\begin{equation}
Z_{N,W}^{(1)}=\sum_{\substack{a,b\leq v\\ (a,W)=(b,W)=1 \\ p|a,b\Rightarrow p\equiv 1\,\,(\text{mod}\,\,4)}}\frac{\mu(a)\mu(b)g_4([a,b])}{g_2(a)g_2(b)[a,b]}\log{\frac{v}{a}}\log{\frac{v}{b}}.
\end{equation}
We write this as 
	\begin{equation*}
	\sum_{\substack{a\leq v \\ (a,W)=1 \\ p|a\Rightarrow p\equiv 1\,\,(\text{mod}\,\,4)}}\frac{\mu(a)g_4(a)}{g_2(a)a}\log{\frac{v}{a}}\sum_{d|a}\frac{d}{g_4(d)}\sum_{\substack{b\leq v \\ (b,W)=1 \\ (b,a)=d \\ p|b\Rightarrow p\equiv 1\,\,(\text{mod}\,\,4)}}\frac{\mu(b)g_4(b)}{g_2(b)b}\log{\frac{v}{b}}.
	\end{equation*}
	Now substitute $b=md.$ The inner sums can be rewritten as 
	\begin{equation*}
	\sum_{d|a}\frac{\mu(d)}{g_2(d)} \sum_{\substack{m\leq v/d \\ (m,aW)=1 \\ p|m\Rightarrow p\equiv 1\,\,(\text{mod}\,\,4)}}\frac{\mu(m)g_4(m)}{g_2(m)m}\log{\frac{v}{md}},
	\end{equation*}
	where we have used the fact $a$ is square-free. To handle the inner sum we consider the generating series 
		\begin{align*} \nonumber
	h_{3}(aW,s) &= \sum_{\substack{n=1 \\ (n,aW)=1 \\ p|n\Rightarrow p\equiv 1\,\,(\text{mod}\,\,4)}}^{\infty}\frac{\mu(n)g_4(n)}{n^sg_2(n)} =\frac{K_3(s)K_4(s)G_{3}(aW,s)G_{4}(aW,s)}{\zeta(s)L(s,\chi_4)},
	\end{align*}
	where
	\begin{align*}
	K_3(s)&=\bigg(1-\frac{1}{2^s}\bigg)^{-1}\prod_{p\equiv 3\,\,(\text{mod}\,\,4)}\bigg(1-\frac{1}{p^{2s}}\bigg)^{-1} \prod_{p\equiv 1\,\,(\text{mod}\,\,4)}\bigg(1-\frac{1}{(p^s-1)^2}\bigg), \\
	K_4(s)&=\prod_{p\equiv 1\,\,(\text{mod}\,\,4)}\bigg(1+\frac{5p-3}{(p+1)(2p-1)(p^s-2)}\bigg), \\
	G_{3}(aW,s) &= \prod_{\substack{p|aW \\ p\equiv 1\,\,(\text{mod}\,\,4)}}\bigg(1-\frac{2}{p^s}\bigg)^{-1}, \\
	G_{4}(aW,s) &= \prod_{\substack{p|aW \\ p\equiv 1\,\,(\text{mod}\,\,4)}}\bigg(1+\frac{5p-3}{(p+1)(2p-1)(p^s-2)}\bigg)^{-1}.
	\end{align*}
	These four functions are analytic in the region $\sigma > 3/4$ (say), where they all satisfy the bound $O(1)$ except for $G_3.$ Since $(a,W)=1$, by multiplicativity we can write $G_3(aW,s) =  G_3(a,s)G_3(W,s).$ As above, we can bound $|G_3(W,s)| \ll \exp(D_0^{1/4}\log\log{D_0})$ in this region. On the other hand, note that
	$$ 1\leq \bigg(1-\frac{2}{p^{3/4}}\bigg)^{-1} \leq 2$$
	whenever $p\geq 7.$ Since we are assuming $a$ is square-free, it follows that we can bound $|G_3(a,s)| \ll \tau(a).$

	  By similar arguments to part (i) we can apply Lemma \ref{lemma:SDL} taking $z=-1, M=\tau(a)\exp(D_0^{1/4}\log\log{D_0})$ and suitable $\delta,c_0.$ Note that for $N\geq 1$ the terms $\mu_k(z)(\log{x})^{z+1-k}/\Gamma(z+2-k)$ are 0 for $1\leq k\leq N.$ Hence we can choose $N=\left \lfloor (\log{x})/ec_1 \right \rfloor$ (for some suitable $c_1>0$) to balance the error terms, yielding a stronger error term of the form $O(Me^{-c_1\sqrt{\log{x}}}).$ Using this choice of $N$ we obtain
	\begin{equation*}
	 \sum_{\substack{m\leq v/d \\ (m,aW)=1 \\ p|m\Rightarrow p\equiv 1\,\,(\text{mod}\,\,4)}}\frac{\mu(m)g_4(m)}{g_2(m)m}\log{\frac{v}{md}} = \frac{K_3(1)K_4(1)G_{3}(aW)G_4(aW)}{L(1,\chi_4)} +O\bigg(Me^{-c_1\sqrt{\log{v/d}}}\bigg).
	\end{equation*}
	This error term contributes
	\begin{equation*}
	 \ll \exp(D_0^{1/4}\log\log{D_0})\sum_{\substack{a\leq v \\ p|a\Rightarrow p\equiv 1\,\,(\text{mod}\,\,4)}}\frac{\mu^2(a)g_4(a)\tau(a)}{g_2(a)a}\log{\frac{v}{a}}\sum_{d|a}\frac{\mu^2(d)}{g_2(d)}e^{-c_1\sqrt{\log{v/d}}}.
	\end{equation*}
	Swapping sums yields
	\begin{equation*}
	\sum_{\substack{d\leq v \\p|d\Rightarrow p\equiv 1\,\,(\text{mod}\,\,4)}}\frac{\mu^2(d)g_4(d)\tau(d)}{g_2(d)^2d}e^{-c_1\sqrt{\log{v/d}}} \sum_{\substack{m\leq v/d \\ (m,d)=1 \\ p|m\Rightarrow p\equiv 1\,\,(\text{mod}\,\,4)}}\frac{\mu^2(m)g_4(m)\tau(m)}{g_2(m)m}\log{\frac{v}{md}}
	\end{equation*}
	The inner sum can be bounded by
	\begin{align*} \nonumber
	&\ll \log{\frac{v}{d}} \prod_{\substack{p\leq v/d \\ p\equiv 1\,\,(\text{mod}\,\,4)}}\bigg(1+\frac{2(4p^2-3p+1)}{p(p+1)(2p-1)}\bigg) \ll \log{\frac{v}{d}} \prod_{\substack{p\leq v/d\\ p\equiv 1\,\,(\text{mod}\,\,4)}}\bigg(1+\frac{4}{p}\bigg), \\
	&\ll \log{\frac{v}{d}}\prod_{\substack{p\leq v/d\\ p\equiv 1\,\,(\text{mod}\,\,4)}}\bigg(1+\frac{1}{p}\bigg)^4 \ll  \bigg(\log{\frac{v}{d}}\bigg)^3 \ll e^{c_2\sqrt{\log{v/d}}}
	\end{align*}
	for some suitably small $c_2>0.$ Note this final bound is valid for all $d\leq v.$ Using the fact $g_4(d)/g_2(d)^2 \leq 1$, we see the total error is 
	\begin{align}\label{eq:apperror}
	&\ll \exp(D_0^{1/4}\log\log{D_0}) \sum_{\substack{d\leq v \\p|d\Rightarrow p\equiv 1\,\,(\text{mod}\,\,4)}}\frac{\mu^2(d)\tau(d)}{d}e^{-c_3\sqrt{\log{v/d}}} 
	\end{align}
	for some $0<c_3<c_1.$ For this sum, we use the generating series, for $\sigma>1$
	\begin{equation*}
	\sum_{\substack{n=1 \\ p|n\Rightarrow p\equiv 1\,\,(\text{mod}\,\,4)}}^{\infty} \frac{\mu^2(n)\tau(n)}{n^s} = \prod_{p\equiv 1\,\,(\text{mod}\,\,4)}\bigg(1+\frac{2}{p^s}\bigg) = \frac{\zeta(s)L(s,\chi_4)K(s)}{\zeta(2s)L(2s,\chi_4)},
	\end{equation*}
	where
	\begin{equation*}
	K(s) = \bigg(1+\frac{1}{2^s}\bigg)^{-1}\prod_{p\equiv 3\,\,(\text{mod}\,\,4)}\bigg(1+\frac{1}{p^{2s}}\bigg)^{-1}\prod_{p\equiv 1\,\,(\text{mod}\,\,4)}\bigg(1+\frac{1}{p^s(p^s+2)}\bigg).
	\end{equation*}
	By the second part of Lemma \ref{lemma:SDL}, taking $z=1$ and $N=\left \lfloor (\log{x})/ec_4 \right \rfloor$ for some suitable $c_4>0$ we get
	\begin{equation*}
	\sum_{\substack{d\leq v \\ p|d\Rightarrow p\equiv 1\,\,(\text{mod}\,\,4)}} \mu^2(d)\tau(d) = \frac{L(1,\chi_4)K(1)}{\zeta(2)L(2,\chi_4)}v+O(ve^{-c_4\sqrt{\log{v}}}).
	\end{equation*}
	Now by splitting the sum at $v^{1-\epsilon}$ and using partial summation one can show that 
	$$ \sum_{\substack{d\leq v \\p|d\Rightarrow p\equiv 1\,\,(\text{mod}\,\,4)}}\frac{\mu^2(d)\tau(d)}{d}e^{-c_3\sqrt{\log{v/d}}} \ll 1+ (\epsilon + e^{-c_3\sqrt{\epsilon \log{v}}})\log{v}.$$
	Hence, taking $\epsilon=\frac{(\log\log{v})^3}{\log{v}},$ we see the error (\ref{eq:apperror}) is $O(\exp(D_0^{1/4}\log\log{D_0})(\log\log{v})^3),$ which is small. Let $\gamma_1(a)=\sum_{d|a}\mu(d)/g_2(d).$ We obtain a main term
	\begin{equation*}
	 \frac{K_3(1)K_4(1)G_{3}(W)G_4(W)}{L(1,\chi_4)}\sum_{\substack{a\leq v \\ (a,W)=1 \\ p|a\Rightarrow p\equiv 1\,\,(\text{mod}\,\,4)}}\frac{\mu(a)g_4(a)\gamma_1(a)G_{3}(a)G_4(a)}{g_2(a)a}\log{\frac{v}{a}}.
	\end{equation*}
	For this sum we consider the generating series, for $\sigma>1$
	\begin{align*} \nonumber
	h_4(W,s) &= \sum_{\substack{n=1 \\ (n,W)=1 \\ p|n\Rightarrow p\equiv 1\,\,(\text{mod}\,\,4)}}\frac{\mu(n)g_4(n)\gamma_1(n)G_3(n)G_4(n)}{n^sg_2(n)} = h_1(W,s)K_5(s)G_5(W,s).
	\end{align*}
	where 
	\begin{align*}
	K_5(s) &= \prod_{p\equiv 1\,\,(\text{mod}\,\,4)} \bigg(1-\frac{(p-1)^2}{(p^s-1)(2p-1)(2p^2-p+1)}\bigg), \\
	G_5(W,s) &= \prod_{\substack{p|W \\ p\equiv 1\,\,(\text{mod}\,\,4)}}\bigg(1-\frac{(p-1)^2}{(p^s-1)(2p-1)(2p^2-p+1)}\bigg)^{-1}.
	\end{align*}
	Applying Lemma \ref{lemma:SDL} with $N=0$ gives
	\begin{align*}\label{eq:rand1}
	Z_{N,W}^{(1)} &= C_{W} (\log{v})^{\frac{1}{2}}+O(\exp(D_0^{1/4}\log\log{D_0})(\log\log{v})^3),
	\end{align*}
	where
	 \begin{equation}\label{eq:CW}
C_W= \frac{K_1(1)K_3(1)K_4(1)K_5(1)G_1(W)G_3(W)G_4(W)G_5(W)}{\Gamma(3/2)L(1,\chi_4)^{\frac{3}{2}}}
\end{equation} 
	which reduces to the stated result (note that
	\begin{align*}
	K_3(1)G_3(W) &=\frac{A^2}{g_1(W_1)^2}\prod_{\substack{p>D_0 \\ p\equiv 1\,\,(\text{mod}\,\,4)}}\bigg(1-\frac{1}{(p-1)^2}\bigg) = \frac{A^2}{g_1(W_1)^2}(1+O(D_0^{-1}))
	\end{align*}
	to get the form stated). We note for future reference that $|C_W| \asymp (\log{D_0})^{3/2}.$
	\\ 
	\item For $Z_{N,W}^{(2)}$ recall from (\ref{eq:Z2}) the definition
\begin{equation}
Z_{N,W}^{(2)}=\sum_{\substack{a,b\leq v\\ (a,W)=(b,W)=1 \\ p|a,b\Rightarrow p\equiv 1\,\,(\text{mod}\,\,4)}}\frac{\mu(a)\mu(b)g_4([a,b])}{g_2(a)g_2(b)[a,b]}\log{\frac{v}{a}}\log{\frac{v}{b}}\sum_{p|[a,b]}g_6(p).
\end{equation}
We can write
	\begin{align*} 
	Z_{N,W}^{(2)} &= \sum_{\substack{D_0<q\leq v \\ q\equiv 1\,\,(\text{mod}\,\,4)}}g_6(q) \sum_{\substack{a,b\leq v\\ (a,W)=(b,W)=1 \\ q|[a,b] \\ p|a,b\Rightarrow p\equiv 1\,\,(\text{mod}\,\,4)}}\frac{\mu(a)\mu(b)g_4([a,b])}{g_2(a)g_2(b)[a,b]}\log{\frac{v}{a}}\log{\frac{v}{b}} \\ 
	&=  \sum_{\substack{D_0<q\leq v \\ q\equiv 1\,\,(\text{mod}\,\,4)}}g_6(q) \bigg[T_1+T_2-T_3\bigg]. 
	\end{align*}
where $q$ is prime, and 
\begin{align*}
T_1 &= \sum_{\substack{a,b\leq v\\ (a,W)=(b,W)=1 \\ q|a \\ p|a,b\Rightarrow p\equiv 1\,\,(\text{mod}\,\,4)}}\frac{\mu(a)\mu(b)g_4([a,b])}{g_2(a)g_2(b)[a,b]}\log{\frac{v}{a}}\log{\frac{v}{b}}, \\
T_2 &= \sum_{\substack{a,b\leq v\\ (a,W)=(b,W)=1 \\ q|b \\ p|a,b\Rightarrow p\equiv 1\,\,(\text{mod}\,\,4)}}\frac{\mu(a)\mu(b)g_4([a,b])}{g_2(a)g_2(b)[a,b]}\log{\frac{v}{a}}\log{\frac{v}{b}}, \\
T_3 &= \sum_{\substack{a,b\leq v\\ (a,W)=(b,W)=1 \\ q|a,b \\ p|a,b\Rightarrow p\equiv 1\,\,(\text{mod}\,\,4)}}\frac{\mu(a)\mu(b)g_4([a,b])}{g_2(a)g_2(b)[a,b]}\log{\frac{v}{a}}\log{\frac{v}{b}}.
\end{align*}
$T_1$ can be evaluated similarly to part (iii) to give 
\begin{align*}
T_1 &= \frac{C_W\beta_1(q)}{q}(\log{v/q})^{\frac{1}{2}}+O\bigg(\frac{\exp(D_0^{1/4}\log\log{D_0})(\log\log{v})^3}{q}\bigg), \\
\end{align*}
where
\begin{align*}
\beta_1(q) &= \frac{\mu(q)g_4(q)\gamma_1(q)G_1(q)G_3(q)G_4(q)G_5(q)}{g_2(q)} = -\frac{q(4q^2-3q+1)}{2(q-1)(2q^2-2q+1)}.
\end{align*}
$T_2$ can be evaluated similarly. For $T_3,$ write $a=a'q, b=b'q$ then $[a,b]=[qa',qb']=q[a',b']$ so that
\begin{equation*}
T_3 = \frac{\mu^2(q)g_4(q)}{g_2(q)^2q} \sum_{\substack{a',b'\leq v/q\\ (a',qW)=(b',qW)=1 \\ p|a',b'\Rightarrow p\equiv 1\,\,(\text{mod}\,\,4)}}\frac{\mu(a')\mu(b')g_4([a',b'])}{g_2(a')g_2(b')[a',b']}\log{\frac{v}{a'}}\log{\frac{v}{b'}}.
\end{equation*}
This is the same form as $Z_{N,W}^{(1)}.$ By the same considerations as above we obtain
\begin{align*} \nonumber
T_3 &= \frac{C_W\beta_2(q)}{q}(\log{v/q})^{\frac{1}{2}}+O\bigg(\frac{\exp(D_0^{1/4}\log\log{D_0})(\log\log{v})^3}{q}\bigg), 
\end{align*}
where
\begin{align*}
\beta_2(q) &= \frac{\mu^2(q)g_4(q)G_1(q)G_3(q)G_4(q)G_5(q)}{g_2(q)^2} = \frac{q^2(4q^2-3q+1)}{2(q-1)^2(2q^2-2q+1)}.
\end{align*}
Thus we obtain
\begin{align*}
Z_{N,W}^{(2)} = C_{W} &\sum_{\substack{D_0<q\leq v \\ q\equiv 1\,\,(\text{mod}\,\,4)}}\frac{2\beta_1(q)-\beta_2(q)}{q}g_6(q)(\log{(v/q)})^{\frac{1}{2}}\\
&+O\bigg(\exp(D_0^{1/4}\log\log{D_0})(\log\log{v})^3\sum_{\substack{D_0<q\leq v \\ q\equiv 1\,\,(\text{mod}\,\,4)}}\frac{g_6(q)}{q}\bigg),
\end{align*}
where $C_W$ is defined as in (\ref{eq:CW}). Recalling the definition of $g_6(q),$ we see the error term contributes 
\begin{equation*}
\ll \exp(D_0^{1/4}\log\log{D_0})(\log\log{v})^3\log{v}.
\end{equation*}
Now note that 
\begin{equation*}
\frac{2\beta_1(q)-\beta_2(q)}{q}g_6(q) = -\frac{(3q-2)(2q+1)\log{q}}{2(q+1)(2q^2-2q+1)} = -\frac{3\log{q}}{2q}+O\bigg(\frac{\log{q}}{q^2}\bigg)
\end{equation*}
This error contributes 
\begin{equation*}
\ll C_W \sum_{\substack{D_0<q\leq v \\ q\equiv 1\,\,(\text{mod}\,\,4)}}\frac{\log{q}}{q^2}(\log{v})^{\frac{1}{2}} \ll C_W (\log{v})^{\frac{1}{2}}
\end{equation*}
which is small. We are left with a main term 
\begin{equation*}
-\frac{3C_W}{2}  \sum_{\substack{D_0<q\leq v \\ q\equiv 1\,\,(\text{mod}\,\,4)}} \frac{\log{q}}{q} \bigg(\log{\frac{v}{q}}\bigg)^{\frac{1}{2}}.
\end{equation*}
By partial summation one can show 
\begin{equation*}
\sum_{\substack{D_0<q\leq v \\ q\equiv 1\,\,(\text{mod}\,\,4)}} \frac{\log{q}}{q} \bigg(\log{\frac{v}{q}}\bigg)^{\frac{1}{2}} = \frac{1}{3}(\log{v})^{\frac{3}{2}}+O((\log{v})^{\frac{1}{2}}\log{D_0}),
\end{equation*}
so that 
\begin{equation*}
Z_{N,W}^{(2)} = -\frac{C_W}{2}(\log{v})^{\frac{3}{2}}+O(\exp(D_0^{1/4}\log\log{D_0})(\log\log{v})^3\log{v}).
\end{equation*}
This simplifies to the stated result.
\end{enumerate}
\end{proof}


\end{document}